\documentclass[10pt]{amsart}
\usepackage{amsmath,amssymb,amscd,amsthm,amsfonts,amstext,amsbsy,dsfont,mathrsfs,hyperref,upgreek,mathtools,stmaryrd,enumitem,bbm, MnSymbol}

\usepackage[textsize=footnotesize,color=blue!40, bordercolor=white]{todonotes}

\definecolor{darkblue}{rgb}{0,0,0.68}
\hypersetup{colorlinks=true,
citecolor=black,
linkcolor=black,
anchorcolor=black
}

\numberwithin{equation}{section}
\usepackage[T1]{fontenc}

\usepackage%[scale=0.682]
[a4paper, margin=1.2in]{geometry}

\newcommand{\BB}{\mathbb{B}}

\newcommand{\RR}{\mathbb{R}}
\newcommand{\UU}{\mathbb{U}}

\newcommand{\VVV}{\mathbb{V}}
\newcommand{\SSS}{\mathbb{S}}

\newcommand{\DD}{\mathbb{D}}

\newcommand{\PP}{\mathbb{P}}
\newcommand{\QQ}{\mathbb{Q}}
\newcommand{\TT}{\mathbb{T}}

\newcommand{\len}{\mathrm{l}} 
\newcommand{\one}{\mathbbm1}
\newcommand{\he}{\mathrm{ht}} 
\newcommand{\pro}{\mathrm{proj}}

\newcommand{\Cl}{\mathrm{Club}} 
\newcommand{\Ns}{\mathrm{NS}} 
 
\newcommand{\RA}{\mathrm{RA}}

\newcommand{\splitt}{\mathrm{split}} 
\newcommand{\succsplit}{\mathrm{succsplit}} 

\newcommand{\cof}{\operatorname{cof}}

\newcommand{\proj}{\operatorname{p}}
\newcommand{\Ord}{{\mathrm{Ord}}}
\newcommand{\Lim}{{\mathrm{Lim}}}
\newcommand{\ran}[1]{{{\rm{ran}}(#1)}}
\newcommand{\dom}[1]{{{\rm{dom}}(#1)}}

\newcommand{\Add}[2]{{\rm{Add}}({#1},{#2})}
\newcommand{\Addd}[2]{{\rm{Add}^*}({#1},{#2})}

\newcommand{\Col}[2]{{\rm{Col}}({#1},{#2})}

\DeclarePairedDelimiterX{\goe}[1]{\prec}{\succ}{#1}

\newcommand{\ka}{\kappa}
\newcommand{\lk}{{<}\ka}
\newcommand{\kk}{{}^{\ka}\ka}
\newcommand{\klk}{{}^{<\ka}\ka}
\newcommand{\Ordk}{{}^{\ka}\Ord}

\newcommand{\ltl}{{}^{\lambda}\lambda}
\newcommand{\ltll}{{}^{<\lambda}\lambda}
\newcommand{\Ordl}{{}^{\lambda}\Ord}

\newtheorem{theorem}{Theorem}[section]
\newtheorem{lemma}[theorem]{Lemma}
\newtheorem{corollary}[theorem]{Corollary}

\newtheorem*{question*}{Question}
\newtheorem{question}[theorem]{Question}
\newtheorem{problem}{Problem}
\newtheorem*{claim*}{Claim}

\newtheorem*{subclaim*}{Subclaim}

\theoremstyle{definition}

\newtheorem{definition}[theorem]{Definition}

\newtheorem{example}[theorem]{Example}
\theoremstyle{remark}
\newtheorem{remark}[theorem]{Remark}

\newenvironment{enumerate-(a)}{\begin{enumerate}[label={\upshape (\alph*)}, leftmargin=2pc]}{\end{enumerate}}

\newenvironment{enumerate-(a)-r}{\begin{enumerate}[label={\upshape (\alph*)}, leftmargin=2pc,resume]}{\end{enumerate}}

\newenvironment{enumerate-(A)}{\begin{enumerate}[label={\upshape (\Alph*)}, leftmargin=2pc]}{\end{enumerate}}

\newenvironment{enumerate-(A)-r}{\begin{enumerate}[label={\upshape (\Alph*)}, leftmargin=2pc,resume]}{\end{enumerate}}

\newenvironment{enumerate-(i)}{\begin{enumerate}[label={\upshape (\roman*)}, leftmargin=2pc]}{\end{enumerate}}

\newenvironment{enumerate-(i)-r}{\begin{enumerate}[label={\upshape (\roman*)}, leftmargin=2pc,resume]}{\end{enumerate}}

\newenvironment{enumerate-(I)}{\begin{enumerate}[label={\upshape (\Roman*)}, leftmargin=2pc]}{\end{enumerate}}

\newenvironment{enumerate-(I)-r}{\begin{enumerate}[label={\upshape (\Roman*)}, leftmargin=2pc,resume]}{\end{enumerate}}

\newenvironment{enumerate-(1)}{\begin{enumerate}[label={\upshape (\arabic*)}, leftmargin=2pc]}{\end{enumerate}}

\newenvironment{enumerate-(1)-r}{\begin{enumerate}[label={\upshape (\arabic*)}, leftmargin=2pc,resume]}{\end{enumerate}}

\begin{document} 

\title[Perfect subsets of generalized Baire spaces]{Perfect subsets of generalized Baire spaces and long 
%Banach-Mazur 
games} 
%The perfect subset property for generalized Baire spaces]{The perfect subset property and determinacy of the Banach-Mazur game for definable subsets of generalized Baire spaces} 
%\date{\today} 

\author{Philipp Schlicht}  

\thanks{The author was partially supported by DFG-grant LU2020/1-1 during the preparation of this paper. }

\maketitle

\begin{abstract} 
We extend Solovay's theorem about definable subsets of the Baire space to the generalized Baire space $\ltl$, where $\lambda$ is an uncountable cardinal with $\lambda^{<\lambda}=\lambda$. In the first main theorem, we show that  that the perfect set property for all subsets 
of ${}^{\lambda}\lambda$ that are definable from elements of $\Ordl$ is consistent relative to the existence of an inaccessible cardinal above $\lambda$. 
In the second main theorem, we introduce a Banach-Mazur type game of length $\lambda$ and show that the determinacy of this game, for all subsets of $\ltl$ that are definable from elements of $\Ordl$ as winning conditions, is consistent relative to the existence of an inaccessible cardinal above $\lambda$. We further obtain some related results about definable functions on $\ltl$ and consequences of resurrection axioms for definable subsets of $\ltl$. 
\end{abstract} 
%is it ok to write: consistent relative to an inaccessible cardinal above $\kappa$?

\setcounter{tocdepth}{2}
\tableofcontents

%%%%%%%%%%%%%%%%%%%%%%%%%%%%%%%%%%%%%%%%%%%%%%%%%%%%%%%%%%%%%%%%
%%%%%%%%%%%%%%%%%%%%%%%%%%%%%%%%%%%%%%%%%%%%%%%%%%%%%%%%%%%%%%%%
\section{Introduction}

%\todo[inline]{REPLACE $\kappa$ with $\lambda$ everywhere in the introduction? consistent notation everywhere else?} 

%This paper is motivated by the fact that the perfect set property and the Baire property are central in descriptive set theory. Therefore, one aims to know if analogues to these properties hold for definable subsets of ${}^{\kappa}\kappa$. 

%A \emph{perfect subset} of the Baire space ${}^{\omega}\omega$ is a nonempty, closed subset without isolated points. 
The \emph{perfect set property} for a subset of the Baire space 
%${}^{\omega}\omega$ 
states that it either contains a perfect subset, i.e. a nonempty, closed subset without isolated points, or is countable. By a classical result, this property is provable for the analytic subsets of the Baire space 
%${}^{\omega}\omega$ 
\cite[Corollary 14.8]{MR2731169}, but not for their complements \cite[Theorem 13.12]{MR2731169}. 
%Among the results about the perfect set property, we recall the following two central results: 
%Two main results known about the perfect set property are the following:  
%(1) All analytic subsets of ${}^{\omega}\omega$ have the perfect set property, where an analytic set is a continuous images of a closed subsets of ${}^{\omega}\omega$, and 
%The perfect set property \todo{cite} for $\bf{\Sigma}^1_1$ subsets of ${}^{\omega}\omega$ is provable in $\ZFC$, while the perfect set property for $\bf{\Pi}^1_1$ sets requires an inaccessible cardinal. 
Moreover, by an important result of Solovay, it is consistent relative to the existence of an inaccessible cardinal, that all subsets of ${}^{\omega}\omega$ that are definable from
countable sequences of ordinals 
 %elements of ${}^{\omega}\Ord$ 
 have the perfect set property \cite[Theorem 2]{MR0265151}.

%The perfect set property for all subsets of ${}^{\omega}\omega$ definable from elements of $\Ord^{\omega}$ by collapsing an inaccessible cardinal to $\omega_1$. 
%\todo[inline]{maybe define the perfect set property here?} 

%Various results have been generalized from the Baire space ${}^{\omega}\omega$ to generalized Baire spaces ${}^{\kappa}\kappa$, where $\kappa$ is an uncountable regular cardinal with $\kappa^{<\kappa}=\kappa$. 
%Motivated by these results, 
%Solovay's result, 
%we consider 

It is natural to ask whether the last result extends to uncountable cardinals $\lambda$. 
%These results motivate the question whether a natural generalization of this property, that is given in Definition \ref{definition: perfect sets} below, is consistent. 
In the uncountable setting, a \emph{perfect subset of $\ltl$} is defined as the set of all cofinal branches of some ${<}\lambda$-closed subtree of the set $\ltll$ of all sequences in $\lambda$ of length strictly less than $\lambda$. 
Accordingly, a subset of $\ltl$ has the \emph{perfect set property} if it either contains a perfect subset or has size at most $\lambda$. 
%\cite{MR0265151}, 
The next question (and variants thereof) was asked by Mekler and V\"a\"an\"anen \cite{vaan},
%\footnote{The paper erroneously claims to solve this question. } 
Kovachev \cite{Kovachev-thesis}, Friedman and others. 

\begin{question} 
Is it consistent, relative to the existence of large cardinals, that for some uncountable cardinal $\lambda$, the perfect set property holds for all 
%$\Pi^1_1$ 
subsets of $\ltl$ that are definable from $\lambda$? 
% that are definable over $H_{\kappa^+}$? 
%every subset $A$ of ${}^{\kappa}\kappa$ with $|A|\geq\kappa^+$ has a perfect subset? 
\end{question} 

The first main result, which we prove in Theorem \ref{perfect subsets of definable sets} below, gives a positive answer to this question. 

\begin{theorem} \label{main theorem on perfect sets in introduction} 
For any uncountable regular cardinal $\lambda$ with an inaccessible cardinal above it, there is a generic extension by a ${<}\lambda$-closed forcing in which 
%with $\kappa^{<\kappa}=\kappa$, 
%it is consistent relative to an inaccessible cardinal above $\kappa$ 
every subset of $\ltl$ that is definable from a $\lambda$-sequence of ordinals 
%\todo{write "$\kappa$-sequence of ordinals", everywhere?} element of $\Ordk$ 
has the perfect set property. 
\end{theorem} 

%This is proved in Theorem \ref{perfect subsets of definable sets} below. 
Assuming that there is a proper class of inaccessible cardinals, this can be extended to the next result, which is proved in Theorem \ref{perfect set property at all regulars} below. 

\begin{theorem} \label{perfect set property at all regulars} 
Assume that there is a proper class of inaccessible cardinals. Then there is a class generic extension 
%of $V$ 
in which 
%such that in $V[G]$, 
for every infinite regular cardinal $\lambda$, every subset of $\ltl$ that is definable from a $\lambda$-sequence of ordinals 
%n element of ${}^{\kappa}\Ord$ 
has the perfect set property. 
%$\mathsf{PSP}^{\kappa}_{od}$ holds for all infinite regular cardinals $\kappa$ in a class forcing extension. 
\end{theorem} 
%\begin{theorem} 
%\todo{rephrase for a class generic extension?} 
%It is consistent relative to a proper class of inaccessible cardinals that for all regular cardinals $\kappa$ with $\kappa^{<\kappa}=\kappa$, any subset of ${}^{\kappa}\kappa$ that is definable from an element of $\Ordk$ has the perfect set property. 
%\end{theorem} 

We will further 
%use the proof of the first main theorem for 
obtain the next result about definable functions in Theorem \ref{function on 2^kappa continuous on a perfect set}. 
%in analogy to \todo{\ \ \ \ cite WOODIN cons strength of proj unif?} a result for the Baire space. 
In the statement, let $[X]^\gamma_{\neq}$ denote the set of sequences $\langle x_i\mid i<\gamma\rangle$ of distinct elements of $X$ for any set $X$ and any ordinal $\gamma$. 
% ${}^{\omega}\omega$. 

\begin{theorem} 
%relative to an inaccessible cardinal 
%\todo{rewrite, check statement below} 
For any uncountable regular cardinal $\lambda$ with $\lambda^{<\lambda}=\lambda$, there is a generic extension by a ${<}\lambda$-closed forcing in which for every $\gamma<\lambda$ and every function
$f\colon [{}^{\lambda}\lambda]_\neq^\gamma \mapsto {}^{\lambda}\lambda$ that is definable from a $\lambda$-sequence of ordinals, 
%elements of $\Ordk$, 
there is a perfect subset $C$ of ${}^{\lambda}\lambda$ such that $f{\upharpoonright }[C]^\gamma$ is continuous. 
\end{theorem} 

%We will give precise definitions for the statement of this question below. 
%A \emph{perfect subtree} $T$ of ${}^{<\kappa}\kappa$ is a tree such that above every node, there is a splitting node, and for ever strictly increasing  sequence $\langle s_\alpha\mid \alpha<\gamma\rangle$ with $\gamma<\kappa$, $s=\bigcup_{\alpha<\gamma} s_\alpha \in T$. 
%A \emph{perfect subset} $C$ of ${}^{\kappa}\kappa$ is a set $C=[T]$, where $T$ is a perfect subtree of ${}^{<\kappa}\kappa$. 
%Question asked by Mekler-V\"a\"an\"anen. Also asked by Friedman in some slides from Luminy? Technique used by Laguzzi (and others?). 
%The question of the consistency of the perfect set property for definable subsets of ${}^{\kappa}\kappa$ can be found in [Mekler-V\"a\"an\"anen]. We show that it holds in $Col(\kappa,<\lambda)$-generic extensions for inaccessible $\lambda>\kappa$ by avoiding to work with subsets of $\kappa$ whose quotient forcing is not equivalent to $Col(\kappa,<\lambda)$. 
%V\"a\"an\"anen asked if it is consistent that the perfect set property holds for more definable subsets of ${}^{\kappa}\kappa$. 
%The first main result is that for regular uncountable cardinals $\kappa$, $\mathsf{PSP}^{\kappa}_{od}$ holds after an inaccessible $\lambda>\kappa$ is Levy-collapsed to $\kappa^+$. 
%We bypass the problem of working with subsets of $\kappa$ with bad quotients by constructing (in the next section) a perfect set of $Add(\kappa,1)$-generics inside an $Add(\kappa,1)$-generic extension such that each quotient forcing is equivalent to $Add(\kappa,1)$. 

We now turn to the Baire property and generalizations thereof, which we study in the second part of this paper. 
%Here, i
It is provable that analytic and co-analytic subsets of ${}^{\omega}\omega$ have the Baire property \cite[Theorem 21.6]{MR1321597}. 
Moreover, Solovay proved that it is consistent, relative to the existence of an inaccessible cardinal, that all subsets of ${}^{\omega}\omega$ that are definable from elements of ${}^{\omega}\Ord$ have the Baire property \cite[Theorem 2]{MR0265151}.

The direct generalization of the Baire property, which we here call \emph{$\lambda$-Baire}, is given in Definition \ref{definition: kappa-Borel etc} below. 
However, the situation for this property in the uncountable setting is very different compared to both the Baire property in the countable setting and the perfect set property in the uncountable setting, 
%from that of the perfect set property. 
%The notion of $\kappa$-Baire subsets of ${}^{\kappa}\kappa$ is given in Definition \ref{definition: kappa-Borel etc}. 
%A subset of ${}^{\kappa}\kappa$ is \emph{$\kappa$-Baire} if it is an element of the $\kappa$-albegra generated 
%However, the analogous statement to Solovay's theorem fails for this notion, 
since there are always $\Sigma^1_1$ subsets of ${}^{\lambda}2$ that are not $\lambda$-Baire by the next example. 
%as is well known. 
To state the example, 
we consider the set 
%let $\Cl_\kappa$ denote the set 
$$\Cl_\lambda=\{x\in {}^{\lambda}2\mid \exists C\subseteq\lambda \text{ club } \forall i\in C\ x(i)\neq 0\}$$ 
%$$\Cl_\lambda=\{x\in {}^{\lambda}\lambda\mid \exists C\subseteq\lambda \text{ club } \forall i\in C\ x(i)\neq 0\}$$ 
of functions coding elements of the club filter as characteristic functions. 
%denote the club filter on $\kappa$. 
%$\Cl_\kappa$ denote the set of $x\in {}^{\kappa}2$ such that there is a club $C$ in $\kappa$ with $x(i)\neq 0$ for all $i\in C$. 

\begin{example} \cite[Theorem 4.2]{MR1880900} 
Suppose that $\lambda$ is a cardinal with $\cof(\lambda)>\omega$. 
Then the set $\Cl_\lambda$ is not a $\lambda$-Baire subset of ${}^{\lambda}2$. 
%and the nonstationary ideal $\Ns_\kappa$ on $\kappa$ are not \todo{define kappa-Baire} $\kappa$-Baire. 
\end{example} 

Moreover, this counterexample is generalized to subsets of ${}^{\lambda}\lambda$ in \cite{MR3235820} as follows. 
If $S$ is a subset of $\lambda$, 
we consider the set 
%let $\Cl_\kappa$ denote the set 
$$\Cl^S_\lambda=\{x\in {}^{\lambda}\lambda\mid \exists C\subseteq\lambda \text{ club } \forall i\in C\ x(i)\in S\}.$$ 
%$\Cl^S_\kappa$ denote the set of $x\in {}^{\kappa}\kappa$ such that there is a club $C$ in $\kappa$ with $x(i)\in S$ for all $i\in C$. 

\begin{example} \cite[Theorem 3.10]{MR3235820} 
Suppose that $\lambda$ is an uncountable cardinal with $\lambda^{<\lambda}=\lambda$ and $S$ is a bi-stationary subset of $\lambda$. 
Then the set $\Cl^S_\lambda$ is not a $\lambda$-Baire subset of ${}^{\lambda}\lambda$. 
%and the nonstationary ideal $\Ns_\kappa$ on $\kappa$ are not \todo{define kappa-Baire} $\kappa$-Baire. 
\end{example} 
%\todo{earlier!!!}It is well known that there are $\Sigma^1_1$ subsets of ${}^{\kappa}\kappa$ without the Baire property. 
%\begin{lemma}[Halko-Shelah] \todo{cite, also write the right example for ${}^{\kappa}\kappa$, from Yurii et al?} 
%The club filter 
%$$\mathrm{C}_{\kappa}=\{x\subseteq \kappa\mid \text{ there is a club }C\text{ in }\kappa \text{ with }C\subseteq x\}$$ 
%and \todo{maybe move this to the introduction} the non-stationary ideal $P(\kappa)\setminus \mathrm{C}_\kappa$ 
%on $\kappa$ do not have the property of Baire. 
%\end{lemma} 
%This result has been extended in \todo{cite Yuri et al} to other regularity properties derived from tree forcings such as \todo{to do}. 

It is worthwhile to mention that there are further 
%\todo{single line looks strange} Further 
strengthenings of this failure that can be found in 
%are proved in 
\cite[Proposition 3.7]{MR3430247}. 

Since the Baire property for subsets of ${}^{\omega}\omega$ is characterized by the Banach-Mazur game \cite[Theorem 8.33]{MR1321597}, it is useful to consider a generalization of this game of uncountable length (see Definition \ref{definition:  Banach-Mazur game} below). 
However, because of the asymmetry of the game at limit times, the condition that a given subset $A$ of $\ltl$ is $\lambda$-Baire is stronger than the determinacy of the Banach-Mazur game of length $\lambda$ for the set $A$ as a winning condition. 
%Thus the following question is natural, in analogy with Theorem \ref{main theorem on perfect sets in introduction}, and 
This motivates the following question, which was asked in \cite{Kovachev-thesis}. 

\begin{question} 
Is it consistent, relative to the existence of large cardinals, that for some uncountable cardinal $\lambda$, the Banach-Mazur game of length $\lambda$ is determined for all 
%$\Pi^1_1$ 
subsets of $\ltl$ that are definable from $\lambda$ as winning conditions? 
%Is it consistent that for some uncountable cardinal $\kappa$, the Banach-Mazur game of length $\kappa$ is determined for all  subsets of ${}^{\kappa}\kappa$ definable from $\kappa$? 
% that are definable over $H_{\kappa^+}$? 
%every subset $A$ of ${}^{\kappa}\kappa$ with $|A|\geq\kappa^+$ has a perfect subset? 
\end{question} 

%Therefore, the Baire property can be generalized in a different way to subsets of ${}^{\kappa}\kappa$. 
The second main result, which we prove in Theorem \ref{determinacy result} below, gives a positive answer to this question. 

\begin{theorem} 
%\todo{rewrite? ... "as a winning condition"? } 
For any uncountable regular cardinal $\lambda$ with an inaccessible cardinal above it, there is a generic extension by a ${<}\lambda$-closed forcing in which 
%with $\kappa^{<\kappa}=\kappa$, 
%it is consistent relative to an inaccessible cardinal above $\kappa$ 
the Banach-Mazur game of length $\lambda$ is determined for any subset of $\ltl$ that is definable from a $\lambda$-sequence of ordinals.  
%an element of $\Ordk$. 

%It is consistent relative to an inaccessible cardinal that the Banach-Mazur game of length $\kappa$ is determined for all subsets of ${}^{\kappa}\kappa$ that are definable from elements of $\Ord^{\kappa}$. 
\end{theorem} 

%This is proved in Theorem \ref{determinacy result} below. 
We will moreover use the Banach-Mazur game to define a generalization of the Baire property, which we call \emph{almost Baire}, in Section \ref{subsection: Banach-Mazur games}, and show that it is consistent that this property holds for the same class of definable sets that is considered above. 
%We prove related results about definable functions on ${}^{\kappa}\kappa$ in Section \ref{subsection: consistency of the almost Baire property}. 

%We \todo{rewrite, define the game} consider \todo{Define the game here?} the perfect set game for a set $A\subseteq {}^{\kappa}\kappa$. This is known to characterize the perfect set property in the sense that $|A|\leq\kappa$ if an only if the second player has a winning strategy (to play outside of $A$) and $A$ has a perfect subset if and only if the first player has a winning strategy (to play inside $A$). 
%Hence this game is determined for sets definable from ordinals and subsets of $\kappa$ after an inaccessible $\lambda>\kappa$ is Levy-collapsed to $\kappa^+$.  

%In the Banach-Mazur game of length $\kappa$ for a set $A\subseteq{}^{\kappa}\kappa$, two players play a strictly increasing sequence $(s_{\alpha})_{\alpha<\kappa}$ of elements of ${}^{<\kappa}\kappa$. The second player wins if $\bigcup_{\alpha<\kappa}s_{\alpha}\in A$. 

We now turn to the question whether the conclusions of the above results follow from strong axioms. 
In the countable setting, 
%Solovay's theorem \cite{MR0265151} has the following consequence. 
it is well known that $M_n^{\#}$ is absolute to all set generic extensions for all natural numbers $n$ 
and that therefore, the theory of $(H_{\omega_1},\in)$ is absolute to all generic extensions if there is a proper class of Woodin cardinals (see \cite{MR2768698,MR3226056}). 
%If there is a proper class of Woodin cardinals, 
%then the theory of $H_{\omega_1}$ is \todo{cite} absolute to all set generic extensions. 
Hence the conclusion of Solovay's theorem \cite[Theorem 1]{MR0265151} for projective sets is provable from a proper class of Woodin cardinals.\footnote{Infinitely many Woodin cardinals are sufficient by \cite{MR955605}.} 

In the uncountable setting, the theory of $(H_{\omega_2},\in)$ is not absolute to all generic extensions that preserve $\omega_1$, since both the existence and non-existence of $\omega_1$-Kurepa trees can be forced by ${<}\omega_1$-closed forcings, assuming the existence of an inaccessible cardinal. 
%for instance it is possible to add a $\omega_1$-Kurepa trees by ${<}\omega_1$-closed forcing, while the non-existence of $\omega_1$-Kurepa trees can be forced by collapsing an inaccessible cardinal to $\omega_2$ by a ${<}\omega_1$-closed forcing. 
% (this was proved by Jack Silver). 
Therefore, we will consider a variant ot the \emph{resurrection axiom} that was introduced by Hamkins and Johnstone \cite{MR3194674}. 
%\emph{iterated resurrection axioms} that were introduced by Viale and Audrito \cite{Audrito-Viale}. 
% and were based on work of 
The idea for such axioms is to postulate that certain properties of the ground model which might be lost in a generic extension 
can be resurrected 
%for certain classes of forcings 
by passing to a further 
%generic 
extension. 

We will see that variants of the conclusions of the above results 
%for subsets of ${}^{\kappa}\kappa$ that are definable from elements of $H_{\kappa^+}$, 
follow from such an axioms for a class of ${<}\lambda$-closed forcings. 
If $\lambda$ is a regular cardinal, we say that $\nu$ is \emph{$\lambda$-inaccessible} if $\nu>\lambda$ is regular and $\mu^{<\lambda}<\nu$ holds for all cardinals $\mu<\nu$. 
The following result is proved in Theorem \ref{implications of resurrection} below. 

\begin{theorem} 
Suppose that $\lambda$ is an uncountable regular cardinal, 
%$\Gamma$ is the class of forcings $\Col{\kappa}{<\nu}$ for regular cardinal $\nu>\kappa$ 
and the resurrection axiom $\RA^{\lambda}$ (see Definition \ref{definition of resurrection axioms} below) holds for the class of forcings $\Col{\lambda}{{<}\nu}$, where $\nu$ is $\lambda$-inaccessible. Then the following statements hold for every subset $A$ of $\ltl$ that is definable over $(H_{\lambda^+},\in)$ with parameters in $H_{\lambda^+}$. 
% has the following properties. 
\begin{enumerate-(1)} 
\item \label{perfect set property from resurrection} 
$A$ has the perfect set property. 
% for every $\PP$-generic filter $G$ over $V$. 
\item \label{almost Baire property from resurrection} 
The Banach-Mazur game of length $\lambda$ 
%$G_{\kappa}(A)$ 
with $A$ as a winning condition is determined. 
%of ${}^{\kappa}\kappa$ definable over $(H_{\kappa^+},\in)$ with parameters in $H_{\kappa^+}$ as winning conditions. 
\end{enumerate-(1)} 
\end{theorem} 
%Hence these models are analogous to Solovay's model. 
%problems studied in 

\iffalse 
\begin{theorem} \label{implications of resurrection} 
Suppose that $\Gamma$ is the class of forcings $\Col{\lambda}{{<}\nu}$, where $\nu$ is $\lambda$-inaccessible. 
Assuming that $\RA^{\lambda}(\Gamma)$ holds, 
the following statements hold for every subset $A$ of ${}^\lambda\lambda$ that is definable over $(H_{\lambda^+},\in)$ with parameters in $H_{\lambda^+}$. 
\begin{enumerate-(1)} 
\item \label{perfect set property from resurrection} 
$A$ has the perfect set property. 
\item \label{almost Baire property from resurrection} 
The game 
$G_{\lambda}(A)$ is determined. 
\end{enumerate-(1)} 
\end{theorem} 
\fi

This paper is organized as follows. 
In the remainder of this section, we will collect several definitions and facts about trees and forcings. 
In Section \ref{section: The perfect set property}, we will prove among other results the consistency of the perfect set property for definable subsets of $\ltl$. 
% and a related result about definable functions on ${}^{\kappa}\kappa$. 
In Section \ref{section: The almost Baire property}, we will prove among other results the consistency of the almost Baire property for definable subsets of $\ltl$. %and related results about definable functions on ${}^{\kappa}\kappa$. 
Finally, in Section \ref{resurrection axioms}, we will derive variants of the conclusions of the main results from resurrection axioms. 

For notation, we will assume throughout this paper that $\kappa$ is an uncountable regular cardinal with $\kappa^{<\kappa}=\kappa$ and $\lambda$ is an uncountable regular cardinal. 

We would like to thank Peter Holy for discussions about the presentation and the referee for various helpful comments. 
%Moreover, 
The results in this paper are motivated by work of Solovay \cite{MR0265151}, Mekler and V\"a\"an\"anen \cite{vaan}, Donder and Kovachev \cite{Kovachev-thesis} 
%Friedman, Hyttinen and Kulikov \cite{MR3235820}, 
and 
%moreover, 
some ideas from this work have already been applied in subsequent work 
%L\"ucke, Motto Ros and Schlicht 
\cite{MR3319716, Hurewicz}. 
\subsection{Trees and perfect sets} 

%\todo{consistent everywhere? write: throughout the paper?} In this section, we 
We always assume that $\lambda$ is a regular uncountable cardinal. 
%We always assume that $\kappa$ is an uncountable regular cardinal with $\kappa^{<\kappa}=\kappa$. 
The \emph{standard topology} (or \emph{bounded topology}) on ${}^{\lambda}\lambda$ is generated by the basic open sets 
$$N_t=\{x\in {}^{\lambda}\lambda\mid t\subseteq x\}$$ 
for $t\in {}^{<\lambda}\lambda$. 
The \emph{generalized Baire space for $\lambda$} is the set ${}^{\lambda}\lambda$ of functions $f\colon \lambda\rightarrow\lambda$ with the standard topology.

Since we will work with definable subsets of ${}^{\lambda}\lambda$, we will use the following notation. 

\begin{definition} \label{definition of Aphi} 
If $\varphi(x,y)$ is a formula with the two free variables $x$, $y$ and $z$ is a set, let 
$$A_{\varphi,z}^\lambda=\{x\in {}^{\lambda}\lambda\mid \varphi(x,z)\}.$$ 
If $\varphi(x)$ is a formula with the free variable $x$, let 
$$A_{\varphi}^\lambda=\{x\in {}^{\lambda}\lambda\mid \varphi(x)\}.$$ 
\end{definition}

The following definition generalizes perfect trees and perfect sets to the uncountable setting. 
% or \emph{standard topology} 
%with the basic open sets $N_s=\{f\in {}^{\kappa}\kappa\mid s\subseteq f\}$ for $s\in {}^{<\kappa}\kappa$. 
%From the viewpoint of descriptive set theory, this space is analogous to the Baire space ${}^{\omega}\omega$ of functions $f\colon\omega\rightarrow\omega$ with the product topology, and in fact it has many similar properties such as a generalization of the Baire category theorem. 
%The descriptive set theory of this space (for uncountable $\kappa$) was first studied by Mekler and V\"a\"an\"anen [Mekler-V\"a\"an\"anen, V\"a\"an\"anen] and later for example in [Hyttinen-Kulikov-Friedman, L\"ucke]. 
%We consider the following generalizations of perfect subtrees of ${}^{<\omega}\omega$, perfect subsets of ${}^{\omega}\omega$, and the perfect set property for subsets of ${}^{\omega}\omega$. 

\begin{definition} \label{definition: perfect sets} 
Suppose that $T$ is a subtree of ${}^{<\lambda}\lambda$, that is a downwards closed subset of ${}^{<\lambda}\lambda$. 
\begin{enumerate-(a)} 
\item 
$\mathrm{pred}_T(t)=\{s\in T\mid s\subsetneq t\}$ and 
%\todo{check with definition below} 
$\len(t)=\dom{t}$. 
%=\mathrm{type}(\mathrm{pred}_T(t))$ for $t\in T$, where $\mathrm{type}$ denotes the order type. 
\item 
A node $s$ in $T$ is \emph{terminal} if it has no direct successors in $T$ and \emph{splitting} if it has at least two direct successors in $T$. 
\item 
A \emph{branch} in $T$ is a sequence $b\in {}^{\lambda}\lambda$ with $b{\upharpoonright}\alpha \in T$ for all $\alpha<\lambda$. 
\item 
The \emph{body} of $T$ is the set $[T]$ of branches in $T$. 
% $[T]=\{f\in {}^{\kappa}\kappa\mid \forall\alpha<\kappa (x\upharpoonright \alpha\in T)\}$
\item 
%\todo{deleted 'limit-closed'. report: is this the same as limit-closed?} $T$ 
%\subseteq {}^{<\kappa}\kappa$ 
$T$ is \emph{closed} (\emph{${<}\lambda$-closed}) if every strictly increasing sequence in $T$ of length ${<}\lambda$ has an upper bound in $T$. 
%\item 
%$T$ 
%\subseteq {}^{<\kappa}\kappa$ 
%\todo{need to define this for subsets. and it only makes sense for sets, not trees. actually Def 3.15 might be sufficient, maybe delete this here} is \emph{limit-closed} if the supremum (i.e. the union) of every strictly increasings sequence in $T$ of length $<\kappa$ is in $T$. 
\item
$T$ 
%\subseteq {}^{<\kappa}\kappa$ 
%of height $\kappa$ 
is \emph{perfect} (\emph{$\lambda$-perfect}) if $T$ is closed and the set of splitting nodes in $T$ is cofinal in the tree order of $T$, that is, above every node there is some splitting node. 
\item 
A subset $A$ of ${}^{\lambda}\lambda$ is \emph{perfect} (\emph{$\lambda$-perfect}, \emph{superclosed}) if $A=[T]=\{x\in {}^{\lambda}\lambda\mid \forall\alpha<\lambda (x{\upharpoonright} \alpha\in T)\}$ for some perfect tree $T$. 
\item 
A subset $A$ of ${}^{\lambda}\lambda$ has the \emph{perfect set property} 
%($\PSP$) 
if $|A|\leq \lambda$ or $A$ has a perfect subset. 
%\item 
%For $\Gamma\subseteq P({}^{\lambda}\lambda)$, $\PSP(\Gamma)$ denotes the perfect set property for all $A\in \Gamma$. 
%\item 
%For a class $X$, 
%\todo{report: added $\lambda$ to notation. check everywhere below} $\PSP^{\lambda}_X$ denotes the perfect set property for all subsets of ${}^{\lambda}\lambda$ definable from parameters in $X$. 
\end{enumerate-(a)} 
\end{definition} 

V\"a\"an\"anen \cite[Section 2]{MR1110032} introduced a different notion of $\lambda$-perfect sets based on a game of length $\lambda$. 
We will see in Section \ref{section: The perfect set property} that the perfect set property associated to this notion is equivalent to our definition. 
Moreover, Kanamori \cite{MR593029} introduced a variant of Sacks forcing for $\lambda$, leading to a corresponding stronger notion of perfect sets (see also \cite{MR3459046}), but our results do not hold for this notion. 

In the following definition, a \emph{$\lambda$-algebra} of subsets of ${}^{\lambda}\lambda$ is a set of subsets of ${}^{\lambda}\lambda$ that is closed under complements, unions of length $\lambda$ and intersections of length $\lambda$. 

\begin{definition} \label{definition: kappa-Borel etc} 
Suppose that $A$, $B$ are subsets of ${}^{\lambda}\lambda$. 
\begin{enumerate-(a)} 
\item 
$A$ is \emph{$\lambda$-Borel} (\emph{Borel}) if it is an element of the smallest $\lambda$-algebra containing the open subsets of ${}^{\lambda}\lambda$. 
% by forming complements and unions and intersections of length $\kappa$. 
\item 
$A$ is \emph{$\lambda$-meager} (\emph{meager}) \emph{in $B$} if $A\cap B$ is the union of $\lambda$ many nowhere dense subsets of $B$, and \emph{$\lambda$-comeager} (\emph{comeager})  \emph{in $B$} if its complement is $\lambda$-meager in $B$. 
Moreover, we will omit $B$ if it is equal to ${}^\lambda\lambda$. 
%If $B={}^{\lambda}\lambda$, then $B$ is omitted. 
\item 
$A$ is \emph{$\lambda$-Baire} (\emph{Baire}) if there is an open subset $U$ of ${}^{\lambda}\lambda$ such that $A\triangle U$ is $\lambda$-meager. 
\end{enumerate-(a)} 
\end{definition} 
%We will write meager instead of $\kappa$-meager and Baire property instead of $\kappa$-Baire property. 
%\todo{where should this definition go?} We write $\len(t)$ for the length of $t$. 
%\todo[inline]{write out} 
%We will write \emph{perfect} instead of $\kappa$-perfect when working with a fixed cardinal $\kappa$. 
%Note that V\"a\"an\"anen [V\"a\"an\"anen: A Cantor-Bendixson theorem, Def. ?] introduced a different type of $\kappa$-perfect set based on a game of length $\kappa$, but \todo{write down as a lemma why they are equivalent} the perfect set property is equivalent for the two definitions. 
%this is equivalent for our purpose, since every such set contains a perfect set and every perfect set (in our sense) satisfies V\"a\"an\"anen's definition. 
%Solovay [Solovay] proved that $\mathsf{PSP}^{\omega}_{od}$ holds after forcing with $Col(\omega,<\lambda)$, where $\lambda$ is inaccessible. The proof uses the factoring of intermediate models, i.e. for every real in a $Col(\omega,<\lambda)$-generic extension for $\lambda>\kappa$ inaccessible, the quotient forcing for the real is forcing equivalent to $Col(\omega,<\lambda)$, i.e. the quotient forcing and $Col(\omega,<\lambda)$ have isomorphic dense subsets. 
%In most parts of the paper, we will fix an uncountable regular cardinal $\kappa$ with $\kappa^{<\kappa}=\kappa$ and use the following simplified notation. 

%\begin{definition} 
%\todo{move to later, when this is used?}For a set $S$, let $\pro(S)=\{x \mid \exists y\ (x,y)\in S\}$. 
%\end{definition} 

%%%%%%%%%%%%%%%%%%%%%%%%%%%%%%%%%%%%%%%%%%%%%%%%%%
\subsection{Forcings} 

%We now state some basic facts about forcing that we will use below. 
%We collect some definitions related to forcing that are used throughout this paper. 
%\todo{does a forcing have to have a 1?} 
A forcing $\PP=(P,\leq)$ consists of a set $P$ and a transitive reflexive relation (also called a pre-order) $\leq$ with domain $P$. We will also write $p\in \PP$ for conditions $p\in P$ 
by identifying $\PP$ with its domain. 
%no $\PP$-generic filter is in the ground model. 
If $\PP$ is a separative partial order, we will assume that $\BB(\PP)$ denotes a fixed Boolean completion such that $\PP$ is a dense subset of $\BB(\PP)$.

\begin{definition} 
%Suppose that $\PP$ is a forcing. 
\begin{enumerate-(a)} 
\item 
An \emph{atom} in a forcing $\PP$ is a condition $p\in\PP$ with no incompatible extensions. Moreover, a forcing $\PP$ is \emph{non-atomic} if it has no atoms. 
\item 
A forcing $\PP$ is \emph{homogeneous} if for all $p,q\in \PP$, there is an automorphism $\pi\colon \PP\rightarrow \PP$ such that $\pi(p)$ and $q$ are compatible. 
%\item 
%\todo{NEED THIS??} Suppose that $\PP$ is a complete Boolean algebra and $X\subseteq \PP$. 
%Then $\langle X\rangle^{\PP}$ denotes the complete Boolean subalgebra of $\PP$ generated by $X$. 
\end{enumerate-(a)} 
\end{definition} 

The sub-equivalence in the next definition is stronger than the standard notion of equivalence for separative partial orders, which 
%is defined by the existence of an isomorphism between 
states that the Boolean completions are isomorphic. This specific definition is used in several constructions in the proofs below. 

\begin{definition} \label{definition: equivalent and similar forcings} 
Suppose that $\PP$, $\QQ$ are forcings. 
\begin{enumerate-(a)} 
%\item 
%An \emph{isomorphism $\theta\colon \PP\rightarrow_{\mathrm{d}} \QQ$ between dense subsets} is an isomorphism $\theta\colon \PP_{\mathrm{d}}\rightarrow \QQ_{\mathrm{d}}$, where $\PP_{\mathrm{d}}$ is a dense subset of $\PP$ and $\QQ_{\mathrm{d}}$ is a dense subset of $\QQ$. 
%states that there are 
% such that $\theta\colon D_{\PP}\rightarrow D_{\QQ}$ is an isomorphism. 
\item 
A \emph{sub-isomorphism $\iota\colon \PP\rightarrow \QQ$} is an isomorphism between $\PP$ and a dense subset of $\QQ$. 
%$\theta\colon \PP_{\mathrm{d}}\rightarrow \QQ_{\mathrm{d}}$, where $\PP_{\mathrm{d}}$ is a dense subset of $\PP$ and $\QQ_{\mathrm{d}}$ is a dense subset of $\QQ$. 
\item 
$\PP$, $\QQ$ are \emph{sub-equivalent} ($\PP\eqsim \QQ$) if there are sub-isomorphisms $\iota\colon \RR\rightarrow \PP$, $\nu\colon \RR\rightarrow\QQ$ for some forcing $\RR$. 
%\todo{should isomorphisms preserve incompatibility?} an isomorphism $\theta\colon \PP\rightarrow_{\mathrm{d}} \QQ$ between dense subsets. 
\item 
$\PP$, $\QQ$ are \emph{equivalent} ($\PP\simeq \QQ$) if there are sub-isomorphisms $\iota\colon \PP\rightarrow \RR$, $\nu\colon \QQ\rightarrow\RR$ for some forcing $\RR$. 
\item 
$\PP$, $\QQ$ are \emph{isomorphic} ($\PP\cong \QQ$) if there is an isomorphism $\iota\colon \PP\rightarrow \QQ$. 
\item \label{definition: equivalent and similar forcings - last item} 
if $\iota\colon \PP\rightarrow \QQ$ is a sub-isomorphism, 
we define a $\PP$-name $\tau^\iota$ for each $\QQ$-name $\tau$ by induction on the rank as 
$$\tau^{\iota}=\{(\sigma^{\iota},p) \mid p\in \PP,\ \exists q\in\QQ\  (\sigma,q)\in \tau,\  \iota(p)\leq q\}.$$ 
%if there is \todo{should isomorphisms preserve incompatibility?} an isomorphism $\theta\colon \PP\rightarrow_{\mathrm{d}} \QQ$ between dense subsets. 
%there are dense subsets $D\subseteq \PP$, $E\subseteq\QQ$ such that $D,E$ are isomorphic. 
%\item 
%$\PP$, $\QQ$ are \emph{similar} if for every $\PP$-generic filter $G$ over $V$, there is a $\QQ$-generic filter $H$ over $V$ with $V[G]=V[H]$, and conversely. We write \todo{find the right symbol: double snake} $\PP\sim \QQ$. 
\end{enumerate-(a)} 
\end{definition}

%\todo[inline]{check everywhere: replace equivalent with sub-equivalent} 

%\todo{define this as a subforcing of ${}^{<\kappa}\kappa$? need more versions?} Let $\Addd{\kappa}{1}=\{p\in \Add{\kappa}{1}\mid \dom{p}\in\Succ\}$. 
%We will \todo{check} will frequently pull back names through complete embeddings as in the following definition. 
It is easy to check that in Definition \ref{definition: equivalent and similar forcings} \ref{definition: equivalent and similar forcings - last item}, 
for any $\PP$-generic filter $G$ and for the upwards closure $H$ of $\iota[G]$ in $\QQ$, $(\tau^\iota)^G=\tau^H$.

\begin{lemma} \label{equivalence of forcings is transitive} 
%\todo[inline]{CITE THIS BELOW AT LEAST ONCE } 
Suppose that $\PP$, $\QQ$, $\RR$ are forcings. 
\begin{enumerate-(1)} 
\item 
If $\iota\colon \PP\rightarrow\QQ$, $\nu\colon\PP\rightarrow\RR$ are sub-isomorphisms, then there is a partial order $\SSS$ and isomorphisms onto dense subforcings $\iota^*\colon\QQ\rightarrow \SSS$, $\nu^*\colon\RR\rightarrow\SSS$ with $\iota^*\iota=\nu^*\nu$. 
\item 
If $\PP\eqsim \QQ$, then $\PP\simeq\QQ$. 
%Any two sub-equivalent forcings are equivalent. 
\item 
The relation $\simeq$ is transitive. 
\end{enumerate-(1)} 
\end{lemma} 
\begin{proof} 
%\todo{read again}
For the first claim, let $\leq_\PP$, $\leq_\QQ$, $\leq_\RR$ be the given forcing preorders. 
We can assume that $\QQ$, $\RR$ are disjoint and let 
% and $\leq_{\QQ,\RR}$ the union of $\leq_\QQ$ and $\leq_\RR$. 
$S=\QQ\cup \RR$. 
%\todo{write with $\QQ_0$, $\QQ_1$ everywhere???} 
Moreover, we define the relation $\leq_S$ on $S$ by $u\leq_S v$ if 
$u\leq_\QQ v$, $u\leq_\RR v$ or for some $p\in \PP$, 
$$ (u\leq_\QQ \iota(p)\text{ and } \nu(p)\leq_\RR v)\text{ or } (u\leq_\RR \nu(p)\text{ and } \iota(p)\leq_\QQ v).$$ 
It is then easy to check that $\leq_S$ is transitive and reflexive, ${\leq_{S}}{\upharpoonright}\QQ=\leq_\QQ$ 
and $\iota$, $\nu$ are sub-isomorphisms into $S$. 
%$\iota[\PP]$, $\nu[\PP]$ are dense in $S$. 

We now let $u\equiv_S v$ if $u\leq_S v$ and $v\leq_S u$ and let $\SSS$ denote the poset that is obtained as a quotient of $S$ by $\equiv_S$ with the partial order induced by $\leq_S$. 
Let further $\iota^*\colon \QQ\rightarrow \SSS$, $\iota(q)=[q]$ and $\nu^*\colon \RR\rightarrow \SSS$, $\nu^*(r)=[r]$, where $[p]$ denotes the equivalence class of $p\in S$ with respect to $\equiv_S$. By the definitions, $\iota^*$, $\nu^*$ are sub-isomorphisms into $\SSS$ that commute in the required fashion. 

Moreover, this immediately implies the second claim. 

For the last claim, suppose that $\PP\simeq\QQ$ and $\QQ\simeq\RR$ are witnessed by sub-isomorphisms $\iota\colon \PP\rightarrow \SSS$, $\lambda\colon \QQ\rightarrow \SSS$, $\mu\colon\QQ\rightarrow \TT$ and $\nu\colon\RR\rightarrow \TT$. 
By the first claim, there is a partial order $\UU$ and sub-isomorphisms $\lambda^*\colon \SSS\rightarrow \UU$, $\mu^*\colon \TT\rightarrow \UU$ with $\lambda^*\lambda=\mu^*\mu$. 
Then $\lambda^*\iota$, $\mu^*\nu$ witness that $\PP\simeq \RR$. 
\end{proof}

% by the first part of the previous lemma. 

%\begin{claim} 
%$\pi$ is a projection that preserves well-ordered infima. 
%\end{claim} 
%\begin{proof} 
%\todo{ok?}It follows by an easy and well-known argument that $\pi$ is a projection. 
%To see that $\pi$ preserves infima, suppose that $X\subseteq Q$ and $q=\inf^\QQ(X)$. It follows immediately from the definition of $\pi$ that $\pi(q)\leq \inf^\PP(\pi[X])$. 
%
%Suppose that $p\in\PP$, $p\geq q$. 
%\end{proof} 

\begin{definition} \label{pull back names} 
Suppose that $\PP$, $\QQ$ are forcings. 
\begin{enumerate-(a)} 
\item 
A \emph{complete subforcing $\PP$ of $\QQ$} ($\PP \lessdot \QQ$) is a subforcing of $\QQ$ such that every maximal antichain in $\PP$ is maximal in $\QQ$. 
\item 
A \emph{complete embedding} $i\colon \PP\rightarrow \QQ$ is a homomorphism with respect to $\leq$ and $\perp$ with the property that for every $q\in\QQ$, there is a condition $p\in \PP$ (called a \emph{reduction of $q$}) such that for every $r\leq p$ in $\PP$, $i(r)$ is compatible with $q$. 
\item 
%(see \cite[Definition 5.2]{MR2768691}) 
Suppose that $i\colon \PP\rightarrow \QQ$ is a complete embedding and $G$ is $\PP$-generic over $V$. 
The \emph{quotient forcing $\QQ/G$ for $G$ in $\QQ$} is defined as the subforcing 
$$\QQ/G=\{q\in \QQ\mid \forall p\in G\ i(p)\parallel q \}$$ 
of $\QQ$. 
Moreover, we fix a $\PP$-name $ \QQ/\PP$ for  for the quotient forcing for $\PP$ in $\dot{G}$, where $\dot{G}$ is a $\PP$-name for the $\PP$-generic filter, and also refer to this as (a name for) the quotient forcing for $\PP$ in $\QQ$. 
%\item 
%\todo{isn't this in the other direction, and do we need this for projections?} 
%Suppose that \todo{$i$ instead of $\pi$?} $\pi\colon \PP\rightarrow \QQ$ is a complete embedding. 
%\todo[inline]{does this make sense both ways? used at top of page 9} 
%We \todo{write $i$ instead of $\pi$} define $\tau^\pi$ for all $\QQ$-names $\tau$ by induction as 
%$$\tau^{\pi}=\{(\sigma^{\pi},p) \mid p\in \PP,\ \exists q\in\QQ\  (\sigma,q)\in \tau,\  \pi(p)\leq q\}.$$ 
\end{enumerate-(a)} 
\end{definition} 
%\todo{write definition. replace $p$ with the set of all  $q\leq p$ in $\ran{\pi}$} If $\pi\colon \PP\rightarrow \QQ$ is a complete embedding and $\sigma$ is a $\QQ$-name, let $\sigma^{\pi}$ denote the $\PP$-name induced by $\sigma$. 

It is a standard fact that a subforcing $\PP$ of $\QQ$ is a complete subforcing if and only if the identity on $\PP$ is a complete embedding. 
%\todo[inline]{REWRITE. CHECK if and how this is used. only for dense embeddings?} 
%In particular, for any dense subforcing $\PP$ of $\QQ$, the $\PP$-names are translated to $\QQ$-names. 
%In the situation of Definition \ref{pull back names}, we have $(\tau^{\pi})^{\pi^{-1}[G]}=\tau^G$ for all $\QQ$-generic filters $G$. 
%This will be often used in the proofs below.  

\begin{definition} \label{definition: quotient forcing} (see \cite[Definition 0.1]{MR717829}, \cite[Definition 5.2]{MR2768691}) 
Suppose that $\PP$ and $\QQ$ are forcings. 
\begin{enumerate-(a)} 
\item 
A 
%\todo{call just "projection"? everywhere?} 
\emph{projection} $\pi\colon \QQ\rightarrow \PP$ is a homomorphism with respect to $\leq$ such that $\pi[\QQ]$ is dense in $\PP$ and for all $q\in \QQ$ and all $p\leq \pi(q)$, there is a condition $\bar{q}\leq q$ with $\pi(\bar{q})\leq p$. 
\item 
Suppose that $\pi\colon \QQ\rightarrow \PP$ is a projection and $G$ is a $\PP$-generic filter over $V$. 
The \emph{quotient forcing $\QQ/G$ for $G$ in $\QQ$ relative to $\pi$} is defined as the subforcing 
$$\QQ/G=\{q\in \QQ\mid \pi(q)\in G\}$$ 
of $\QQ$. 
Moreover, we fix a $\PP$-name $(\QQ/\PP)^{\pi}$ for the quotient forcing for $\dot{G}$ in $\QQ$ relative to $\pi$, where $\dot{G}$ is a $\PP$-name for the $\PP$-generic filter, and will refer to this as (a name for) the quotient forcing for $\PP$ in $\QQ$ relative to $\pi$. 
%$\QQ$ in $\PP$, \todo{say something more, what is this name?} given by a $\PP$-name for the $\QQ$-generic filter induced by $\pi$. 
\end{enumerate-(a)} 
\end{definition} 

In Definition  \ref{definition: quotient forcing}, 
by standard facts about quotient forcing, 
for any $\PP$-generic filter $G$ over $V$, 
any $\QQ/G$-generic filter $H$ over $V[G]$ is $\QQ$-generic over $V$. 
Moreover, 
%if $\dot{\QQ}$ is a $\PP$-name for the quotient forcing, then 
any $\QQ$-generic filter $H$ over $V$ induces the $\PP$-generic filter $G=\pi[H]$ over $V$ and $H$ is $[(\QQ/\PP)^\pi]^H$-generic over $V[G]$ with 
%the $\QQ*\dot{\QQ}$-generic filter $\pi[G]*H$ and 
$V[H]=V[G*H]$. 
%It is easy to see that in Definition \ref{definition: quotient forcing}, a
Assuming that $\PP$ and $\QQ$ have weakest elements $\one_\PP$ and $\one_\QQ$, respectively, it is easy to see that the condition that $\pi[\QQ]$ is dense in $\PP$ in Definition \ref{definition: quotient forcing} is equivalent to the condition that $\pi(\one_\QQ)=\one_\PP$. 
%In the situation of Definition \ref{definition: quotient forcing}, 

It is easy to check that the following map is actually a projection. 

%\todo{move} 
\begin{definition} \label{definition of natural projection} 
Suppose that $\PP$, $\QQ$ are complete Boolean algebras and $\PP$ is a complete subalgebra of $\QQ$. 
We define the \emph{natural projection} $\pi\colon \QQ\rightarrow \PP$ by 
\begin{center}
$\pi(q)=\inf_{p\in\PP,\ p\geq q}p.$ 
\end{center} 
\end{definition}

We will further use the following notation when working with quotient forcings induced by names. 
%If $\RR$ is a forcing, let $\BB(\RR)$ denote the Boolean completion of the \todo{need this? write further above?} separative quotient of $\RR$, which will in fact be the Boolean completion in all cases below, since we will only apply this to separative forcings. 
%but since \todo{check} we will only apply this to separative forcings, it will simply be the Boolean completion. 

\begin{definition} \label{definition: generated Boolean subalgebra} 
If $\PP$ is a complete Boolean algebra and $\sigma$ is a $\PP$-name for a subset of a set $x$, let $\BB(\sigma)=\BB^{\PP}(\sigma)$ denote the complete Boolean subalgebra of $\PP$ that is generated by the Boolean values $\llbracket y\in \sigma \rrbracket_{\PP}$ for all $y\in x$. 
%Moreover, if $\RR$ is a separative forcing and $\sigma$ is an $\RR$-name, then we will assume that $\RR$ is a subset of 
Moreover, we will also use this notation if $\sigma$ is a name for a set that can be coded as a subset of a ground model set in an absolute way. 
\end{definition}

Moreover, we will often add Cohen subsets to a regular cardinal $\kappa$ with $\kappa^{<\kappa}=\kappa$. 
%\todo{check consistent usage everywhere} 
The following definition of the forcing for adding Cohen subsets is non-standard, but essential in several proofs below. 
%In the proofs below, we will need the following non-standard definition of the forcing for adding a Cohen subset of $\kappa$. 
% is non-standard, but it is essential for the following proofs. 
In the following definitions, let $\mathrm{Succ}$ denote the class of successor ordinals. 
% together with $0$. 

\begin{definition} 
Suppose that $\lambda$ is a regular uncountable cardinal.  
\begin{enumerate-(a)} 
\item 
$\Add{\lambda}{1}$ is defined as the forcing 
%with conditions 
$$\Add{\lambda}{1}=\{p\colon \alpha\rightarrow \lambda\mid \alpha<\lambda\},$$ ordered by reverse inclusion. 
\item 
$\Addd{\lambda}{1}$ is defined as the dense subforcing 
$$\Addd{\lambda}{1}=\{p\in \Add{\lambda}{1}\mid \dom{p}\in \mathrm{Succ}\}$$ 
of $\Add{\lambda}{1}$. 
\item 
$\Add{\lambda}{\gamma}$ is defined as the ${<}\lambda$-support product $\prod_{i<\gamma}\Add{\lambda}{1}$ for any ordinal $\gamma$. 
%=\emptyset\ \text{ or }\exists \alpha<\kappa\ \dom{p}=\alpha+1\},$$ 
%ordered by reverse inclusion. 
\end{enumerate-(a)} 
\end{definition} 

We will 
%quite 
often use the following standard facts about adding Cohen subsets and collapse forcings.

\begin{lemma} \label{forcing equivalent to Add(kappa,1)} 
Suppose that $\lambda$ is a regular uncountable cardinal.  
%Let $\PP=\{p\in \Add{\kappa}{1}\mid \dom{p}\in \Succ\}$. 
\begin{enumerate-(1)} 
\item 
If  $\lambda^{<\lambda}=\lambda$ and $\PP$ is a non-atomic ${<}\lambda$-closed forcing of size $\lambda$, then $\PP$ has a dense subset that is isomorphic to $\Addd{\lambda}{1}$. In particular, $\PP$ is sub-equivalent to $\Add{\lambda}{1}$. 
\item 
\cite[Lemma 2.2]{MR2387944} 
Suppose that $\nu>\lambda$ is a cardinal with $\nu^{<\lambda}=\nu$, $\PP$ is a separative ${<}\lambda$-closed forcing of size $\nu$ and $\one_\PP$ forces that $\nu$ has size $\lambda$. 
Then $\PP$ has a dense subset that is isomorphic to the dense subforcing 
$$\mathrm{Col}^*(\lambda,\nu)=\{p\in \Col{\lambda}{\nu}\mid \dom{p}\in\mathrm{Succ}\}$$ 
of $\Col{\lambda}{\nu}$. 
In particular, $\PP$ is sub-equivalent to $\Col{\lambda}{\nu}$. 
%$\PP \lessdot \Add{\kappa}{1}$ and $|\PP|=\kappa$, then $\PP$ has a dense subset isomorphic to $\Addd{\kappa}{1}$. 
\end{enumerate-(1)} 
\end{lemma} 
\begin{proof} 
Since the proof of the first claim is straightforward and well-known, we do not include it here. 
The proof of the second claim is an adaptation of the proof of \cite[Lemma 26.7]{MR1940513} that can be found in \cite[Lemma 2.2]{MR2387944}. 
\end{proof}

%\todo[inline]{define various quotient notations?????? $\QQ/\PP$ if $\PP$ is a complete subforcing of $\QQ$, $(\QQ/\PP)^{\pi}$ if $\pi\colon \QQ\rightarrow \PP$ is a projection} 

%\todo{what did I mean here? $(\PP\times \RR)/\QQ$ and $(\PP/\QQ)\times \RR)$?}In the situation of Definition \ref{definition: quotient forcing}, for any forcing $\RR$ the quotient forcing $(\PP\times \RR)/\QQ$ is isomorphic to $(\PP\times \QQ)/\RR$. We will use this in the proofs below. 

\iffalse 
\begin{lemma} \label{restriction of projection to dense subforcing} 
\todo{probably comment out, if not cited below} Suppose that $\pi\colon \QQ\rightarrow \RR$ is a projection and $\PP$ is a dense subforcing of $\QQ$. Then $\pi\upharpoonright \PP\colon \PP\rightarrow \RR$ is a projection and $\pi[\PP]$ is a dense subforcing of $\RR$. 
\end{lemma} 
\begin{proof} 
Suppose that $p\in \PP$ and $r\leq \pi(p)$. 
Since $\pi$ is a projection, there is some $\bar{p}\leq p$ in $\PP$ with $\pi(\bar{p})\leq r$. Since $\PP$ is dense in $\QQ$, there is some $q\leq \bar{p}$ in $\QQ$ with $\pi(q)\leq \pi(\bar{p})\leq r$. 
The second claim follows from the fact that $\pi\upharpoonright \PP$ is a projection. 
\end{proof} 

\fi

%DEFINE: 
%If $\sigma\colon \QQ\rightarrow \RR$ is a forcing projection, let $(\QQ/\RR)^{\sigma}$ denote the quotient of $\RR$ in $\QQ$ with respect to $\sigma$. 

%DEFINE: Quotient forcing for $\PP \subseteq \BB(\QQ)$, instead of $\PP \subseteq \QQ$???? WE USE THIS OFTEN!!! 

%%%%%%%%%%%%%%%%%%%%%%%%%%%%%%%%%%%%%%%%%%%%%%%%%%
\subsection{A counterexample for quotient forcings} 
%A quotient of $\Add{\kappa}{1}$ that does not preserve stationary sets} 
%A forcing with no strategically closed quotient} 
%\todo[inline]{maybe call: a counterexample for quotient forcings?} 

The following well-known example of a quotient of $\Add{\kappa}{1}$ that does not preserve stationary subsets of $\kappa$ shows that the proofs of regularity properties for definable subsets of ${}^{\omega}\omega$ in Solovay's model do not generalize to any uncountable regular cardinal. 
% $\lambda$. 

%The following folklore result is an obstacle in generalizing the perfect set property to uncountable $\kappa$. 
%Let $Add(\kappa,1)$ denote the forcing to add a Cohen subset of $\kappa$ with bounded conditions. 
For any uncountable regular cardinal $\kappa$ with $\kappa^{<\kappa}=\kappa$, we define a complete subforcing of the Boolean completion $\BB(\Add{\kappa}{1})$ of $\Add{\kappa}{1}$ such that its quotient forcing in $\BB(\Add{\kappa}{1})$ does not preserve stationary subsets of $\kappa$. 
%In the following lemma, $\Adddd{\kappa}{1}$ denotes the forcing with conditions $p\colon \alpha\rightarrow 2$ for $\alpha<\kappa$, ordered by inclusion. 
%This forcing is equivalent to $\Add{\kappa}{1}$ by Lemma \ref{forcing equivalent to Add(kappa,1)}. 

For any condition $p\in\Add{\kappa}{1}$, we consider the set $s_p=\{\alpha\in \mathrm{dom}(p)\mid p(\alpha)\neq 0\}$. 
Suppose that $\dot{G}$ is an $\Add{\kappa}{1}$-name for the generic filter $G$ and $\dot{S}$ is an $\Add{\kappa}{1}$-name for $\bigcup_{p\in G}s_p$. 
Then $\one_{\Add{\kappa}{1}}$ forces that $\dot{S}$ is a bi-stationary subset of $\kappa$. 
%the subset of $\kappa$ with characteristic function $f_G=\bigcup G$. Then $\dot{S}^G$ is a fat stationary subset of $\kappa$. 
Moreover, if $S$ is a subset of $\kappa$, we define  
% that is coded by $p$. For any subset $S$ of $\kappa$, let 
$$\mathbb{Q}_S=\{p\in \Add{\kappa}{1}\mid %\alpha<\omega_1,\ %s_p\subseteq S,\ 
s_p\text{ is a closed subset of }S\}.$$ 
%If $S$ is a fat stationary subset of $\kappa$, let $\QQ_S$ denote the subforcing 
%Let $s_p=\{\alpha\in \mathrm{dom}(p)\mid p(\alpha)\neq 0\}$ for $p\in {}^{<\kappa}2$. 
%Suppose that $S\subseteq \kappa$ \todo{define fat stationary before the lemma} is \emph{fat stationary} in $\kappa$, i.e. it is a stationary subset of $\kappa$ and contains closed subsets of arbitrary length $\alpha<\kappa$. 
%$$\mathbb{Q}_S=\{p\in \Add{\kappa}{1}\mid %\alpha<\omega_1,\ %s_p\subseteq S,\ 
%s_p\text{ is a closed subset of }S\}$$ 
%of $\Add{\kappa}{1}$. 

\begin{lemma} \label{a forcing with bad quotient} 
%Suppose that $\dot{G}$ is an $\Add{\omega_1}{1}$-name for the generic filter, $\dot{S}$ is an $\Add{\omega_1}{1}$-name for the subset of $\kappa$ with characteristic function $f_G=\bigcup G$ and 
Suppose that $\kappa^{<\kappa}=\kappa$ and $\dot{\mathbb{Q}}$ is an $\Add{\kappa}{1}$-name for $\mathbb{Q}_{\dot{S}}$, where $\dot{S}$ is defined as above. 
Then $\Add{\kappa}{1}*\dot{\mathbb{Q}}$ is sub-equivalent to $\Add{\kappa}{1}$. 
%where $S\subseteq\omega_1$ is the $Add(\omega_1,1)$-generic. 
%Then there is a dense subset of $Add(\omega_1,1)*\dot{\mathbb{Q}}$ which is isomorphic to $Add(\omega_1,1)$. 
\end{lemma} 
\begin{proof} 
%\todo{define $\len(t)$ first} 
The set 
$$\{(p,\check{q})\mid p,q\in \Add{\kappa}{1},\ \dom{p}=\dom{q},\ p\Vdash_{\Add{\kappa}{1}}\check{q}\in \dot{\QQ}\}$$ 
is a non-atomic ${<}\kappa$-closed dense subset of $\Add{\kappa}{1}*\QQ$, hence $\Add{\kappa}{1}*\QQ$ is sub-equivalent to $\Add{\kappa}{1}$ by Lemma \ref{forcing equivalent to Add(kappa,1)}. 
%Let $D=\{(p,\check{q})\mid p\colon \alpha\rightarrow 2,\ q\colon \alpha\rightarrow 2,\ \alpha<\omega_1,\ s_q\subseteq s_p,$ $s_q$ is a closed subset of $\omega_1 \}$. 
%We claim that $D$ is a dense subset of $\Add{\omega_1}{1}*\dot{\mathbb{Q}}$. 
%uppose that $(p,\dot{q})\in \Add{\omega_1}{1}*\dot{\mathbb{Q}}$. 
%Suppose that $G*H$ is $\Add{\omega_1}{1}*\dot{\mathbb{Q}}$-generic over $V$ with $(p,\dot{q})\in G*H$. 
%Suppose that $f=\bigcup G$ and $g=\bigcup H$. 
%$(f,g)\in ({}^{\kappa}2)^2$ is $Add(\kappa,1)$-generic over $V$ with $p\subseteq f$ and $q\subseteq g$. 
%Suppose that $r\leq p$, $r\subseteq f$ and $r$ decides $\dot{q}$, i.e. there is some $q\in {}^{<\kappa}2$ such that $r\Vdash \dot{q}=\check{q}$. 
%$r\leq p$ and $r\Vdash \dot{q}=\check{q}$ for some $q\in {}^{<\kappa}2$. 
%We can assume that $r$ has successor length $\alpha$ and $\alpha\geq |q|$. 
%Let $s=g\upharpoonright \alpha$. Then $(r,s)\leq (p,\dot{q})$ and $(r,s)\in D$. 
%Since $D$ is $<\kappa$-closed and has size $\kappa$, this proves the claim. 
\end{proof} 

It is forced by $\one_{\PP}$ that $\dot{\QQ}$ shoots a club through $\dot{S}$ and hence $\dot{\QQ}$ is not stationary set preserving. Thus, $\dot{\QQ}$ is a name for the required quotient forcing. In particular, such a forcing fails to be ${<}\kappa$-closed. 

%\begin{remark} 
Now suppose that $\kappa$ is an uncountable regular cardinal and $\lambda>\kappa$ is inaccessible. 
An argument analogous to the proof of \cite[Theorem 1]{MR0265151} shows that after forcing with $\Col{\kappa}{{<}\lambda}$, all $\bf{\Sigma}^1_1$ subsets of ${}^{\kappa}\kappa$ have the perfect set property. However, this proof fails to work for $\bf{\Pi}^1_1$ subsets of ${}^{\kappa}\kappa$ precisely because some quotient forcings, such as the ones appearing in Lemma \ref{a forcing with bad quotient}, are not necessarily ${<}\kappa$-closed. 
%\end{remark} 
%easily modified to show that the perfect set property for $\bf{\Sigma}^1_1$ subsets of ${}^{\kappa}\kappa$ holds in $Col(\kappa,<\lambda)$-generic extensions. In the case of arbitrary subsets of ${}^{\kappa}\kappa$ definable from ordinals and subsets of ${}^{\kappa}\kappa$, the proof fails by the following example. 
%\begin{example} Suppose \todo{rewrite} $\kappa^{<\kappa}=\kappa>\omega$. Let $\mathbb{P}$ consist of the pairs $(p,q)\in ({}^{<\kappa}2)^2$ such that $s_q$ is closed and $s_q\subseteq s_p$ for the sets $s_p,s_q\subseteq\kappa$ with characteristic functions $p,q$. The conditions are ordered by coordinatewise end extension. Let $\mathbb{Q}=\mathbb{P}*Col(\kappa,<\lambda)^{\check{}}$. 
%\end{example} 
%
%The \todo{rewrite} forcing $\mathbb{P}$ is $<\kappa$-closed and has size $\kappa$, so it is forcing equivalent to $Add(\kappa,1)$ (in the sense that there are isomorphic dense subsets of these two forcings). Hence $\mathbb{Q}$ is equivalent to $Col(\kappa,<\lambda)$. 
%Moreover, it is easy to see that $\mathbb{P}$ is equivalent to $Add(\kappa,1)*\dot{\mathbb{S}}$, where $Add(\kappa,1)$ adds (with conditions of size $<\kappa$) a Cohen subset of $\kappa$, which is stationary, and $\dot{\mathbb{S}}$ shoots a club through its complement. Since $Col(\kappa,<\lambda)$ is $<\kappa$-closed, it does not change the truth of $\Sigma^1_1$ statements over ${}^{\kappa}\kappa$ and hence does not destroy the stationarity of subsets of $\kappa$. Hence the quotient of $Add(\kappa,1)$ in $Col(\kappa,<\lambda)$ is not equivalent to $Col(\kappa,<\lambda)$. 

\begin{remark} 
%\todo{check} 
In the situation of Lemma \ref{a forcing with bad quotient}, $\one_{\PP}$ forces that $\QQ_{\dot{S}}$ is 
%in the $\Add{\kappa}{1}$-generic extension, $S$ 
%is a stationary subset of $\kappa$ and the forcing $\QQ_S$ is 
$\lk$-distributive, since it appears in a two-step iteration which is $\lk$-distributive. 
%is a quotient forcing in $\Add{\kappa}{1}$. 
% is $\lk$-dsitributive. 
%it appears in a ${<}\kappa$-distributive $2$-step iteration. 
However, in general one needs to require more conditions on $S$ to ensure that $\QQ_S$ is $\lk$-distributive. For instance, assuming that the $\mathrm{GCH}$ holds, it is sufficient that $S$ is a \emph{fat stationary subset of $\kappa$} in the sense that for every club $C$ in $\kappa$, $S\cap C$ contains closed subsets of arbitrarily large order types below $\kappa$ (see \cite[Theorem 1 \& Theorem 2]{MR716625}). 
\end{remark} 

%A subset $S$ of a regular cardinal $\kappa$ is \emph{fat stationary} if for every $\alpha<\kappa$, there is a closed subset of $S$ with order type $\alpha$. 

%%%%%%%%%%%%%%%%%%%%%%%%%%%%%%%%%%%%%%%%%%%%%%%%%%
%%%%%%%%%%%%%%%%%%%%%%%%%%%%%%%%%%%%%%%%%%%%%%%%%%
\section{The perfect set property} \label{section: The perfect set property} 

We always assume that $\kappa$ is an uncountable regular cardinal with $\kappa^{<\kappa}=\kappa$ and that $\lambda$ is an uncountable regular cardinal. 
We define the \emph{length} of various types of objects in the next definition. 

\begin{definition} 
%\todo{maybe in earlier section?}
\begin{enumerate-(a)} 
\item 
%\todo{check with later definition of $\len$ and $\he$} 
Let $\len(s)=\dom{s}$ for any function $s$. 
%for $s\in {}^{<\kappa}V$. 
\item 
Let $\len(t)=\sup_{s\in t}\len(s)$ 
%\todo{should it say $l(s)+1$ instead of $l(s)$ in the definition of $\he(t)$?} 
for $t\subseteq{}^{<\lambda}\lambda$. 
\item 
Let $\len(p)=\len(t)$ 
for $p=(t,s)$ and $t,s\subseteq {}^{<\lambda}\lambda$. 
%$t\subseteq{}^{<\kappa}\kappa$. 
%\item 
%Let $l(t,s)=l(t)$ for $s,t\subseteq {}^{<\kappa}\kappa$ with $s\subseteq t$. 
\end{enumerate-(a)} 
\end{definition}

%%%%%%%%%%%%%%%%%%%%%%%%%%%%%%%%%%%%%%%%%%%%%%%%%%
\subsection{Perfect set games} 

The perfect set property is characterized by the \emph{perfect set game}. 

\begin{definition} 
The \emph{perfect set game} $F_\lambda(A)$ of length $\lambda$ for a subset $A$ of ${}^{\lambda}2$ is defined as follows. 
%The game has length $\kappa$. 
The first (even) player, player I, plays some $s_\alpha\in {}^{{<}\lambda}2$ in all even rounds $\alpha$. 
% and the second (odd) player plays in all odd rounds. 
The second (odd) player, player II, plays some $s_\alpha\in {}^{{<}\lambda}2$ in all odd rounds $\alpha$. 
Together, they play a strictly increasing sequence $\vec{s}=\langle s_\alpha\mid \alpha<\lambda\rangle$ with $s_\alpha\in {}^{<\lambda}2$ for all $\alpha<\lambda$. 
Player II has to satisfy the additional requirement that $\len(s_{\alpha+1})=\len(s_\alpha)+1$ for all even ordinals $\alpha<\lambda$. 
The combined sequence $\vec{s}$ of moves of both players defines a sequence 
$$\bigcup_{\alpha<\lambda} s_\alpha= x=\langle x(i)\mid i<\lambda \rangle \in {}^{\lambda}2.$$ 
Player I wins if $x\in A$. 
Moreover, if $t\in {}^{<\lambda}2$, 
the game $F^{t}_\lambda(A)$ is defined as $F_\lambda(A)$ with the additional requirement that $t\subseteq s_0$ for the first move $s_0$ of player I. 
\end{definition} 

The perfect set game characterizes the perfect set property for subsets of ${}^{\lambda}2$ in the following sense. 

\begin{lemma} \cite[Lemma 7.2.2]{Kovachev-thesis} \label{characterization of perfect set property by game} 
Suppose that $A$ is a subset of ${}^{\lambda}2$ and $t\in {}^{<\lambda}2$. 
\begin{enumerate-(1)} 
\item \label{I wins if the set has a perfect subset} 
Player I has a winning strategy in $F^t_\lambda(A)$ if and only if $A\cap N_t$ has a perfect subset. 
\item 
Player II has a winning strategy in $F^t_\lambda(A)$ if and only if $|A\cap N_t|\leq\lambda$. 
\end{enumerate-(1)} 
\end{lemma} 

The perfect set property is equivalent to the following variant defined in \cite[Section 2]{MR1110032}. 

\begin{definition} 
The game $V_\lambda(A)$ of length $\lambda$ for a subset $A$ of ${}^{\lambda}2$ is defined as follows. 
%The game has length $\kappa$. 
The first (even) player, player I, plays an ordinal $\alpha_i$ in all even rounds $i$. 
% and the second (odd) player plays in all odd rounds. 
The second (odd) player, player II, plays an element $x_i$ of $A$ in all odd rounds $i$. 
Moreover, the sequence 
%Player I plays a strictly increasing continuous sequence 
$\langle \alpha_i\mid i<\lambda\rangle$ of moves of player I has to be continuous. 
% of ordinals below $\kappa$. 
%Player II plays a sequence $\langle x_i\mid i<\kappa\rangle$ of elements of $A$. 
Player II wins if for all $i<j<\lambda$, $x_i{\upharpoonright} \alpha_i= x_j{\upharpoonright} \alpha_i$ and $x_i\neq x_j$. 
\end{definition} 
%We denote V\"a\"an\"anen's game for a subset $A$ of ${}^{\kappa}\kappa$ defined in \todo{cite} by $V(A)$. 

\begin{lemma} 
Suppose that $A$ is a subset of $\ltl$. 
Then $A$ has a perfect subset if and only there is a closed subset $C$ of $A$ such that player II has a winning strategy in $V_\lambda(C)$. 
\end{lemma} 
\begin{proof} 
If $A$ has a perfect subset $C$, then it is straightforward to define a winning strategy for player II in $V_\lambda(C)$. 
%can win by playing only elements of $C$. 

Now suppose that $C$ is a closed subset of $A$ and that player II has a winning strategy $\sigma$ in $V_\lambda(C)$. 
%\todo{check the indices in the proof} 
Using $\sigma$, we can inductively construct $\langle x_s, t_s, \gamma_s\mid s\in {}^{<\lambda}2\rangle$ such that the following conditions hold for all $r, s\in {}^{<\lambda}2$. 
%$x_s\in {}^{\kappa}\kappa$, $t_s\in {}^{<\kappa}\kappa$, $\gamma_s\in \Ord$ and 
\begin{enumerate-(1)} 
\item 
\begin{enumerate-(a)} 
\item 
$t_r\subsetneq t_s$ if $r\subsetneq s$. 
\item 
$t_{r^\smallfrown\langle 0\rangle} \perp t_{r^\smallfrown {\langle 1\rangle}}$. 
% for all $r$, 
\item 
$t_s=\bigcup_{u\subsetneq s} t_u$ if $\len(s)$ is a limit. 
\end{enumerate-(a)} 
\item 
$\len(t_s)=\gamma_s$. 
\item 
$t_s\subseteq x_s$. 
% for all $s$, 
\item 
%$\langle x_{z\upharpoonright (\alpha+1)}, \gamma_{z\upharpoonright\alpha} \mid \alpha \in C_z\rangle$ \ \\ 
%$\langle x_{\emptyset}\rangle^\smallfrown 
Let $c_s$ denote the closure of the set $\{\alpha<\len(s)\mid \exists \bar{\alpha}\ \alpha=\bar{\alpha}+1,\ s(\bar{\alpha})=1\}$ and $\pi\colon c_s\rightarrow \delta_s$ its transitive collapse. 
Then 
$$\langle \gamma_{s\upharpoonright \pi^{-1}(\alpha)}, x_{s\upharpoonright\pi^{-1}(\alpha)}  \mid \alpha< \delta_s\rangle$$ 
is a partial run of $V_\lambda(C)$ according to $\sigma$.

%where $s_z$ is the closure of the set $\{\alpha<\len(s)\mid s(\alpha)=1\}$. 
% in $\kappa$. 
\end{enumerate-(1)} 

The last condition ensures the existence of partial runs that split exactly at the times $\alpha$ where $s$ has successor length $\alpha$ and the last value $1$ and at the limits of such times. In particular, whenever $\alpha<\delta_s$, $\pi^{-1}(\alpha)=\bar{\alpha}+1$ and $s(\bar{\alpha})=1$, the partial run for $s$ is extended by player I playing an ordinal $\gamma_{s{\upharpoonright}(\bar{\alpha}+1)}$ and player II responding with an element $x_{s{\upharpoonright}(\bar{\alpha}+1)}$ of $A$ that splits from $x_{s{\upharpoonright}\bar{\alpha}}$, and whenever $s(\bar{\alpha})=0$, the partial run for $s$ is not extended. 

We thus obtain a perfect tree $T=\{t\in {}^{<\lambda}2\mid \exists s\in {}^{<\lambda}2\ t\subseteq t_s\}$. 
%and by the construction, every branch in $T$ is in $C$
Since $C$ is closed, it follows from the construction that $[T]\subseteq C$, proving the claim. 
%We define runs $\vec{s}_t$ for all $t\in {}^{<\kappa}2$ such that $\vec{s}_{t^\smallfrown\langle 0\rangle}(\gamma)=\vec{s}_{t^\smallfrown\langle 1\rangle}(\gamma)$ and $\vec{s}_{t^\smallfrown\langle 0\rangle}(\gamma+1)\neq \vec{s}_{t^\smallfrown\langle 1\rangle}(\gamma+1)$ for $\len(t)=\gamma$. 
%
%We define $\langle u_s\mid s\in {}^{<\kappa}2\rangle$ such that for all $t\in {}^{<\kappa}2$ 
%\begin{enumerate-(a)} 
%\item 
%$t_y\upharpoonright \len(s)=u_s$ if $s\subseteq y$, 
\end{proof}

%%%%%%%%%%%%%%%%%%%%%%%%%%%%%%%%%%%%%%%%%%%%%%%%%%
\subsection{The perfect set property for definable sets} \label{subsection perfect set property}
%Consistency of the perfect set property} 

%\todo[inline]{maybe call: The perfect set property for definable sets?} 

We will show that forcing with $\Add{\kappa}{1}$ adds a perfect set of $\Add{\kappa}{1}$-generic elements of $\kk$ whose quotient forcings are sub-equivalent to $\Add{\kappa}{1}$. More precisely, each of these elements 
%of $\kk$ 
will have an $\Add{\kappa}{1}$-name that generates a complete subalgebra of $\BB(\Add{\kappa}{1})$ whose quotient forcing in $\BB(\Add{\kappa}{1})$ is 
%and names for them such that for the complete subalgebra of $\BB(\Add{\kappa}{1})$ generated by each name, the quotient forcing is again 
sub-equivalent to $\Add{\kappa}{1}$. 
%\begin{definition} Suppose $V[G]$ is a generic extension of $V$ by a $<\kappa$-closed forcing and $x\in V[G]$. Then $x$ \emph{factors} in $(V,V[G])$ if $V[G]$ is a generic extension of $V[x]$ by a $<\kappa$-closed forcing. 
%\end{definition} 

This will be proved by considering the 
following forcing $\PP$. The forcing adds a perfect subtree of ${}^{<\kappa}\kappa$ by approximations of size ${<}\kappa$. 

\begin{definition} \label{definition of the forcing adding a perfect tree} 
Let $\mathbb{P}$ denote the set of %the 
pairs $(t,s)$ such that  
\begin{enumerate-(a)} 
\item 
$t\subseteq{}^{<\kappa}\kappa$ 
%\todo{should we work with ${}^{<\kappa}2$ instead?} 
is a tree of size ${<}\kappa$, 
\item 
every node $u\in t$ has at most two direct successors in $t$, 
\item 
$s\subseteq t$ and 
if $u\in t$ is non-terminal in $t$, then $u\in s$ if and only if $u$ has exactly one successor in $t$. 
%\todo{i.e. neighbors of nodes in $s$ cannot be in $t$} is such that every $p\in s$ has successor height and $q^\smallfrown\beta\notin t$ for $q,\alpha$ with $p=q^\smallfrown \alpha\in s$ and all $\beta\neq\alpha$. 
\end{enumerate-(a)} 
Let $(t,s)\leq (u,v)$ if $u\subseteq t$ 
%$v\subseteq s$ 
and $s\cap u=v$. 
%The conditions ordered by coordinatewise reverse inclusion. 
\end{definition} 
%\todo{maybe it is simpler to let $s$ denote the nodes in $t$ which can only have one successor?}

The set $s$ marks the non-branching nodes in the tree. 
It follows from the definition of $\PP$ that the forcing adds a perfect binary splitting subtree of ${}^{<\kappa}\kappa$. 
Since every decreasing sequence of length $<\kappa$ in $\PP$ has an infimum, $|\PP|=\kappa$ and $\PP$ is non-atomic, the forcing is sub-equivalent to $\Add{\kappa}{1}$ by 
%\todo{reference}
Lemma \ref{forcing equivalent to Add(kappa,1)}. 

In the remainder of this section, we write $T_G=\bigcup_{(t,s)\in G} t$ if $G$ is a $\mathbb{P}$-generic filter over $V$. 

\begin{lemma} 
Suppose that $G$ is $\mathbb{P}$-generic over $V$ and $T=T_G$. 
Then $V[G]=V[T]$. 
\end{lemma} 
\begin{proof} 
Since $T\in V[G]$, it is sufficient to show that $G\in V[T]$. 
Since $G$ is generic, for all $(t,s)\in \mathbb{P}$, $(t,s)\in G$ if and only if $(t,s)$ is compatible with all conditions in $G$. 
Hence the elements $(t,s)$ of $G$ are exactly the pairs $(t,s)$ such that $s\subseteq t\subseteq T$ and $s$ is the set of $u \in t$ such that $u$ has exactly one direct successor in $T$. Hence $G\in V[T]$. 
\end{proof} 
%Then $(t,s)\in G$ if and only if $t$ is a subtree of $T$ of size $<\kappa$ with $s\subseteq t\subseteq T$, every $p\in s$ has successor height, and $q^\smallfrown\beta\notin T$ for $q,\alpha$ with $p=q^\smallfrown\alpha$ and all $\beta\neq\alpha$. 
%Therefore $V[G]=V[T]$. 
%Hence we will say that $T$ is $\mathbb{P}$-generic. 
%\end{proof} 

If $b=\cup g$ for some $\Add{\kappa}{1}$-generic filter $g$ over $V$ as in the next lemma, we will also say that $b$ is $\Add{\kappa}{1}$-generic over $V$. 

\begin{lemma} \label{distinct branches are mutually generic}
Suppose that $G$ is $\mathbb{P}$-generic and $b,c$ are distinct branches in $T=T_G$. 
Then there is an $\Add{\kappa}{1}\times \Add{\kappa}{1}$-generic filter $g\times h$ over $V$ in $V[T]$ such that $b=\bigcup g$ and $c=\bigcup h$. 
\end{lemma} 
\begin{proof} 
Suppose that $b$, $c$ are distinct branches in $T$ and $\sigma$, $\tau$ are $\PP$-names for $b$, $c$ in the sense that $\sigma^G=b$ and $\tau^G=c$. 
Moreover, let $\dot{T}$ be a $\mathbb{P}$-name for $T$. 
% and $\sigma$, $\tau$ are $\PP$-names for two distinct branches $b$, $c$ in $T$ in the sense that 
Then there is a condition $p_0\in G$ with $p_0\Vdash_{\mathbb{P}}\sigma,\tau\in [\dot{T}]$ and $p_0\Vdash_{\mathbb{P}} \sigma\neq\tau$. 
We can assume that $p_0=\one_\PP$ by replacing $\sigma$, $\tau$ with names that satisfy these conditions for $p_0=\one_\PP$. 

Now suppose that $D$ is a dense open subset of $\Add{\kappa}{1}\times \Add{\kappa}{1}$ and let 
$$E=\{q\in \mathbb{P}\mid \exists (u,v)\in D,\ q\Vdash_{\mathbb{P}} u\subseteq \sigma,\ v\subseteq\tau\}.$$ 
\begin{claim*} 
$E$ is dense. 
% below $p_0$. 
%For every $q\leq p\in\mathbb{P}$ there is some $r\leq q$ and some $(u,v)\in D$ such that  $r\Vdash u\subseteq\sigma$ and $r\Vdash v\subseteq\tau$. 
\end{claim*} 
\begin{proof} 
Suppose that $p\in\PP$. 
%\leq p_0$. 
Since $\one_\PP\Vdash_{\PP} \sigma\neq\tau$, we can assume by extending $p$ that for some $\alpha<\len(p)$, $p$ decides $\sigma(\alpha)$, $\tau(\alpha)$ and these values are different. 
%$p\Vdash_\PP \sigma(\alpha)\neq\tau(\alpha)$ 
We let $q_0=p$ and choose successively for each $n\in\omega$ an extension 
$q_{n+1}\leq q_n$ such that $\len(q_n)<\len(q_{n+1})$ and $q_{n+1}$ decides both $\sigma{\upharpoonright} \len(q_n)$ and $\tau{\upharpoonright} \len(q_n)$. %and $l(q_n)<l(q_{n+1})$. 
Finally, let $q=\inf_{n\in\omega} q_n$ and suppose that $q=(t_q, s_q)$. 
%let $\gamma= l(q)$. 
Since the lengths $ \len(q_n)$ form a strictly increasing sequence, $\gamma= \len(q)=\sup_{n\in\omega}\len(q_n)$ is a limit ordinal. 
Moreover, by the choice of the sequence of conditions, there are $u,v\in {}^{\gamma}\kappa$ with $u\neq v$ and $q\Vdash_{\mathbb{P}} \sigma{\upharpoonright} \gamma=u,\ \tau{\upharpoonright} \gamma= v$. 

%\begin{subclaim*} 
We first claim that $(u{\upharpoonright} \alpha),\ (v{\upharpoonright} \alpha) \in t_q$ for all $\alpha<\gamma$. It is sufficient to prove that $(u{\upharpoonright} \alpha) \in t_q$ for all $\alpha<\gamma$ by symmetry. 
%\end{subclaim*} 
%\begin{proof} 
To see this, suppose towards a contradiction that $u{\upharpoonright} \alpha\notin t_q$ for some $\alpha<\gamma$. 
Suppose that $\alpha$ is minimal. 
%Then there is some $r\leq q$ with $r=(t_r,s_r)$ and $u\upharpoonright (\alpha+1)\notin t_r$. 
We extend $q=(t_q,s_q)$ to $r=(t_r,s_r)$ as follows. We choose $\beta<\kappa$ with $u(\alpha)\neq\beta$ and let $t_r=t_q\cup\{u{\upharpoonright}\alpha, (u{\upharpoonright}\alpha)^\smallfrown\langle\beta\rangle\}$ and $s_r=s_q\cup\{u{\upharpoonright}\alpha\}$. Then $r\Vdash_{\PP}u(\alpha)=\beta$ and hence $r\Vdash_{\PP} u\not\subseteq \sigma$, contradicting the fact that $q\Vdash_{\mathbb{P}} \sigma{\upharpoonright}\gamma=u$ by the choice of $q$ and $u$. 
This shows that $u{\upharpoonright} \alpha\in t_q$ for all $\alpha<\gamma$. 
%The same argument shows that $v{\upharpoonright} \alpha\in t_q$ for all $\alpha<\gamma$. 
%\end{proof} 

Since $D$ is dense in $\Add{\kappa}{1}\times \Add{\kappa}{1}$, there are conditions $\bar{u}\leq u$, $\bar{v}\leq v$ with $(\bar{u},\bar{v})\in D$. 
%We can assume that $\gamma:=|q|$ is a limit ordinal and that $q\Vdash (\sigma\upharpoonright\gamma,\tau\upharpoonright\gamma)=(\check{x},\check{y})$ for some $x,y\in {}^{\gamma}\kappa$, by extending $q$ in $\omega$ steps. 
%Then find $(u,v)\leq(x,y)$ in $Add(\kappa,1)^2$ with $(u,v)\in D$. 
Since $D$ is open, we can assume that $\len(\bar{u})=\len(\bar{v})=\delta$ for some limit ordinal $\delta$ with $\gamma<\delta<\kappa$. 
Now let 
$$x= \{ \bar{u}{\upharpoonright} \eta\mid \gamma\leq \eta <\delta\} \cup \{ \bar{v}{\upharpoonright} \eta\mid \gamma\leq \eta <\delta\}.$$ 
Moreover, let $\bar{t}=t_q\cup x$, $\bar{s}=s_q\cup x$ and 
$r=(\bar{t}, \bar{s})$. 
Then $r\in \PP$ and $r\leq p$. 

\begin{subclaim*} 
$r\Vdash_{\mathbb{P}} \bar{u}\subseteq\sigma,\ \bar{v}\subseteq \tau$. 
\end{subclaim*} 
\begin{proof} 
It is sufficient to prove $r\Vdash_{\PP}\bar{u}\subseteq \sigma$ by symmetry. Since $r\leq q$ and $q\Vdash_{\mathbb{P}} \sigma{\upharpoonright}\gamma=u$ by the choice of $u$, we have $r\Vdash_{\PP}\sigma{\upharpoonright}\gamma=u$. 
Since $u=\bar{u}{\upharpoonright}\gamma\in x\subseteq s_q\cup x=\bar{s}$ by the definition of $x$ and $\bar{s}$ and since $r=(\bar{t},\bar{s})\in \PP$, the node $u=\bar{u}{\upharpoonright}\gamma$ has the unique direct successor $\bar{u}{\upharpoonright}(\gamma+1)$ in $t$. Hence $r\Vdash_{\PP}\sigma{\upharpoonright}(\gamma+1)=\bar{u}{\upharpoonright}(\gamma+1)$. An analogous argument shows inductively that $r\Vdash_{\PP}\sigma{\upharpoonright}(\eta+1)=\bar{u}{\upharpoonright}(\eta+1)$ for all $\eta$ with $\gamma\leq\eta<\delta$. Hence $r\Vdash_{\PP}\sigma{\upharpoonright}\delta=\bar{u}$. 
\end{proof} 
%$r\Vdash_{\mathbb{P}} \bar{u}\subseteq\sigma,\ \bar{v}\subseteq \tau$, so
This implies that $r\leq p$ and $r\in E$, proving the claim. 
%Suppose that $|u|=\delta$ and $|v|=\eta$. 
%Suppose that $q=(t,s)$. 
%Let $a=t\cup\{ u\upharpoonright \alpha\mid \gamma<\alpha<\delta\}\cup \{ v\upharpoonright \alpha\mid \gamma\leq \alpha<\eta\}$ 
%and $b= s\cup \{x\upharpoonright (\alpha+1)\mid \gamma<\alpha<\delta\}\cup\{y\upharpoonright (\alpha+1)\mid \gamma<\alpha<\eta\}$. 
%Then $r=(a,b)$ is as required. 
\end{proof} 
Let $g=\{s\in {}^{<\kappa}\kappa\mid s\subseteq b\}$, $h=\{s\in {}^{<\kappa}\kappa\mid s\subseteq h\}$. 
The previous claim implies that $g\times h$ is $\Add{\kappa}{1}\times \Add{\kappa}{1}$-generic over $V$. 
%Then the set of these conditions $q$ is dense in $\mathbb{P}$ and the claim follows from the genericity of $T$. 
\end{proof} 

%Let $\bar{t}=\{w\in {}^{\beta}\kappa\mid \beta<\kappa$ and $p\upharpoonright (\alpha+1)\in t$ for all $\alpha<\beta\}$ denote the closure of $t$ in ${}^{<\kappa}\kappa$. \todo{is this used somewhere?}

We obtain the same result for ${<}\kappa$ many branches in $T_G$. 

\begin{lemma} \label{many distinct branches are mutually generic} 
Suppose that $G$ is $\mathbb{P}$-generic and $\langle b_i\mid i<\gamma\rangle$ is a sequence of distinct branches in $T=T_G$ for some $\gamma<\kappa$. Then there is an $\Add{\kappa}{\gamma}$-generic filter $\prod_{i<\gamma} g_i$ over $V$ in $V[G]$ with $b_i=\bigcup g_i$ for all $i<\gamma$. 
%Then there is an $\Add{\kappa}{1}\times \Add{\kappa}{1}$-generic filter $g\times h$ over $V$ in $V[T]$ such that $b=\bigcup g$ and $c=\bigcup h$. 
\end{lemma} 
\begin{proof} 
The proof is as the proof of Lemma \ref{distinct branches are mutually generic}, but instead of working with names $\sigma$, $\tau$ for branches in $T_G$ with $\one_\PP\Vdash_\PP \sigma\neq \tau$, we work with a sequence $\langle \sigma_i\mid i<\gamma\rangle$ of names for branches in $T_G$ with $\one_\PP\Vdash_\PP \sigma_i\neq \sigma_j$ for all $i<j<\gamma$. 
\end{proof} 

We will show that for every branch $b$ of $T=T_G$, the quotient forcing relative to a name for $b$ is equivalent to $\Add{\kappa}{1}$. 
%\todo{ maybe we can show this for a $\PP$-name $\dot{\RR}$ for a homogeneous $<\kappa$-closed forcing? we don't need it though } 
% i.e. there are dense subsets $D_0\subseteq \mathbb{P}_b$ and $D_1\subseteq Add{\kappa}{1}$ such that $D_0$, $D_1$ are isomorphic. 
Suppose that $\dot{T}$ is a $\mathbb{P}$-name for $T$ and 
$\dot{b}$ is a $\mathbb{P}$-name for a branch in $\dot{T}$, in the sense that these properties are forced by $\one_\PP$. 
Moreover, if $p\in\mathbb{P}$, let 
%\todo{check: dots everywhere!} 
$\dot{b}_p
%=\dot{b}_{t,s}
=\{(\alpha,\beta)\mid p\Vdash \dot{b}(\alpha)=\beta\}$. 

\begin{lemma} \label{decided part of branch} 
If $p=(t,s)\in\mathbb{P}$ and $\gamma\subseteq\dom{\dot{b}_p}$, then 
\begin{enumerate-(1)} 
\item 
$\dot{b}_p{\upharpoonright}\beta\in t$ for all $\beta<\gamma$, if $\gamma$ is a limit, and 
\item 
$\dot{b}_p{\upharpoonright}\gamma\in t$ if $\gamma$ is a successor. 
\end{enumerate-(1)} 
%then $b_p\subseteq t$. 
%\upharpoonright (\gamma+1)\in t$ for all $\gamma<\delta$. 
\end{lemma} 
\begin{proof} 
Suppose that $\gamma$ is least such that the claim fails. 
%Let $\delta=l(b_p)$. 
First suppose that $\gamma$ is a limit. In this case, we define $q=(u,v)\leq p$ by $u=t\cup\{\dot{b}_p{\upharpoonright}\gamma,(\dot{b}_p{\upharpoonright}\gamma)^\smallfrown\langle\eta\rangle\}$ for some $\eta\neq \dot{b}_p(\gamma)$ and $v=s\cup\{\dot{b}_p{\upharpoonright}\gamma\}$. Then $q\Vdash_{\PP}\dot{b}(\gamma)=\eta$, contradicting the definition of $\dot{b}_p$. 
Now suppose that $\gamma$ is a successor. Then $\gamma=\beta+1$ and 
%Suppose that $\gamma<\delta$ is least with $b_p\upharpoonright (\gamma+1)\notin t$. 
$\dot{b}_p{\upharpoonright} \gamma=r^\smallfrown \langle\alpha\rangle$ 
for some $r\in t$ with $\len(r)=\beta$. In particular, $p\Vdash_\PP \dot{b}(\beta)=\alpha$. We distinguish two cases. 

First suppose that $r\in s$. 
If $r$ has a successor $r^\smallfrown \langle \eta\rangle$ in $t$, 
then this successor is unique and $\alpha\neq \eta$, since we have $r^\smallfrown\langle \alpha\rangle=\dot{b}_p{\upharpoonright} \gamma \notin t$ by the assumption on $\gamma$. 
Then $p\Vdash_{\mathbb{P}} \dot{b}(\beta)=\eta$, contradicting the fact that $p\Vdash_\PP \dot{b}(\beta)=\alpha$. 
If $r$ has no successor in $t$, let $\eta$ be an ordinal below $\kappa$ with $\eta\neq\alpha$. 
Let $u=t\cup\{r^\smallfrown \langle \eta\rangle\}$, $v=s$
%\cup\{r^\smallfrown \langle \eta\rangle\}$ 
and $q=(u,s)$. 
Then $q\Vdash_{\mathbb{P}} \dot{b}(\beta)=\eta$, contradicting the fact that $p\Vdash_\PP \dot{b}(\beta)=\alpha$. 

Second, suppose that $r\notin s$. 
If $r$ is non-terminal in $t$, 
then $r$ has exactly two successors 
%Then there are $\zeta,\eta<\kappa$ with $\zeta\neq \eta$ and 
$r^\smallfrown \langle \zeta\rangle$, $r^\smallfrown \langle \eta\rangle$ 
%are the unique successors of $r$ 
in $t$ with $\zeta, \eta\neq \alpha$. 
Then $p\Vdash_{\mathbb{P}} \dot{b}{\upharpoonright}\gamma\in t$, contradicting the fact that $p\Vdash_\PP \dot{b}(\beta)=\alpha$. 
If $r$ is terminal in $t$, 
let $\zeta, \eta$ be distinct ordinals below $\kappa$ with $\zeta, \eta\neq \alpha$. 
Let $u=t\cup\{r^\smallfrown \langle \zeta\rangle, r^\smallfrown \langle \eta\rangle\}$, $v=s$ 
and $q=(u,s)$. 
Then $q\Vdash_{\mathbb{P}} \dot{b}(\beta)\neq\alpha $, contradicting the fact that $p\Vdash_\PP \dot{b}(\beta)=\alpha$. 
%If there is a unique $\beta$ with $r^\smallfrown \beta\in t$, let $u=t$ and $v=s\cup\{r^\smallfrown \beta\}$. Then $(u,v)\leq (t,s)$ and $b_{u,v}=\beta$, contradicting the assumption that $b_{t,s}(\gamma)=\alpha\neq\beta$. 
%Then there are $\zeta,\eta<\kappa$ with $\zeta \neq \eta$ and $r^\smallfrown \langle \zeta \rangle, r^\smallfrown \langle \eta \rangle \in t$. 
%If there are $\beta\neq\gamma$ with $r^\smallfrown\beta\in t$ and $r^\smallfrown \gamma\in t$, then $(t,s)\Vdash \dot{b}(\gamma)\in \{\beta,\gamma\}$, contradicting the assumption that $\gamma\in dom(b_{t,s})$ and $b_{t,s}\upharpoonright (\gamma+1)\notin t$. 
%If there is no $\beta$ with $r^\smallfrown\beta\in t$, choose some $\beta\neq\alpha$. Let $u=t\cup\{r^\smallfrown\beta\}$ and $v=s\cup\{r^\smallfrown\beta\}$. Then $(u,v)\leq(t,s)$ and $b_{u,v}(\gamma)=\beta$, contradicting the assumption that $b_{t,s}(\gamma)=\alpha\neq\beta$. 
\end{proof} 

%\todo{replace $\DD$ with $\PP^*$ everywhere} 
Let $\mathbb{P}^*$ denote the set of conditions $p=(t,s)\in \mathbb{P}$ such that $\len(t)$ is a limit ordinal and $\len(\dot{b}_p)=\len(t)$. 
% is a limit. 
%of limit length $\gamma:=|(t,s)|=|t|$ and with $|b_{t,s}|=\gamma$. 
%Notice that $\gamma$ is the maximal possible height of $b_{t,s}$. 

\begin{lemma} \label{dense subset of P} 
$\mathbb{P}^*$ is dense in $\mathbb{P}$. 
\end{lemma} 

\begin{proof} Suppose that $p\in \mathbb{P}$ and let $p_0=p=(t_0,s_0)$. We choose successively for each $n\in\omega$ a condition $p_{n+1}=(t_{n+1},s_{n+1})$ that decides $\dot{b}{\upharpoonright}\len(t_n)$ with $p_{n+1}\leq p_n$ and $\len(t_{n})<\len(t_{n+1})$. 
Let $t=\bigcup_{n\in\omega} t_n$, $s=\bigcup_{n\in\omega} s_n$ and 
$q=(t,s)$. 
By the construction, $\len(q)$ is a limit and $\len(q)\leq \len(\dot{b}_q)$. 
%We \todo{check definition of l (+/-1) here} 
Moreover, we have $\len(b_q)\leq \len(q)$ by Lemma \ref{decided part of branch} and hence $q\in \PP^*$. 
%Suppose that $p=(t,s)$ has length $\gamma$. Then $\gamma$ is a limit ordinal and $p$ decides $\dot{b}\upharpoonright \gamma$. 
\end{proof} 

%Notice that every decreasing sequence $((t_{\alpha},s_{\alpha}))_{\alpha<\gamma}$ in $\mathbb{D}$ of length $\gamma<\kappa$ has a greatest lower bound $(t,s)\in \mathbb{D}$, where $t$ is the closure of $\bigcup_{\alpha<\gamma} t_{\alpha}$ in ${}^{<\kappa}2$ and $s=\bigcup_{\alpha<\gamma} s_{\alpha}$. 

We will expand $\PP$ to determine the quotient forcing in $V[G]$ for a branch 
%$b$ 
$\dot{b}^G\in [T_G]$. The precise statement is given in Lemma \ref{quotient forcing of a branch} below. 

% or more precisely, for the complete subforcing \todo{need?} of $\BB(\Add{\kappa}{1})$ induced by a name for this branch, 
%where $G$ is a $\PP$-generic filter over $V$. 
%\todo{check again with definition 1.17} More precisely, we will consider a $\PP$-name $\dot{b}$ for such a branch and define a forcing $\QQ$ with the following properties. The forcing $\QQ$ is equivalent to $\PP$, so that one obtains a corresponding $\QQ$-name $\dot{b}_{\QQ}$. Moreover, the forcing contains a complete subforcing $\QQ_0$ such that $\dot{b}_{\QQ}$ generates the same extension of $V$ as the restriction of the $\QQ$-generic filter to $\QQ_0$. The following lemmas will further show that the quotient forcing for $\QQ_0$ in $\QQ$ is equivalent to $\Add{\kappa}{1}$. 
%We then show that the name generates a complete subforcing $\QQ_0$ of $\QQ$ that 

Suppose that $\dot{b}$ is a $\PP$-name for a branch in $T_{\dot{G}}$, where $\dot{G}$ is a name for the $\PP$-generic filter, in the sense that this is forced by $\one_\PP$. 
Let 
$$\mathbb{Q}=\{(\dot{b}_p,q)\mid p\in\mathbb{P}^*\text{ and }(q=p\text{ or }q={1}_{\PP})\}$$ 
and for all $(u,p), (v,q)\in \mathbb{P}$, let $(u,p)\leq (v,q)$ if $v\subseteq u$ and $p\leq_{\mathbb{P}} q$. 
Moreover, let
$$\mathbb{Q}_0=\{(\dot{b}_p,\one_\PP)\mid p\in\mathbb{P}^*\}$$ 
$$\mathbb{Q}_1=\{(\dot{b}_p,p)\mid p\in\mathbb{P}^*\}.$$ 
%Suppose that $\mathbb{Q}=\mathbb{Q}_0\cup\mathbb{Q}_1$ is partially ordered by reverse inclusion in the first coordinate. 
Then $\QQ=\QQ_0\cup \QQ_1$, 
%\subseteq\{\one_{\QQ}\}$ 
%\{(\emptyset, \emptyset)\}$ 
$\mathbb{Q}_1$ is a dense subforcing of $\QQ$ and $\QQ_0\cap \QQ_1$ contains at most $\one_{\QQ}$. 
We further consider the map $e\colon \PP^*\rightarrow \QQ_1$, $e(p)=(\dot{b}_p,p)$. 
Since $e$  an isomorphism, $\PP^*$ is dense in $\PP$ and $\QQ_1$ is dense in $\QQ$, it follows that the forcings 
%$p\mapsto (b_{t,s},(t,s))$ is an isomorphism of $\mathbb{D}$ onto $\mathbb{Q}_1$ and 
%$g\colon \Add{\kappa}{1}\rightarrow \QQ_0$, $g(p)=(p,1_{\mathbb{P}})$ are isomorphisms. 
%$r\mapsto (r,\emptyset)$ is an isomorphism between $Add(\kappa,1)$ and $\mathbb{Q}_0$. 
%\end{enumerate-(1)} 
%\end{lemma} 
%\begin{proof} 
%The proof is straightforward.  
%\end{proof} 
$\PP$, $\QQ$ are sub-equivalent. 
%The forcings $\mathbb{P}$ and $\mathbb{Q}$ are forcing equivalent by Lemma \ref{dense subset of P} and \todo{citation isn't working} Lemma \ref{equivalence of P and Q}. 
%Since $\mathbb{D}$ is dense in $\mathbb{P}$, we immediately obtain the following. 

\begin{lemma} \label{compatible conditions for perfect set}
The map $\pi=\pi_{\QQ,\QQ_0}\colon \QQ\rightarrow \QQ_0$, $\pi(\dot{b}_p,r)=(\dot{b}_p,\one_\PP)$ is a projection. 
\end{lemma} 
\begin{proof} 
By the definition, $\pi$ is a homomorphism with respect to $\leq$ and it is surjective onto $\QQ_0$. 
%it it sufficient to prove that for all $q\in \QQ$ and all $p\leq \pi(q)$, there is a condition $\bar{q}\leq q$ with $\pi(\bar{q})\leq p$. 

To prove the remaining requirement for projections, first suppose that $u=(\dot{b}_p,p)\in \mathbb{Q}_1$ and $v=(\dot{b}_q,1_{\mathbb{P}})\in \mathbb{Q}_0$ are conditions with $v\leq \pi(u)$. In particular, $\dot{b}_p\subseteq \dot{b}_q$ and hence $\len(p)\leq\len(q)$. It is sufficient to show that $u,v$ are compatible in $\mathbb{Q}$, since for any extension $w\leq u,v$, we have $\pi(w)\leq v$ by the definition of $\pi$ and since $v\in \QQ_0$. 

% of the form $q=(c,\emptyset)$ with $b_{t,s}\subseteq c$ are compatible in $\mathbb{Q}$. 
To see that $u$, $v$ are compatible, suppose that $p=(t,s)$. 
Since $p\in \mathbb{P}^*$, $\len(p)$ is a limit and $\dot{b}_p$ is cofinal in $t$ by the definition of $\PP^*$. 
Let $$\bar{t}=t\cup \{\dot{b}_q\upharpoonright\alpha \mid  \len(\dot{b}_p)\leq\alpha<\len(\dot{b}_q)\}$$ 
$$\bar{s}=s\cup \{\dot{b}_q\upharpoonright \alpha \mid \len(\dot{b}_p)\leq\alpha<\len(\dot{b}_q)\}$$ 
and $\bar{p}=(\bar{t},\bar{s})$. 
Then $\bar{p}\in\PP$ and $\bar{p}\leq p$. Moreover, it follows from Lemma \ref{decided part of branch} that 
%$w\Vdash_{\mathbb{P}} \dot{b} 
$\dot{b}_q\subseteq \dot{b}_{\bar{p}}$. 

We can choose a condition $r\leq \bar{p}$ with $r\in \mathbb{P}^*$, since $\PP^*$ is dense in $\PP$, and let $w=(\dot{b}_r,r)\in \mathbb{Q}_1$. 
Since $u=(\dot{b}_p,p)$ and $r\leq p$, we have 
%$\dot{b}_q\subseteq \dot{b}_{\bar{p}}\subseteq b_r$, we have 
$w\leq u$. 
Since $v=(\dot{b}_q,1_{\mathbb{P}})$ and $\dot{b}_q\subseteq \dot{b}_{\bar{p}}\subseteq \dot{b}_r$, we have $w\leq v$, and in particular, $u$, $v$ are compatible. 
%Since $v\in \QQ_0$, it follows from $w\leq v$ that $\pi(w)\leq v$ by the definition of $\pi$, proving the claim. 

Second, suppose that $u=(\dot{b}_p,\one_{\PP})\in \mathbb{Q}_0$ and $v$ is as above. Since $(\dot{b}_p,\one_{\PP})\leq u$, the required statement follows from the property of $(\dot{b}_p,p)$ that we just proved. 
%Hence $u$, $v$ are compatible, proving the claim. 
%Since $\dot{b}_q\subseteq \dot{b}_{\bar{p}}\subseteq \dot{b}_r$, we have $w\leq v$. 
%Then $b_p\subseteq b_q\subseteq b_r$. 
%Since $(t,s)\in\mathbb{D}$, $|b_{t,s}|=\gamma$ for some limit ordinal $\gamma$ and $|c|=\delta$ for some limit ordinal $\delta\geq\gamma$. 
%Let $u=t\cup \{c\upharpoonright \alpha\mid\alpha<\delta\}$ and $v=s\cup\{c\upharpoonright(\alpha+1)\mid \alpha<\delta\}$. Then $b_{u,v}=c$ and hence $(u,v)\in\mathbb{D}$. Then $r=(c,(u,v))\leq p,q$ (and $r\in\mathbb{Q}_1$). 
\end{proof} 

\begin{lemma} \label{complete subforcing for perfect set} 
$\mathbb{Q}_0$ is a complete subforcing of $\mathbb{Q}$. 
\end{lemma} 
\begin{proof} 
%It follows from the definition of $\mathbb{Q}$ and $\mathbb{Q}_0$ that for all $p,q\in \mathbb{Q}_0$, $p\leq_{\mathbb{Q}_0} q$ if and only if $p\leq_{\mathbb{Q}} q$, and $p\perp_{\mathbb{Q}_0} q$ if and only if $p\perp_{\mathbb{Q}}q$. 
%The order and compatibility in $\mathbb{Q}_0$ is equal to the order and compatbility in $\mathbb{Q}$ 
It is sufficient to show that every maximal antichain $A$ in $\mathbb{Q}_0$ is maximal in $\mathbb{Q}$. 
Let 
$$D_0=\{p\in \mathbb{Q}_0\mid \exists q\in A\ p\leq q\}$$ 
$$D=\{p\in \mathbb{Q}\mid \exists q\in A\ p\leq q\}.$$ 

It is sufficient to show that 
%\begin{claim*} 
$D$ is dense in $\mathbb{Q}$, since this implies that $A$ is a maximal antichain in $\mathbb{Q}$. 
%\end{claim*} 
%\begin{proof} 
To see that $D$ is dense, suppose that $u\in \mathbb{Q}$. If $u\in \mathbb{Q}_0$, then there is a condition $v\leq u$ in $D_0\subseteq D$, since $D_0$ is dense in $\QQ_0$ by the assumption that $A$ is maximal in $\QQ_0$. 
Now suppose that $u=(\dot{b}_p,p)\in \mathbb{Q}_1$. 
Since $D_0$ is dense in $\mathbb{Q}_0$, there is some $v=(\dot{b}_q,1_{\mathbb{P}})\in D_0$ with $\dot{b}_p\subseteq \dot{b}_q$. 
Since $v\leq \pi(u)$ and $\pi$ is a projection by Lemma \ref{compatible conditions for perfect set}, there is some $w\leq u$ with $\pi(w)\leq v$. 
%Then 
%There is a common extension 
%$w\leq u,v$. 
% by Lemma \ref{compatible conditions for perfect set}. 
Then $w\leq \pi(w)\leq v\in D_0$ and hence $w\in D$ by the definition of $D$, proving that $D$ is dense in $\QQ$. 
%\end{proof} 
%Since $D$ is dense in $\QQ$, 
%To see that every dense set $D\subseteq \mathbb{Q}_0$ is predense in $\mathbb{Q}$, suppose that $p=(b_{t,s},(t,s))\in\mathbb{Q}_1$. Choose $q=(c,\emptyset)\leq_{\mathbb{Q}}(b_{t,s},\emptyset)$ with $q\in D$. Then $p,q$ are compatible by the previous argument. Hence $D$ is predense in $\mathbb{Q}$. 
\end{proof} 

Let $e\colon \PP^*\rightarrow \QQ_1$, $e(p)=(\dot{b}_p,p)$ be the isomorphism between $\PP^*$ and $\QQ_1$ that was given after the definition of $\QQ$ above. 
%the restriction of the $\QQ$-generic filter to $\QQ_0$ and $\dot{b}_{\QQ}^H$ generate the same extension of $V$. 
If $G$ is a $\mathbb{P}$-generic filter over $V$, 
%We work with $\mathbb{Q}$ instead of $\mathbb{P}$ as follows. 
%Let 
then the upwards closure 
$$H=\{q\in \mathbb{Q}\mid \exists p\in G\ e(p)\leq q\}$$ 
of $e[G]$ in $\QQ$ 
%Then $H$ 
is a $\QQ$-generic filter over $V$. In the following, we will write $T_H=T_G$, where $T_G$ is the perfect tree adjoined by $G$ that is given after the definition of $\PP$ above. 

%\todo[inline]{define the Boolean completion in the intro?}

%\todo{do we need this again later?} 
%We will use the following notation for quotient forcings given by names. 
%If $\RR$ is a forcing, let $\BB(\RR)$ denote the Boolean completion of the \todo{need this? write further above?} separative quotient of $\RR$, which will in fact be the Boolean completion in all cases below, since we will only apply this to separative forcings. 
%but since \todo{check} we will only apply this to separative forcings, it will simply be the Boolean completion. 
%Moreover, if $\RR$ is a separative forcing and $\sigma$ is an $\RR$-name, then we will assume that $\RR$ is a subset of $

Since it is convenient to work with complete Boolean algebras, we will now check that $\PP$ is separative. 

\begin{lemma} 
$\PP$ is a separative partial order. 
\end{lemma} 
\begin{proof} 
It is easy to see that $\PP$ is a partial order. 
To show that $\PP$ is separative, suppose that $(t,s)$, $(v,u)$ are conditions in $\PP$ with $(t,s)\not\leq (v,u)$. 

We first assume that $v\subseteq t$. Then $s\cap v\neq u$. We claim that $(t,s)$, $(v,u)$ are already incompatible. Otherwise there is a common extension $(y,x)$, so that $x\cap t=s$ and $y\cap v=u$. However, this implies that $s\cap v=(x\cap t)\cap v=x\cap v=u$, contradicting the fact that $s\cap v\neq u$. 

We now assume that $v\not\subseteq t$ and choose some $w\in v\setminus t$. 
We can assume that $(t,s)$, $(v,u)$ are compatible, so that $(t\cup v, s\cup u)$ is a condition. 
We define $y\subseteq t\cup v$ by removing all nodes strictly above $w$. 
To define $x$, we first let $\bar{x}=(y\setminus \{w\})\cap (s\cup u)$. 
Let $x=\bar{x}$ if $w\in u$ and  $x=\bar{x}\cup\{w\}$ otherwise. 
The choice of $x$ implies that $(y,x)$, $(v,u)$ are incompatible, since $w\in x\Leftrightarrow w\notin u$. 
This is sufficient, since $(y,x)\leq (t,s)$. 
\end{proof} 

Moreover, it follows from the previous lemma and Lemma \ref{decided part of branch} that $\QQ$ is also a separative partial order. 

If $\RR$ is a complete Boolean algebra and $\sigma$ is an $\RR$-name for an element of $\kk$, as a special case of the notation given in Definition \ref{definition: generated Boolean subalgebra}, we will write $\BB(\sigma)=\BB^{\RR}(\sigma)$ for the complete Boolean subalgebra of $\RR$ that is generated by the Boolean values $\llbracket \sigma(\alpha)=\beta\rrbracket_{\RR}$ for ordinals $\alpha,\beta<\kappa$. 

%\todo[inline]{define $\BB(\PP)$ in the intro? we need $\RR\subseteq \BB(\RR)$, since $\dot{c}$ might be an $\RR$-name that we understand as a $\BB(\RR)$-name } 

We will use the following terminology for quotient forcings relative to elements of $\kk$ in a generic extension. 

\begin{definition} 
Suppose that $\RR$ is a separative forcing, $\SSS$ is any other forcing, 
%$\dot{c}$ is an $\RR$-name for an element of $\kk$ and
$G$ is $\RR$-generic over $V$ and $c\in V[G]$ is a set that can be coded as a subset of a ground model set in an absolute way. 
%$c\in (\kk)^{V[G]}$. 
% and $\SSS$ is a forcing. 
We say that \emph{$c$ has $\SSS$ as a quotient in $V[G]$} if there is a $\RR$-name $\dot{c}$ with $\dot{c}^G=c$ such that for the $\BB(\dot{c})$-generic filter $G_0=G\cap \BB(\dot{c})$, the quotient forcing $[\BB(\RR)/\BB(\dot{c})]^{G_0}$ is equivalent to $\SSS$ in $V[G_0]$. 
\end{definition} 
%notation for quotient forcings given by names. 
%If $\RR$ is a forcing, let $\BB(\RR)$ denote the Boolean completion of the \todo{need this? write further above?} separative quotient of $\RR$, which will in fact be the Boolean completion in all cases below, since we will only apply this to separative forcings. 
%but since \todo{check} we will only apply this to separative forcings, it will simply be the Boolean completion. 
%If $\BB$ is a complete Boolean algebra and $\dot{c}$ is an $\BB$-name for an element of $\kk$, we will write $\BB(\dot{c})=\BB^{\BB}(\dot{c})$ for the complete Boolean subalgebra of $\BB$ that is generated by the Boolean values $\llbracket \dot{c}(\alpha)=\beta\rrbracket_{\BB}$ for ordinals $\alpha,\beta<\kappa$. 

\begin{lemma} \label{quotient forcing of a branch} 
$\one_{\BB(\dot{b})}$ forces that the quotient forcing $\BB(\PP)/\BB(\dot{b})$ 
%for $\BB(\dot{b})$ in $\BB(\PP)$ 
is sub-equivalent to $\Add{\kappa}{1}$. 
\end{lemma} 
\begin{proof} 
%\todo{DELETE 2x!!! check with definition in section 1.2} 
Let $\dot{b}_{\QQ}$ denote the $\QQ$-name induced by the $\PP$-name $\dot{b}$ via the sub-isomorphism $e\colon \PP^*\rightarrow \QQ$ defined above. Since $e$ induces an isomorphism $\BB(\PP)\cong\BB(\QQ)$ on the Boolean completions, 
% as witnessed by the isomorphism that is induced by $e$, 
it is sufficient to prove the claim for $\QQ$, $\dot{b}_{\QQ}$ instead of $\PP$, $\dot{b}$. 
Moreover, it follows from the definition of $\QQ_0$ that $\BB(\dot{b}_\QQ)$ is equal to the complete subalgebra of $\BB(\QQ)$ generated by $\QQ_0$. 
%$V[\dot{b}_{\QQ}^H]=V[H\cap \QQ_0]$ for any $\QQ$-generic filter $H$ over $V$. 
%$\one_\QQ$ forces that the quotient forcing $\QQ/\QQ_0$ for $\QQ_0$ in $\QQ$ forces that 
Since $\QQ_0$ is a complete subforcing of $\QQ$ by Lemma \ref{complete subforcing for perfect set}, it is therefore sufficient to prove that 
$\QQ_0$ forces that the quotient forcing $\QQ/\QQ_0$ 
%for $\BB(\dot{b}_\QQ)$ in $\BB(\QQ)$ 
is equivalent to $\Add{\kappa}{1}$. 

It follows from Lemma \ref{distinct branches are mutually generic} that $\QQ$ forces that there is an $\Add{\kappa}{1}$-generic filter over $V[\dot{b}_{\QQ}]$ in $V[\dot{G}]$, where $\dot{G}$ is a name for the $\QQ$-generic filter, and therefore $\QQ$ forces that the quotient forcing $\QQ/\QQ_0$ is non-atomic. 

We have that $\pi\colon \QQ\rightarrow \QQ_0$ is a projection (with  $\pi{\upharpoonright}\QQ_0=\mathrm{id}_{\QQ_0}$) by Lemma \ref{compatible conditions for perfect set} and $\QQ_0$ is a complete subforcing of $\QQ$ by Lemma \ref{complete subforcing for perfect set}. 
% (and moreover $\pi{\upharpoonright}\QQ_0=\mathrm{id}_{\QQ_0}$). 
Since 
%\todo{state as lemma before?} 
moreover $\pi(q)\geq q$ for all $q\in \QQ$, it is easy to check that $\QQ_0$ forces that the quotient forcing $\QQ/\QQ_0$ given in Definition \ref{pull back names} and the quotient forcing $(\QQ/\QQ_0)^{\pi}$ with respect to $\pi$ given in Definition \ref{definition: quotient forcing} are equal. % by \cite[Remark 5.2]{MR2768691}
Hence we can consider $(\QQ/\QQ_0)^{\pi}$ instead of $\QQ/\QQ_0$. 

%Suppose that $\dot{b}^H=c$. 
%$H$ is $\mathbb{Q}$-generic over $V$ and 
%Suppose that $p_0\in \mathbb{Q}$ and $p_0\Vdash \dot{c}\in [T_{\dot{H}}]$, where $\dot{H}$ is the canonical name for the $\mathbb{Q}$-generic filter. 
%and $c$ is a branch in $T_H$, t
%Then the quotient forcing in $V[H]$ for 
%$c$ given by
%the complete subalgebra \todo{need?} of $\BB(\Add{\kappa}{1})$ induced by $\dot{b}$ is equivalent to $\Add{\kappa}{1}$. 
%\end{lemma} 
%\begin{proof} 
%By Lemma \ref{compatible conditions for perfect set} and Lemma \ref{complete subforcing for perfect set} and by 
Now suppose that $G_0$ is $\QQ_0$-generic over $V$ and $b=\dot{b}^{G_0}$. 
By the definition of the quotient forcing with respect to $\pi$ in Definition \ref{definition: quotient forcing}, we have 
%for $c$ in $\QQ$ 
$$[(\QQ/\QQ_0)^{\pi}]^{G_0}=\{(\dot{b}_p,q)\in\QQ\mid \pi(\dot{b}_p,q)\in G_0\}=\{(\dot{b}_p,q)\in \QQ\mid \dot{b}_p\subseteq b\}.$$ 
It follows from the definitions of $\PP^*$ and $\QQ$ that the last set in the equation is a ${<}\kappa$-closed subset of $\QQ$. 
%By Definition \ref{definition: quotient forcing}, 
%Therefore, the quotient forcing is equal to the ${<}\kappa$-closed subset $\{(b_p,p)\in \QQ\mid b_p\subseteq c\}$ of $\QQ$. 
%This forcing is $<\kappa$-closed and i
%there is an $\Add{\kappa}{1}$-generic filter over $V[c]$ in $V[G]$ by Lemma \ref{distinct branches are mutually generic} and hence the quotient forcing is non-atomic. 
Since we already argued that the quotient forcing is non-atomic, it is sub-equivalent to $\Add{\kappa}{1}$ by Lemma \ref{forcing equivalent to Add(kappa,1)}. 
\end{proof} 

The next result shows that the statement of the previous lemma also holds for names for sequences of length ${<}\kappa$ 
%that consist 
of branches in $T_G$. 
%as the next result shows. 
%In the statement of the next lemma, let 
For the statement of the result, we assume that $\gamma<\kappa$, $\dot{G}$ is a $\PP$-name for the $\PP$-generic filter and 
%that is given after Definition \ref{definition of the forcing adding a perfect tree}. 
$\sigma$ is a $\PP$-name for a sequence of length $\gamma$ 
%that consists 
of distinct branches in $T_{\dot{G}}$, in the sense that this is forced by $\one_\PP$. 

\begin{lemma} \label{quotient forcing for a sequence of branches} 
$\one_{\BB(\sigma)}$ forces that the quotient forcing $\BB(\PP)/\BB(\sigma)$ 
%for $\BB(\dot{b})$ in $\BB(\PP)$ 
is sub-equivalent to $\Add{\kappa}{1}$. 
\end{lemma} 
\begin{proof} 
%\todo{is this ok?} 
The proof is analogous to the proof of Lemma \ref{quotient forcing of a branch}, but instead of working with a name $\dot{b}$ for a branch in $T_{\dot{G}}$, we work with the name $\sigma$ for a sequence of branches in $T_{\dot{G}}$. As in the definitions of $\QQ$, $\QQ_0$ before Lemma \ref{compatible conditions for perfect set}, we can define variants of these forcings with respect to $\sigma$ instead of $\dot{b}$ and thus obtain the required properties as in the proofs of Lemma \ref{compatible conditions for perfect set} and Lemma \ref{complete subforcing for perfect set}. 
\end{proof} 

The previous two lemmas imply that $\dot{b}^G$ and $\sigma^G$ have $\Add{\kappa}{1}$ as a quotient in $V[G]$ for every $\PP$-generic filter $G$ over $V$. 

\begin{lemma} \label{perfect set with good quotients} 
Suppose that $\lambda$ is an uncountable regular cardinal, $\mu>\lambda$ is inaccessible and 
$G$ is $\Add{\lambda}{1}$-generic over $V$. 
Then in $V[G]$, there is a perfect subtree $T$ of $\ltll$ such that for every $\gamma<\lambda$, every sequence $\langle x_i\mid i<\gamma\rangle$ of distinct branches of $T$ is $\Add{\lambda}{\gamma}$-generic over $V$ and has $\Add{\lambda}{1}$ as a quotient in $V[G]$. 
%each $x\in C$ is $\Add{\kappa}{1}$-generic over $V$ and has $\Add{\kappa}{1}$ as a quotient in $V[G]$. 
%the quotient forcing for $x$ is equivalent to $\Add{\kappa}{1}$. 
\end{lemma} 
\begin{proof} 
Since $\PP$ is sub-equivalent to $\Add{\lambda}{1}$, there is a $\PP$-generic filter $H$ over $V$ with $V[G]=V[H]$. 
%We work in 
Let $C=[T_H]^{V[H]}$, where $T_H$ is the tree given after Definition \ref{definition of the forcing adding a perfect tree}. 

We first assume that $\gamma=1$. 
By Lemma \ref{distinct branches are mutually generic}, every $x\in C$ is $\Add{\lambda}{1}$-generic over $V$ and by Lemma \ref{quotient forcing of a branch}, every $x\in C$ has $\Add{\lambda}{1}$ as a quotient in $V[G]$.  

The proof is analogous for arbitrary $\gamma<\lambda$. By Lemma \ref{many distinct branches are mutually generic}, any sequence $\vec{x}=\langle x_i\mid i<\gamma\rangle$ of distinct elements of $C$ is $\Add{\lambda}{\gamma}$-generic over $V$ 
and by 
%it can be shown as in the proof of 
Lemma \ref{quotient forcing for a sequence of branches}, 
$\vec{x}$ has  $\Add{\lambda}{1}$ as a quotient in $V[G]$. 
% shows that there is a name $\sigma$ for $\vec{x}$ such that $\one_{\BB(\sigma)}$ forces that $\BB(\Add{\kappa}{1})/\BB(\sigma)$ is equivalent to $\Add{\kappa}{1}$. 
\end{proof} 
%\begin{question} 
%Does \todo{to do} the same proof show that $Col(\kappa,\gamma)$ ($Col(\kappa,<\lambda)$) adds a perfect set of factoring $Col(\kappa,\gamma)$-generics ($Col(\kappa,<\lambda)$-generics)? 
%\end{question} 

%\todo[inline]{STATE A VERSION FOR ${<}\kappa$ many branches?} 

In the next proof, we will  use the following notation $\Col{\lambda}{X}$ for subforcings of the Levy collapse $\Col{\lambda}{{<}\mu}$. 
Suppose that $\lambda<\mu$ are cardinals and $X\subseteq\mu$ is not an ordinal (to avoid a conflict with the notation for the standard collapse). We then write 
$$ \Col{\lambda}{X}=\{p\in \Col{\lambda}{{<}\mu}\mid \dom{p}\subseteq X\times \lambda\}.$$ 
Let further 
$G_X=G\cap \Col{\lambda}{X}$ and $G_\gamma=G\cap\Col{\lambda}{{<}\mu}$ for any $\Col{\lambda}{{<}\mu}$-generic filter $G$ over $V$ and any $\gamma<\mu$. 
%Let $Col(\kappa,S)=\{p\in Col(\kappa,<\lambda)\mid dom(p)\subseteq S\times \kappa\}$ and $G_S=G\cap Col(\kappa,S)$ for $S\subseteq\lambda$. Since $Col(\kappa,\mu)$ absorbs all $<\kappa$-closed forcings of size $\leq\mu$ if $\mu^{<\kappa}=\mu$ by [Fuchs: Maximality principles:..., Lemma 2.2], $Col(\kappa,(\mu,\lambda))$ is forcing equivalent to $Col(\kappa,<\lambda)$ for all $\mu<\lambda$. 

The notation $\Col{\lambda}{X}$ will be used for intervals $X$, for which we use the standard notation 
$$(\alpha,\gamma)=\{\beta\in \Ord\mid \alpha<\beta<\gamma\}$$ 
$$[\alpha,\gamma)=\{\beta\in \Ord\mid \alpha\leq\beta<\gamma\}.$$ 

Moreover, we will use the following consequence of Lemma \ref{forcing equivalent to Add(kappa,1)} in the next proof. 
Suppose that $\lambda$ is regular and $\mu>\lambda$ is inaccessible. 
%If $\mu$ is a regular cardinal with $\kappa\leq \mu<\lambda$, $\mu^{<\kappa}=\mu$,\todo{we don't need $\mu$} 
If $\RR$ is a separative ${<}\lambda$-closed forcing of size ${<}\mu$ and $\gamma<\mu$ is an ordinal, then $\RR\times \Col{\lambda}{{<}\mu}$ and $\Col{\lambda}{[\gamma,\mu)}$ are sub-equivalent. 

\begin{theorem} \label{perfect subsets of definable sets} 
Suppose that $\lambda$ is an uncountable regular cardinal, $\mu>\lambda$ is inaccessible and $G$ is $\Col{\lambda}{{<}\mu}$-generic over $V$. 
Then in $V[G]$, every subset of $\ltl$ that is definable from an element of ${}^{\lambda}V$ has the perfect set property. 
% in every extension by $Col(\kappa,<\lambda)$ for $\lambda>\kappa$ inaccessible. 
\end{theorem} 
\begin{proof} 
%For \todo{check for $\mathsf{PSP}$ for sets definable from parameters in $Ord^{\kappa}$} $\kappa=\omega$ this was proved by Solovay. Suppose that $\kappa$ is uncountable. 
%and $G$ is $Col(\kappa,<\lambda)$-generic over $V$. 
Suppose that $\varphi(x,y)$ is a formula with two free variables and $z\in \Ord^{\lambda}$. Using the set $A_{\varphi,z}^\lambda$ given in Definition \ref{definition of Aphi}, let 
$$(A_{\varphi,z}^\lambda)^{V[G]}=\{x\in (\ltl)^{V[G]}\mid V[G]\vDash \varphi(x,z)\}.$$ 
Moreover, for any 
%transitive 
subclass $M$ of $V[G]$, let 
$$A^M=(A_{\varphi,z}^\lambda)^{V[G]}\cap M.$$ 
To prove the perfect set property for $A^{V[G]}$ in $V[G]$, suppose that in $V[G]$, $A^{V[G]}$ has size $\lambda^+$. 
%$|A|^{V[G]}=(\kappa^+)^{V[G]}$. 
We will show that $A$ has a perfect subset in $V[G]$. 
%We will write $G_\gamma=G\cap \Col{\kappa}{{<}\gamma}$ for any $\gamma<\lambda$. 
%$\gammaG\cap $ 
% in $V[G]$. 
%We will use $A$ as a term for the set defined by $\phi$ from $\alpha$ and $y$ in $V[G]$ and in intermediate extensions which contain $\alpha$ and $y$. 

Since $\Col{\lambda}{{<}\mu}$ has the $\mu$-cc, there is some $\gamma<\mu$ with $z\in V[G_\gamma]$. 
Since $A^{V[G]}$ has size $\lambda^+$ in $V[G]$, there is some ordinal $\nu$ with $\gamma<\nu<\mu$ and $A^{V[G_\gamma]}\neq A^{V[G_{\nu}]}$. 
%are cofinally many ordinals $\delta$ with $\gamma<\delta<\mu$ with $A^{V[G_\gamma]}\neq A^{V[G_{\delta}]}$. 
% since otherwise the size of $A^{V[G]}$ in $V[G]$ is at most $\kappa$ and this contradicts the assumption. 
Moreover, it follows from the definition of $A^M$ that this inequality 
%$A^{V[G_\gamma]}\neq A^{V[G_{\delta}]}$ 
remains true when $\nu$ 
%\geq \gamma$ 
increases. Let $\nu$ be a cardinal with $\gamma<\nu<\mu$, $\nu^{<\lambda}=\nu$ and $A^{V[G_\gamma]}\neq A^{V[G_{\nu}]}$. 
%$$A^{V[G_\gamma]}\neq A^{V[G_{\mu+1}]}.$$ 
%Such ordinals exist, since otherwise the size of $A^{V[G]}$ in $V[G]$ can be at most $\kappa$, contradicting the assumption. 
%Moreover, we can assume that $\mu$ is a regular cardinal with $\mu^{<\kappa}=\mu$ by increasing $\mu$. 
%Since $|A|^{V[G]}=(\kappa^+)^{V[G]}$, there is a regular cardinal $\mu$ with $\gamma<\mu<\lambda$ and $\mu^{<\kappa}=\mu$ such that $A^{V[G_\gamma]}\neq A^{V[G_\mu]}$. 

%\todo{is this used?} Since the forcing $\Col{\kappa}{{<}\mu+1}$ 
%is separative, ${<}\kappa$-closed, has size $\mu$ and collapses $\mu$ to have size $\kappa$,  
%it is equivalent to $\Col{\kappa}{\mu}$ by Lemma \ref{forcing equivalent to Add(kappa,1)}. 

The forcing $\Col{\lambda}{[\nu+1,\mu)}$ is sub-equivalent to $\Add{\lambda}{1}\times \Col{\lambda}{{<}\mu}$ by the remarks before the statement of this theorem. 
Hence there is an $\Add{\lambda}{1}\times \Col{\lambda}{{<}\mu}$-generic filter $g\times h$ over $V[G_{\nu+1}]$ with $V[G]=V[G_{\nu+1}\times g\times h]$. 

\begin{claim*} 
$A^{V[G_{\nu+1}]}\neq A^{V[G_{\nu+1}\times g]}$. 
\end{claim*} 
\begin{proof} 
We will prove the claim by writing the extension $V[G]$ with the generic filters added in a different order. 
% and we will use this to prove the claim by factoring the extension $V[G]$ in the following different ways. 
For the original generic filter $G$, we have 
$$ V[G]=V[G_\gamma\times G_{[\gamma,\nu+1)}\times G_{(\nu,\mu)}],$$ 
but we can also write $V[G]$ as 
$$V[G]= V[G_\gamma\times G_{[\gamma,\nu+1)}\times g\times h]$$ 
by the choice of $g$, $h$ above. 

Since $\nu^{<\lambda}=\nu$ and $\nu$ has size $\lambda$ in $V[G_{\nu+1}]$, $\Col{\lambda}{[\gamma,\nu+1)}$ is a non-atomic ${<}\lambda$-closed forcing of size $\lambda$. Hence it is sub-equivalent to $\Add{\lambda}{1}$ in $V[G_{\nu+1}]$ 
%in $V[G_{\mu+1}]$ 
by Lemma \ref{forcing equivalent to Add(kappa,1)}.  
%and non-atomic, and hence it is equivalent to $\Add{\kappa}{1}$. 
%This allows us to work with a $\Col{\kappa}{[\gamma,\mu+1)}$-generic filter instead of an $\Add{\kappa}{1}$-generic filter. 
% since it is $<\kappa$-closed and has size $\kappa$ in $V[G_{\mu+1}]$. 
It follows that there is a $\Col{\lambda}{[\gamma,\nu+1)}$-generic filter $k$ over $V[G_{\nu+1}]$ with $$V[G_{\nu+1}\times g]=V[G_{\nu+1}\times k].$$  
Hence we can write $V[G]$ as 
$$V[G]= V[G_\gamma\times G_{[\gamma,\nu+1)}\times k\times h]$$ 
by replacing $g$ with $k$ in the factorization above. By changing the order, we trivially obtain 
$$V[G]= V[G_\gamma\times k \times G_{[\gamma,\nu+1)}\times h]$$ 

We have $A^{V[G_\gamma]}\neq A^{V[G_{\nu+1}]}$ by the choice of $\nu$. By the last factorization of $V[G]$, this implies that 
$$A^{V[G_{\gamma}]}
%neq A_{\varphi,y}^{V[G]}\cap V[G_\gamma\times k]=
\neq A^{V[G_\gamma\times k]}$$
by homogeneity of the forcings. 
%Since $\Coll{\kappa}{[\gamma,\mu+1)}$ and $\Coll{\kappa}{(\mu,\lambda)}$ are homogeneous, this remains true when $G_{[\gamma,\mu+1)}$ is replaced with $k$ and $G_{(\mu,\lambda)}$ is replaced with $G_{[\gamma,\mu+1)}\times h$. 
%any other $\Coll{\kappa}{(\gamma,\lambda)}$-generic filter. 
%Hence $A^{V[G_{\gamma}]}\neq A_{\varphi}^{V[G]}\cap V[G_\gamma\times k]=A^{V[G_\gamma\times k]}$. 
Hence we can find some $x\in A^{V[G_\gamma\times k]}\setminus A^{V[G_{\gamma}]}=A^{V[G_\gamma\times k]\setminus V[G_{\gamma}]}$. 
In particular, $x\notin V[G_{\gamma}]$. 
%A^{V[G_\gamma\times k]}=
%collapse forcings are homogeneous, this holds for any $Col(\kappa,(\gamma,\mu+1))$-generic extension of $V[G_{\gamma}]$. 
%Hence $A^{V[G_\gamma]}\neq A^{V[G_{\gamma}\times h]}$. 
Since the filters $G_{[\gamma,\nu+1)}$ and $k$ are mutually generic over $V[G_{\gamma}]$ by the choice of $k$, we have $V[G_{\nu+1}]\cap V[G_{\gamma}\times k]= V[G_{\gamma}]$. 
However, this implies that $x$ cannot be in $V[G_{\nu+1}]$, since it is not in $V[G_{\gamma}]$. 
Since we also have  
$$x\in V[G_\gamma\times k]\subseteq V[G_{\nu+1}\times k]=V[G_{\nu+1}\times g],$$ 
it now follows that $x\in V[G_{\nu+1}\times g]\setminus V[G_{\nu+1}]$ and thus $x\in A^{V[G_{\nu+1}\times g]}\setminus A^{V[G_{\nu+1}]}$, proving the claim. 
%$$x\in A^{V[G_\gamma\times k]}\setminus A^{V[G_{\mu+1}]}\subseteq A^{V[G_{\mu+1}\times k]}\setminus A^{V[G_{\mu+1}]} = A^{V[G_{\mu+1}\times g]}\setminus A^{V[G_{\mu+1}]}.$$  
%where $[\gamma,\delta)=\{\alpha\in Ord\mid \gamma\leq\alpha<\delta\}$. 
%The forcing $\Add{\kappa}{1}\times \Coll{\kappa}{(\gamma,\mu+1)}$ is equivalent to $\Coll{\kappa}{(\gamma,\mu+1)}$ and 
\end{proof} 

We have 
$$V[G]= V[G_{\nu+1}\times g\times h]$$ 
by the choice of $g$, $h$ above. 
We now choose an $\Add{\lambda}{1}$-name $\sigma$ witnessing the previous claim. More precisely, $\sigma$ is an $\Add{\lambda}{1}$-name in $V[G_{\nu+1}]$ for a new element of $\ltl$ such that $\one_{\Add{\lambda}{1}}$ forces that $\sigma\in A_{\varphi,y}$ in every further $\Col{\lambda}{{<}\mu}$-generic extension. Such a name exists by the maximality principle applied to $\Add{\lambda}{1}$. 
%Suppose that $\sigma$ is an $\Add{\kappa}{1}$-name for a new element of $A_{\varphi}^{V[G]}$ over $V[G_{\mu+1}]$. 
%Since $\Add{\kappa}{1}$ is homogeneous, we can choose $\sigma$ so that is forced by $\one_{\Add{\kappa}{1}}$ that $\sigma$ is an element of $A_\varphi$ in all further $\Col{\kappa}{<\lambda}$-generic extensions, by the maximality principle. 

Since the forcing $\PP$ given in Definition \ref{definition of the forcing adding a perfect tree} is sub-equivalent to $\Add{\lambda}{1}$, we can replace the $\Add{\lambda}{1}$-generic filter $g$ with a $\PP$-generic filter. 
Since the definition of $\PP$ is absolute between models with the same $V_\lambda$, the definition of $\PP$ yields the same forcing in $V$ and $V[G_{\nu+1}\times h]$. 
Let $g_\PP$ be a $\PP$-generic filter over $V[G_{\nu+1}\times h]$ with 
$$V[G_{\nu+1}\times g_\PP\times h]=V[G_{\nu+1}\times g\times h].$$ 
%Suppose that $\bar{g}\times \bar{h}$ is $\mathbb{P}\times \Col{\kappa}{<\lambda}$-generic over $V[G_{\mu+1}]$ with $V[G]=V[G_{\mu+1}\times \bar{g}\times\bar{h}]$. 
%\todo{change} The definition of $\PP$ is absolute between $V[G_{\mu+1}]$ and $V[G_{\mu+1}\times h]$. 

\begin{claim*} 
In $V[G]$, the set $[T_{g_\PP}]$ is a perfect subset of $A^{V[G]}$. 
\end{claim*} 
\begin{proof} 
Since $T_{g_\PP}$ is a perfect tree and therefore $[T_{g_\PP}]$ is a perfect set, it is sufficient to show that it is a subset of $A^{V[G]}$. 
%We have that e

Every branch $b$ in $T_{g_\PP}$ 
%,c$ are distinct branches in $T_{\bar{g}}$. 
is $\Add{\lambda}{1}$-generic over $V[G_{\nu+1}\times h]$ by Lemma \ref{distinct branches are mutually generic} applied to forcing with $\PP$ over the model $V[G_{\nu+1}\times h]$. 
%Note that we force with the forcing $\PP$ as defined $V[G_{\mu+1}\times h]$, but since the definition of $\PP$ is absolute between models with the same $V_\kappa$, this is equal to the forcing $\PP$ defined in $V$. 
%in definition of $\PP$ is absolute between $V[G_{\mu+1}]$ and $V[G_{\mu+1}\times h]$. 
Moreover, every branch $b$ in $T_{g_\PP}$ has $\Add{\lambda}{1}$ as a quotient in $V[G_{\nu+1}\times g_\PP \times h]$ over $V[G_{\nu+1}\times h]$ 
%$V[G_{\mu+1}\times h]$ 
%is equivalent to $\Add{\kappa}{1}$ 
by Lemma \ref{quotient forcing of a branch} applied to the same situation. 
%$V[G_{\mu+1}\times h]$. 
It follows that every branch $b$ in $T_{g_\PP}$ has $\Add{\lambda}{1}\times \Col{\lambda}{{<}\mu}$ and hence also $\Col{\lambda}{{<}\mu}$ as a quotient in $V[G]$ over $V[G_{\nu+1}\times h]$. 
%over $V[G_{\mu+1}]$ is equivalent to $\Add{\kappa}{1}\times \Col{\kappa}{{<}\lambda}$ and hence equivalent to $\Col{\kappa}{{<}\lambda}$. 

Since we identify the branch $b$ with an $\Add{\lambda}{1}$-generic filter over $V[G_{\nu+1}\times h]$ that is given by Lemma \ref{distinct branches are mutually generic}, we will also write $\sigma^b$. 
By the choice of $\sigma$ and by the previous statements, we have 
$$\sigma^b\in (A_{\varphi,y}^\kappa)^{V[G_{\nu+1}\times g_\PP\times h]}=A^{V[G]},$$ 
proving the claim. 
\end{proof} 
The last claim completes the proof of Theorem \ref{perfect subsets of definable sets}, since the set $[T_{g_\PP}]$ witnesses the perfect set property of $A^{V[G]}$. 
\end{proof} 

From the last result, we immediately obtain the consistency of the perfect set property for all subsets of $\ltl$ with $\mathsf{DC}_\lambda$. 
%\todo{say more?} 
For instance, it is consistent relative to the existence of an inaccessible cardinal that this is the case in the $\lambda$-Chang model $\mathsf{C}^{\lambda}= L(\Ord^{\lambda})$. 
We further obtain the following global version of the perfect set property. 

\begin{theorem} \label{perfect set property at all regulars} 
Suppose that there is a proper class of inaccessible cardinals. Then there is a class generic extension 
%$V[G]$ 
of $V$ in which 
%such that in $V[G]$, 
for every infinite regular cardinal $\lambda$, the perfect set property holds for every subset of $\ltl$ that is definable from an element of ${}^{\lambda}V$. 
%$\mathsf{PSP}^{\kappa}_{od}$ holds for all infinite regular cardinals $\kappa$ in a class forcing extension. 
\end{theorem} 
\begin{proof} 
Let $C$ be the closure of the class of inaccessible cardinals and $\omega$ and let $\langle \kappa_{\alpha}\mid \alpha\geq 1\rangle$ be the order-preserving enumeration of $C$. 
%the class of all inaccessible cardinals, their limit points and the additional element $\omega$. 

We define the following Easton support iteration $\langle \mathbb{P}_{\alpha},\dot{\mathbb{P}}_{\alpha}\mid \alpha\in \Ord\rangle$ with bounded support at regular limits and full support at singular limits. 
% with bounded support at regular limits is and unbounded support at singular limits. 
Let $\mathbb{P}_0=\{\one\}$. 
%In every successor step \todo{add notation for successor steps} the next interval between successive inaccessible cardinals is collapsed. 
%If $\alpha=0$ or $\alpha$ is a successor, let $\dot{\PP}_{\alpha}$ be a $\PP_{\alpha}$-name for $\Col{\kappa_{\alpha}}{{<}\kappa_{\alpha+1}}$. 
If $\alpha>0$, let $\dot{\nu}_\alpha$ be a $\PP_\alpha$-name for the least regular cardinal $\nu\geq \kappa_\alpha$ 
%with $\kappa_\alpha\leq\nu<\kappa_{\alpha+1}$ 
that is not collapsed by $\PP_\alpha$ and let $\dot{\PP}_{\alpha}$ be a $\PP_{\alpha}$-name for $\Col{\dot{\nu}_\alpha}{{<}\kappa_{\alpha+1}}$. 
%Such a name $\dot{\nu}_\alpha$ exists, 
%Now suppose that $\alpha$ is a limit. If $\kappa_\alpha$ is regular and therefore inaccessible, let $\dot{\PP}_{\alpha}$ be a $\PP_{\alpha}$-name for $\Col{\kappa_{\alpha}}{{<}\kappa_{\alpha+1}}$. If $\kappa_\alpha$ is singular, let $\dot{\nu}_\alpha$ be a $\PP_\alpha$-name for the least regular cardinal $\nu$ with $\kappa_\alpha< \nu<\kappa_{\alpha+1}$ that is not collapsed by $\PP_\alpha$ and let $\dot{\PP}_{\alpha}$ be a $\PP_{\alpha}$-name for $\Col{\dot{\nu}_\alpha}{{<}\kappa_{\alpha+1}}$. 
%Suppose that $\gamma$ is a regular limit and let $\kappa=\cup_{\alpha<\gamma} \kappa_{\alpha}$. Let $\dot{\mathbb{P}}_{\gamma}=Col(\kappa_{\gamma},<\kappa_{\gamma+1})^{\check{}}$. 
%Suppose that $\gamma$ is a singular limit and let $\dot{\kappa}$ denote a $\mathbb{P}_{\gamma}$-name for the least cardinal $\kappa>\sup_{\alpha<\gamma}\kappa_{\alpha}$ which is not collapsed. Let $\dot{\mathbb{P}}_{\gamma}=Col(\kappa,<\kappa_{\gamma})^{\check{}}$. 
Moreover, we can assume that the names $\dot{\PP}_\alpha$ are chosen in a canonical fashion, so that the iteration is definable. 

Let $\PP$ be the iterated forcing defined by this iteration and let further $\dot{\PP}^{(\alpha)}$ be a $\PP_\alpha$-name for the tail forcing of the iteration at stage $\alpha$. 
% as in \todo{cite Jech: factor lemma}. 
It follows from the definition of the iteration that $\one_{\PP_\alpha} \Vdash_{\PP_\alpha} \dot{\PP}^{(\alpha)}$ is ${<}\kappa_\alpha$-closed 
%for all $\alpha\in\Ord$ 
and that $\PP_\alpha$ is strictly smaller than $\kappa_{\alpha+1}$ for all $\alpha\in\Ord$. 

Now suppose that $G$ is $\PP$-generic over $V$. 
We will write $G_\alpha=G\cap \PP_\alpha$ and $\PP^{(\alpha)}=(\dot{\PP}^{(\alpha)})^{G_\alpha}$ for $\alpha\in\Ord$. Moreover, let $\nu_\alpha=\dot{\nu}^{G_\alpha}$ for $\alpha\geq 1$ and $\vec{\nu}=\langle \nu_\alpha\mid \alpha\geq 1\rangle$. 

\begin{claim*} 
\begin{enumerate-(1)} 
\item 
If $\kappa_\alpha$ is inaccessible in $V$, then $\kappa_\alpha$ remains regular in $V[G]$, $\nu_\alpha=\kappa_\alpha$ and $\kappa_\alpha^{+V[G]}=\kappa_{\alpha+1}$. 
\item 
If $\kappa_\alpha$ is a singular limit in $V$, then $\nu_\alpha>\kappa_\alpha$ and $\kappa_\alpha^{+V[G]}=\nu_\alpha$. 
\end{enumerate-(1)} 
\end{claim*} 
\begin{proof} 
If $\kappa_\alpha$ is inaccessible in $V$, it follows from the $\Delta$-system lemma that $\PP_\alpha$ has the $\kappa_\alpha$-cc. The remaining claims easily follow from this and the fact that $\PP^{(\alpha+1)}$ is ${<}\kappa_{\alpha+1}$-closed. 

If $\kappa_\alpha$ is a singular limit in $V$, then 
$\nu_\alpha$ is the least regular cardinal strictly above $\kappa_\alpha$ in $V[G_\alpha]$ by the definition of $\dot{\nu}_\alpha$. 
Moreover, $\nu_\alpha$ is not collapsed in $V[G]$, since $\PP^{(\alpha+1)}$ is ${<}\kappa_{\alpha+1}$-closed. 
\end{proof} 
%\begin{claim*} 
%f $\kappa_\alpha$ is a singular limit, then $\nu_\alpha>\kappa_\alpha$ and $\kappa_\alpha^{+V[G]}=\nu_\alpha$. 
%\end{claim*} 
%\begin{proof} 
%It follows from the definition of $\nu_\alpha$ that $\nu_\alpha$ is the least regular cardinal strictly above $\kappa_\alpha$ in $V[G_\alpha]$. 
%Let $C^*=\{\nu_\alpha\mid \alpha\in\Ord\}$. 
%Let $C_0$ denote the class of cardinals $\kappa_\alpha$, where $\alpha=0$ or $\alpha$ is a successor. Let $C_1$ denote the class of ordinals $\nu_\alpha$, where $\alpha$ is a limit. 

By the previous claim, $\vec{\nu}$ enumerates the class of infinite regular cardinals in $V[G]$. 
%We first show that $C\subseteq \mathrm{Card}^{V[G]}$. It is sufficient to show that for all $\alpha$ such that $\kappa_\alpha$ is a regular cardinal in $V$, $\kappa_\alpha$ is also a regular cardinal in $V[G]$. 
%If $\alpha=0$, then $\kappa_0=\omega$ and hence the claim follows from the definition. 
%If $\alpha$ is a successor, then it is easy to see from the definition above that $\PP_\alpha$ has the $\kappa_\alpha$-cc and $\one_{\PP_\alpha} \Vdash_{\PP_\alpha} \dot{\PP}^{(\alpha)}$ is ${<}\kappa_\alpha$-closed and hence $\kappa_\alpha$ is a regular cardinal in $V[G]$. 
%Now suppose that $\alpha$ is a limit. Since we assume that $\kappa_\alpha$ is regular, we have $\kappa_\alpha=\alpha$. Using the $\Delta$-system lemma, it is easy to see that $\PP_\alpha$ has the $\kappa_\alpha$-cc. 
%This implies that $\one_{\PP_\alpha}\Vdash_{\PP_\alpha}\dot{\nu}_\alpha=\kappa_\alpha$. 
%Since moreover $\one \Vdash_{\PP_\alpha} \dot{\PP}^{(\alpha)}$ is ${<}\kappa_{\alpha+1}$-closed, 
%Hence $\PP$ preserves the regularity of $\kappa_\alpha$. 
%If $\alpha$ is a limit and $\kappa_\alpha$ is inaccessible, then $\kappa_\alpha=\alpha$, $\PP_\alpha$ has the $\kappa_\alpha$-c.c. and $\one \Vdash_{\PP_\alpha} \dot{\PP}^{(\alpha)}$ is $<\kappa_\alpha$-closed. Hence $\PP$ preserves the regularity of $\kappa_\alpha$. 
%We now show that $\mathrm{Card}^{V[G]}\subseteq C$. 
Therefore, we suppose that $\alpha\geq 1$, $\kappa=\nu_\alpha$ and $A$ is a subset of $\kk$ in $V[G]$ that is definable from an element of ${}^{\kappa}V$. 

\begin{claim*} 
$A$ has the perfect set property in $V[G]$. 
\end{claim*} 
\begin{proof} 
Since $\one_{\PP_{\alpha}}$ forces that 
%\Vdash_{\PP_{\alpha}} 
$\dot{\PP}_{\beta}$ is homogeneous for all $\beta\in\Ord$, the tail forcing $\PP^{(\beta)}$ is homogeneous for all $\beta\in\Ord$. 
Since $\PP^{(\alpha+1)}$ is homogeneous, $A$ is an element of $V[G_{\alpha+1}]$. 
Since $\kappa=\nu_\alpha$, $\dot{\PP}_\alpha$ is a name for $\Col{\kappa}{{<}\kappa_{\alpha+1}}$ and hence $A$ has the perfect set property in $V[G_{\alpha+1}]$ by \cite[Theorem 2]{MR0265151} for $\kappa=\omega$ and by Theorem \ref{perfect subsets of definable sets} for $\kappa>\omega$. 
Since $\PP^{(\alpha+1)}$ is $\kappa_{\alpha+1}$-closed, this implies that $A$ has the perfect set property in $V[G]$. 
\end{proof} 

The last claim completes the proof of Theorem \ref{perfect set property at all regulars}. 
%The claim follows from Solovay's theorem \cite{MR0265151}, Theorem \ref{perfect subsets of definable sets} and from the closure of the tail forcings. 
%If $\alpha$ is a limit and $\kappa_\alpha$ is singular, then $\PP_\alpha$ has the $\kappa_\alpha$-c.c. and $\one \Vdash_{\PP_\alpha} \dot{\PP}^{(\alpha)}$ is $<\kappa_\alpha$-closed. Hence $\PP$ preserves the regularity of $\kappa_\alpha$. 
%
%If \todo{this is all nonsensical. rewrite} $\alpha$ is a singular limit, then $\kappa_{\alpha}$ is preserved (as a cardinal), since $|\mathbb{P}_{\leq\alpha}|<\kappa_{\alpha}$ and $\mathbb{P}_{>\alpha}$ is $<\kappa_{\alpha+1}$-closed. Each $\kappa_{\alpha+1}$ is preserved (as a cardinal), since $\mathbb{P}_{\leq\alpha+1}$ has the $\kappa_{\alpha}$-c.c. and $\mathbb{P}_{>\alpha+1}$ is $<\kappa_{\alpha+1}$-closed. So $C$ is the class of regular cardinals in the extension. 
%To see that $PSP^{\kappa_{\alpha}}_{od}$ holds, we factor $\mathbb{P}$ as $\mathbb{P}_{<\alpha}*Col(\kappa_{\alpha},<\kappa_{\alpha+1})*\dot{\mathbb{P}}_{>\alpha}$. Then $PSP^{\kappa_{\alpha}}_{od}$ holds in the $\mathbb{P}_{<\alpha}*Col(\kappa_{\alpha},<\kappa_{\alpha+1})$-extension by Theorem ???. \todo{link} Since $\dot{\mathbb{P}}_{>\alpha}$ is $<\kappa_{\alpha+1}$-closed and homogeneous, it does not add new subsets of ${}^{\kappa}\kappa$ definable from ordinals and subsets of $\kappa$. Hence $PSP^{\kappa_{\alpha}}_{od}$ holds in the extension. 
\end{proof} 

We further remark that the conclusion of Theorem \ref{perfect subsets of definable sets} has the following consequence. 
We define the \emph{Bernstein property} for a subset $A$ of $\ltl$ to mean that $A$ or its complement in $\ltl$ have a perfect subset. 

\begin{lemma} \label{perfect set property and Bernstein property} 
%\todo{\ \ \ \ \ \ $\lambda$?} 
Suppose that $\lambda$ is an uncountable regular cardinal 
%with $\kappa^{<\kappa}=\kappa$ 
and all subsets of $\ltl$ that are definable from elements of $\Ordl$ have the perfect set property. 
Then the following statements hold. 
\begin{enumerate-(1)} 
\item 
All subsets of $\ltl$ that are definable from elements of $\Ordl$ have the Bernstein property. 
\item 
There is no well-order on $\ltl$ that is definable from an element of $\Ordl$. 
\end{enumerate-(1)} 
\end{lemma} 
\begin{proof} 
The first claim is immediate. 
%The assumptions imply the Bernstein property for all subsets $A$ of ${}^{\kappa}\kappa$ that are definable from elements of ${}^{\kappa}\kappa$. 
% i.e. $A$ or ${}^{\kappa}\kappa\setminus A$ has a perfect subset. 
To prove the second claim, suppose towards a contradiction that there is a well-order on $\ltl$ that is definable from an element of $\Ordl$. 
%Since $\kappa^{<\kappa}=\kappa$, 
Using a standard construction, one can then construct a definable Bernstein set by induction. 
%along a definable well-ordering of ${}^{\kappa}\kappa$, using an enumeration of all perfect subsets of ${}^{\kappa}\kappa$. 
\end{proof}

We finally use the previous results to prove a result about definable functions on $\kk$. 
In the statement of the next result, let $[X]^\gamma_{\neq}$ denote the set of sequences $\langle x_i\mid i<\gamma\rangle$ of distinct elements of $X$ for any set $X$ and any ordinal $\gamma$. 
%of $n$-element subsets of $\kk$ with the set of ordered $n$-tuples in $\kk$. 

\begin{theorem}\label{function on 2^kappa continuous on a perfect set} 
Suppose that $\lambda$ is an uncountable regular cardinal, 
%with $\lambda^{<\lambda}=\lambda$, 
$\RR$ is a ${<}\lambda$-distributive forcing and $\gamma<\lambda$. 
% and $V$ is a model of $\ZFC$.
\begin{enumerate-(1)}
\item 
Suppose that  $G$ is $\Add{\lambda}{1}\times \RR$-generic over $V$. Then in $V[G]$, for every function
$f\colon [{}^{\lambda}\lambda]^\gamma_{\neq}\mapsto{}^{\lambda}\lambda$ that is definable from an element of $V$, there is a perfect subset $C$ of $\ltl$ such that $f{\upharpoonright} [C]^\gamma_{\neq}$ is continuous. 
\item 
Suppose that  $G$ is $\Add{\lambda}{\lambda^+}\times \RR$-generic over $V$. Then in $V[G]$, for every function
$f\colon [{}^{\lambda}\lambda]^\gamma_{\neq}\mapsto{}^{\lambda}\lambda$ that is definable from an element of ${}^{\lambda}V$, there is a perfect subset $C$ of $\ltl$ such that $f{\upharpoonright} [C]^\gamma_{\neq}$ is continuous.
%\item 
%Suppose that $\lambda\geq\kappa^+$ and $H$ is $\Add{\kappa}{\lambda}$-generic over $V$. Then in $V[H]$, for every function
%$f\colon [{}^{\kappa}\kappa]^n \mapsto {}^{\kappa}\kappa$ that is definable from ordinals and from a subset of $\kappa$, there is a perfect set $C$ such that $f{\upharpoonright} [C]^n$ is continuous.
%\item Suppose that $\lambda>\kappa$ is inaccessible and $H$ is $\Col{\kappa}{<\lambda}$-generic over $V$. Then in $V[H]$, for every function
%$f\colon [{}^{\kappa}\kappa]^n \mapsto {}^{\kappa}\kappa$ that is definable from ordinals and from subsets of $\kappa$ [WHY NOT $\Ordk$?], there is a perfect set $C$ such that $f{\upharpoonright} [C]^n$ is continuous. 
\end{enumerate-(1)}
\end{theorem} 
\begin{proof} 
We can assume that $\lambda^{<\lambda}=\lambda$ by replacing $V$ with an intermediate model. 

To prove the first claim, it suffices to consider the trivial forcing $\RR=\{\one\}$, since it is easy to see that this implies the claim for arbitrary ${<}\lambda$-distributive forcings $\RR$. 

By 
%\todo{check the statement} 
Lemma \ref{perfect set with good quotients}, there is a perfect subset $C$ of $\ltl$ in $V[G]$ 
such that for every sequence $\vec{x}=\langle x_i\mid i<\gamma\rangle$ of distinct elements of $C$, $\vec{x}$ is $\Add{\lambda}{\gamma}$-generic over $V$ and has $\Add{\lambda}{1}$ as a quotient in $V[G]$ over $V$. 

Suppose that in $V[G]$, we have a function $f\colon [\ltl]^\gamma_{\neq}\rightarrow \ltl$ that is definable from an element of $V$. Then there is a formula $\varphi(\vec{x},y,\alpha,t)$ and some $y\in V$ such that for all $\vec{x}\in[\ltl]^\gamma_{\neq}$ in $V[G]$, $\alpha<\lambda$ and $t\in \ltll$, we have 
$$ f(\vec{x}){\upharpoonright} \alpha=t \Leftrightarrow V[G]\vDash \varphi(\vec{x},y,\alpha,t). $$ 
%where $\varphi$ is a formula with an ordinal parameter, which we omit. 
%Let $V[x]$ denote the least transitive model of $\mathsf{ZF}$ which contains $x$ as an element and $V$ as a subclass. 
Moreover, let $\psi(\vec{x},y,\alpha,t)$ denote the formula 
$$\one_{\Add{\lambda}{1}} \Vdash_{\Add{\kappa}{1}} \varphi(\vec{x},y,\alpha,t).$$ 

For each sequence of distinct elements of $C$ of length $\gamma$, we consider the $\Add{\lambda}{\gamma}$-generic extension $V[\vec{x}]$ of $V$. 
Since $\vec{x}$ has $\Add{\lambda}{1}$ as a quotient in $V[G]$, we have for all $\alpha<\lambda$ and $t\in\ltll$ that 
$$ f(\vec{x}){\upharpoonright} \alpha=t \Leftrightarrow \one_{\Add{\lambda}{1}} \Vdash_{\Add{\lambda}{1}}^{V[\vec{x}]} \varphi(\vec{x},y,\alpha,t)\Leftrightarrow V[\vec{x}]\vDash \psi(\vec{x},y,\alpha,t).$$ 
In particular, it follows that $f(\vec{x})\in V[\vec{x}]$. 
% for all $\vec{x}\in [C]^n$. 
%Suppose that $\psi(x,y)$ holds in $V[G]$ if and only if $f(x)=y$, where $\psi$ is a formula with an ordinal parameter, which we omit.
% in the forcing language with a predicate for $V$. 

\begin{claim*} 
$f{\upharpoonright} [C]^\gamma_{\neq}$ is continuous.
\end{claim*} 
\begin{proof} 
Let $\sigma$ be an $\Add{\lambda}{\gamma}$-name for the sequence of $\Add{\lambda}{1}$-generic subsets of $\lambda$ added by the $\Add{\lambda}{\gamma}$-generic filter. 
%with $\sigma^{\vec{x}}=\vec{x}$ for all $\vec{x}\in [C]^n$.

For every $\vec{x}\in [C]^\gamma_{\neq}$ and every $\alpha<\lambda$, there is a condition $\vec{p}=\langle p_i\mid i<\gamma\rangle$ in the $\Add{\lambda}{\gamma}$-generic filter added by $\vec{x}$ with 
$$\vec{p}\Vdash_{\Add{\lambda}{\gamma}}^{V} \psi(\sigma,y,\alpha,f(\vec{x}){\upharpoonright}\alpha).$$ 
Since $\vec{p}$ is in the generic filter added by $\vec{x}$, we have $p_i\subseteq x_i$ for all $i<\gamma$. 
Now suppose that $\vec{y}=\langle y_i\mid i<\gamma\rangle$ is a sequence of distinct elements of $C$ with $p_i\subseteq y_i$ for all $i<\gamma$. 
By the choice of $\vec{p}$ and the fact that $\vec{y}$ is $\Add{\lambda}{\gamma}$-generic over $V$ and has $\Add{\lambda}{1}$ as a quotient in $V[G]$, we have $f(\vec{x}){\upharpoonright} \alpha= f(\vec{y}){\upharpoonright} \alpha$. 
% for all $n$-tuples $\vec{y}=(y_0,\dots,y_{n-1})$ of distinct elements of $C$ with $p_i\subseteq y_i$ for all $i<n$. 
%Since any such $n$-tuple $\vec{y}$ is $\Add{\kappa}{n}$-generic over $V$. 
It follows that $f{\upharpoonright} [C]^\gamma_{\neq}$ is continuous. 
\end{proof}

To prove the second claim, it suffices to consider the trivial forcing $\RR$, as in the first claim. 
%since it is easy to see that this implies the claim for arbitrary ${<}\kappa$-distributive forcings $\RR$. 
Suppose that $f$ is defined from the parameter $y\in{}^{\lambda}V$. 
We write $G_\alpha=G\cap\Add{\lambda}{\alpha}$ for $\alpha<\lambda^+$. 
Since $\Add{\lambda}{\lambda^+}$ is $\lambda^+$-cc, there is some $\alpha<\lambda^+$ with $y\in V[G_\alpha]$. Since $V[G]$ is an $\Add{\lambda}{\lambda^+}$-generic extension of $V[G_\alpha]$, the claim now follows from the first claim. 
\end{proof}

%%%%%%%%%%%%%%%%%%%%%%%%%%%%%%%%%%%%%%%%%%%%%%%%%%%%%%%%%%%%%%%%
%%%%%%%%%%%%%%%%%%%%%%%%%%%%%%%%%%%%%%%%%%%%%%%%%%%%%%%%%%%%%%%%
\section{The almost Baire property}  \label{section: The almost Baire property} 

In the first part of this section, we define an analogue to the Baire property. 
%that can consistently hold for all sufficiently definable subsets of ${}^{\kappa}\kappa$. 
This property is characterized by a Banach-Mazur type game (see \cite[Section 8.H]{MR1321597}) of uncountable length. 
%$\kappa$. 
%We always assume that $\kappa$ is an uncountable regular cardinal with $\kappa^{<\kappa}=\kappa$. 

%%%%%%%%%%%%%%%%%%%%%%%%%%%%%%%%%%%%%%%%%%%%%%%%%%%%%%%%%%%%%%%%
\subsection{Banach-Mazur games} \label{subsection: Banach-Mazur games}

In this section, we assume that $\nu$ is an infinite regular cardinal with $\nu^{<\nu}=\nu$ (we write $\nu$ instead of $\kappa$, since $\nu=\omega$ is allowed). 
The \emph{standard topology} (or \emph{bounded topology}) on ${}^{\nu}\nu$ is generated by the basic open sets 
$$N_t=\{x\in {}^{\nu}\nu\mid t\subseteq x\}$$ 
for $t\in {}^{<\nu}\nu$. 
The analogue to the Baire property will be defined using the following types of functions. 

\begin{definition} 
Suppose that $f\colon {}^{<\nu}\nu\rightarrow {}^{<\nu}\nu$ is given. 
\begin{enumerate-(1)} 
\item 
$f$ is a \emph{homomorphism} if for all $s\subsetneq t$ in ${}^{<\nu}\nu$, we have $f(s)\subsetneq f(t)$. 
\item 
$f$ is \emph{continuous} if for every limit $\gamma<\nu$ and every strictly increasing sequence $\langle s_\alpha\mid \alpha<\gamma\rangle$ in ${}^{<\nu}\nu$, we have 
$$f(\bigcup_{\alpha<\gamma} s_\alpha)= \bigcup_{\alpha<\gamma} f(s_\alpha).$$ 
\item 
%\begin{enumerate-(i)} 
%\item 
%$f$ is a \emph{homomorphism} if $f(s)\subseteq f(t)$ for all $s\subseteq t$ in ${}^{<\kappa}\kappa$. 
%\item 
$f$ is \emph{dense} if for all $s\in {}^{<\nu}\nu$, the set $$\{f(s^\smallfrown \langle \alpha\rangle) \mid \alpha<\nu\}$$ is dense above $f(s)$ in the sense that for any $t\supseteq f(s)$, there is some $\alpha<\nu$ with $f(s^\smallfrown\langle \alpha\rangle)\supseteq t$. 
\item 
If $f$ is a homomorphism, let $f^*$ denote the function $f^*\colon {}^{\nu}\nu\rightarrow {}^{\nu}\nu$ defined by 
$$f^*(x)=\bigcup_{\alpha<\nu}f(x{\upharpoonright} \alpha).$$ 
\end{enumerate-(1)} 
%\item 
%$[f]=\{\bigcup_{\alpha<\kappa} f(x\upharpoonright \alpha)\mid x\in {}^{<\kappa}\kappa\}$. 
%\item \todo{it's not clear if we need all this\ \\ 
%better: write all this without separate items} 
%$A$ is \emph{large} if there is a dense map $f\colon {}^{<\kappa}\kappa\rightarrow {}^{<\kappa}\kappa$ with $[f]\subseteq A$. 
%\item 
%$A$ has the \emph{almost Baire property} if for every $t\in {}^{<\kappa}\kappa$, $A\cap N_t$ is $\kappa$-meager in $N_t$ or $A\cap N_t$ is large in $N_t$. 
%\end{enumerate-(i)} 
\end{definition} 
%\begin{definition} 
%A map $f\colon {}^{<\kappa}\kappa\rightarrow {}^{<\kappa}\kappa$ is \emph{successor-dense} if for all $s\in {}^{<\kappa}\kappa$, the set $\{f(s^\smallfrown \langle \alpha\rangle\mid \alpha<\kappa\}$ is a dense subset of $N_{f(s)}$. 
%\end{definition} 

By using such functions on ${}^{<\nu}\nu$, we can characterize comeager subsets of ${}^{\nu}\nu$, which were defined in Definition \ref{definition: kappa-Borel etc}, as follows. 

\begin{lemma} \label{characterization of dense continuous homomorphism} 
Suppose that $\nu$ is an infinite cardinal with $\nu^{<\nu}=\nu$ and $t\in {}^{<\nu}\nu$. 
A subset $A$ of ${}^{\nu}\nu$ is comeager in $N_t$ if and only if there is a dense continuous homomorphism $f\colon {}^{<\nu}\nu\rightarrow {}^{<\nu}\nu$ with $f(\emptyset)=t$ and $\ran{f^*}\subseteq A$. 
\end{lemma} 
\begin{proof} 
To prove the first implication, suppose that $A$ is comeager in $N_t$. 
Then there is a sequence $\langle U_\alpha\mid \alpha<\nu \rangle$ of dense open subsets of $N_t$ with $\bigcap_{\alpha<\nu} U_\alpha\subseteq A$. 
Since any intersection of strictly less than $\nu$ many dense open subsets of $N_t$ is again dense open in $N_t$, we can assume that $U_\beta\subseteq U_\alpha$ for all $\alpha<\beta<\nu$. 

We now define $f(s)$ by induction on $\len(s)$. 
%For $\len(s)=0$, we define 
Let $f(\emptyset)=t$. 
In the successor case, suppose that $\len(s)=\gamma$ and that $f(s)$ is defined. 
%Let $N^{f(s)}=\{t\in {}^{<\kappa}\kappa\mid s\subseteq t,\ s\neq t\}$. 
Since $U_\gamma$ is a dense open subset of $N_t$, the set 
$$K=\{u\supsetneq f(s)\mid N_u\subseteq U_\gamma\}$$ 
 is dense above $f(s)$ in the sense that for every $t\supseteq f(s)$, there is some $v\in K$ with $t\subseteq v$. 
%subset $N$ of $N^{f(s)}$ of size $\kappa$ such that for all $t\in N$, $N_t\subseteq U_\gamma$ and $t\neq f(s)$. 
Since $\nu^{<\nu}=\nu$, we can choose an enumeration $\langle t_\alpha\mid \alpha<\nu\rangle$ of $K$. 
%of a dense subset of $N^{f(s)}$ such that $N_{t_\alpha}\subseteq U_\gamma$ for all $\alpha<\kappa$. 
We then define $f(s^\smallfrown \langle \alpha\rangle) = t_\alpha$ for all $\alpha<\nu$. 
In the limit case, suppose that $\len(s)=\gamma$ is a limit and that $f(s{\upharpoonright} \bar{\gamma})$ is defined for all $\bar{\gamma}<\gamma$. We then define $f(s)=\bigcup_{\bar{\gamma}<\gamma} f(s{\upharpoonright} \bar{\gamma})$.

It follows from the construction that $f$ satisfies the required properties and that $\ran{f^*}\subseteq \bigcap_{\alpha<\nu} U_\alpha\subseteq A$. 

To prove the reverse implication, suppose that $f$ satisfies the conditions stated above. 
%suppose that $f\colon {}^{<\kappa}\kappa\rightarrow {}^{<\kappa}\kappa$ is a dense continuous injective homomorphism with $\ran{f^*}\subseteq A$. 
For any $x\in {}^{\nu}\nu$, let 
$$K_x=\{s\in{}^{<\nu}\nu\mid t\subseteq s,\ f(s)\subseteq x\}.$$ 
Since $f(\emptyset)=t$, $K_x$ is nonempty for any $x\in N_t$. 

\begin{claim*} 
For any $x\in N_t$, if $K_x$ has no maximal elements, then $x\in A$. 
%Suppose that $x\in {}^{\kappa}\kappa$ and for all $s\in {}^{<\kappa}\kappa$ with $f(s)\subseteq x$, there is some $t\in {}^{<\kappa}\kappa$ with $s\subsetneq t\subseteq x$ and $f(t)\subseteq x$. Then $x\in A$. 
\end{claim*} 
\begin{proof} 
Since $K_x$ is nonempty and has no maximal elements, we can build a strictly increasing sequence $\langle s_\alpha\mid \alpha<\nu\rangle$ in ${}^{<\nu}\nu$ with $s_0=t$ and $f(s_\alpha)\subseteq x$ for all $\alpha<\nu$. By the definition of $f^*$, this implies that $x\in \ran{f^*}\subseteq A$. 
\end{proof} 

For any $s\in {}^{<\nu}\nu$ with $t\subseteq s$, we now consider the set $C_s$ of $x\in N_t$ such that $s$ is a maximal element of $K_x$. 
%Let $C_s=\{x\in {}^{\kappa}\kappa\mid f(s)\subseteq x\text{ and for all }t\in{}^{<\kappa}\kappa \text{ with }s\subsetneq t\subseteq x, f(t)\not\subseteq x\}$. 
It is easy to see that $C_s$ is a closed nowhere dense subset of $N_{f(s)}$. 
Since $N_t \setminus A \subseteq \bigcup_{s\supseteq t} C_s$ by the previous claim, it follows that $A$ is comeager in $N_t$. 
%\todo{delete this} Suppose that $f\colon {}^{<\kappa}\kappa\rightarrow {}^{<\kappa}\kappa$ is a continuous successor-dense injective homomorphism with $[f]\subseteq A$. 
%Let $U_\alpha=\bigcup_{t\in {}^{\alpha}\kappa} N_{f(t)}$. 
%We claim that $U_\alpha$ is open dense $\alpha<\kappa$. 
%If $\alpha$ is a successor, this follows from the assumption that $f$ is successor-dense. 
%If $\alpha$ is a limit, $U_\alpha=\bigcap_{\bar{\alpha}<\alpha} U_{\bar{\alpha}}$, since $f$ is continuous, hence $U_\alpha$ is open dense by the inductive assumption. 
%Since $[f]=\bigcap_{\alpha<\kappa} U_\alpha$, $A$ is comeager. 
\end{proof} 

We now define an asymmetric version of the Baire property, using the functions above. 
%To avoid such counterexamples, we consider the following assymmetric analogy to the Baire property. 

\begin{definition} \label{definition of almost Baire} 
A subset $A$ of ${}^{\nu}\nu$ is \emph{almost $\nu$-Baire} (\emph{almost Baire}) if there is a dense homomorphism $f\colon {}^{<\nu}\nu\rightarrow {}^{<\nu}\nu$ with one of the following properties. 
\begin{enumerate-(a)} 
\item 
$\ran{f^*}\subseteq A$. 
\item 
$\ran{f^*}\subseteq {}^{\nu}\nu \setminus A$, $f$ is continuous and $f(\emptyset)=\emptyset$. 
\end{enumerate-(a)} 
\end{definition} 

Since every homomorphism is continuous for $\nu=\omega$, it follows immediately from Lemma \ref{characterization of dense continuous homomorphism} that a subset $A$ of ${}^{\omega}\omega$ is almost Baire if and only if there is some $t\in {}^{<\omega}\omega$ such that $A$ is comeager in  $N_t$ or ${}^{\omega}\omega\setminus A$ is comeager. 
% in ${}^{\omega}\omega$. 
It can be easily seen that this implies that for every class $\Gamma$ of subsets of ${}^{\omega}\omega$ that is closed under continuous preimages, the almost Baire property for all sets in $\Gamma$ is equivalent to the Baire property for all sets in $\Gamma$. 
%\begin{lemma} 
%A subset $A$ of ${}^{\omega}\omega$ is almost Baire if and only if it is comeager or there is some $t\in {}^{<\omega}\omega$ such that $A\cap N_t$ is comeager in $N_t$. 
%\end{lemma} 
%\begin{proof} 
%This follows from Lemma \ref{characterization of dense continuous homomorphism}. 
%\end{proof} 

The continuity in the definition of almost Baire is necessary by the next result. To state this result, let $\Cl_\nu$ denote the set 
$$\Cl_\nu=\{x\in {}^{\nu}\nu\mid \exists C\subseteq\nu \text{ club } \forall i\in C\ x(i)\neq 0\}$$ 
of functions coding elements of the club filter on $\nu$ as characteristic functions, and 
$$\Ns_\nu=\{x\in {}^{\nu}\nu\mid \exists C\subseteq\nu \text{ club } \forall i\in C\ x(i)= 0\}$$ 
%$$\Ns_\nu=\{x\in {}^{\nu}\nu\mid \nu\setminus x\in \Cl_\nu\}$$ 
the set of functions coding elements of the non-stationary ideal on $\nu$. 
% (more precisely, $\Cl_\nu$ is the set of functions coding elements of the club filter as characteristic functions). 
%The above examples for non-$\kappa$-Baire sets are analogous for this notion. 
%Let $\Cl_\kappa$ denote the set of $x\in {}^{\kappa}\kappa$ such that there is a club $C$ in $\kappa$ with $x(i)\neq 0$ for all $i\in C$. 
%Let $\Ns_\kappa$ denote the set of $x\in {}^{\kappa}\kappa$ such that there is a club $C$ in $\kappa$ with $x(i)= 0$ for all $i\in C$. 

\begin{lemma} \label{almost Baire property for omega} 
%\todo{this doesn't really work. only on ${}^{\kappa}2$. define the club filter appropriately?} 
$\Cl_\nu$ and $\Ns_\nu$ are almost Baire subsets of ${}^{\nu}\nu$, but for every dense continuous homomorphism $f\colon {}^{<\nu}\nu\rightarrow {}^{<\nu}\nu$, we have $\ran{f^*}\cap \Cl_\nu\neq\emptyset$ and $\ran{f^*}\cap \Ns_\nu\neq \emptyset$. 
\end{lemma} 
\begin{proof} 
%\todo{write more?? no...} 
It is easy to see that $\Cl_\nu$ and $\Ns_\nu$ are almost Baire subsets of ${}^{\nu}\nu$. 

Since the remaining claims are symmetric, it is sufficient to prove that $\ran{f^*}\cap \Cl_\nu\neq\emptyset$. 
% by symmetry. 
%To show that $\ran{f^*}\cap \Cl_\kappa\neq\emptyset$, 
We define a sequence $\langle x(\gamma)\mid \gamma<\nu \rangle$ with values in $\nu$ by the following induction. 
%and $\langle t_\gamma\mid \gamma<\kappa\rangle$ with $t_\gamma=\langle x(\alpha)\mid \alpha<\gamma\rangle$ by induction. 
%We define a sequence $\langle t_\alpha \mid \alpha<\kappa \rangle$ by induction. 
Suppose that $\gamma<\nu$, $s=\langle x(\alpha)\mid \alpha<\gamma\rangle$ is already defined and $\len(f(s))=\delta$. 
If $\gamma$ is a successor, since $f$ is dense, there is some $\eta<\nu$ such that $f(s^\smallfrown \langle \eta\rangle)(\delta)=1$. 
% for $t=s^\smallfrown \langle \eta\rangle$. 
If $\gamma$ is a limit, the same conclusion follows from the additional assumption that $f$ is continuous. 
% (if $f$ is not continuous, this only works for successors $\gamma$). 
In both cases, we let $x(\gamma)=\eta$. 

By the construction, we have $x\in \ran{f^*}\cap \Cl_\nu$ and hence $\ran{f^*}\cap \Cl_\nu\neq\emptyset$, proving the claim. 
%The argument for $\ran{f^*}\cap \Ns_\kappa\neq\emptyset$ is analogous. 
\end{proof} 

The motivation for the definition of the almost Baire property comes from its connection with the following game. 
%The following game characterizes the almost Baire property. 
%shows the relationship of the almost Baire property with the Banach-Mazur game. 

\begin{definition} \label{definition:  Banach-Mazur game}
%\todo{conflict of notation with the perfect set game?} 
The \emph{Banach-Mazur game} $G_\nu(A)$ of length $\nu$ for a subset $A$ of ${}^{\nu}\nu$ is defined as follows. 
%The game has length $\kappa$. 
The first (even) player, player I, plays an element of ${}^{<\nu}\nu$ in each even round. 
% and the second (odd) player plays in all odd rounds. 
The second (odd) player, player II, plays an element of ${}^{<\nu}\nu$ in each odd round. 
Together, they play a strictly increasing sequence $\vec{s}=\langle s_\alpha\mid \alpha<\nu\rangle$ with $s_\alpha\in {}^{<\nu}\nu$ for all $\alpha<\nu$. 
Thus the sequence of moves of both players defines a sequence 
$$\bigcup_{\alpha<\nu} s_\alpha=x=\langle x(i)\mid i<\nu \rangle \in {}^{\nu}\nu$$ 
and the first player wins this run if $x\in A$. 

The Banach-Mazur game of length $\nu$ with these rules, but without a specific winning set, is denoted by $G_\nu$. 
Moreover, for any $t\in {}^{<\nu}\nu$, 
the game $G^{t}_\nu(A)$ is defined as $G_\nu(A)$ but with the additional requirement that $t\subseteq s_0$ for the first move $s_0$ of player I. 
\end{definition} 

We will also consider the games $G_{\nu}^2(A)$ and  $G_{\nu}^{2,(s,t)}(A)$ for $(s,t)\in ({}^{<\nu}\nu)^2$ with $\len(s)=\len(t)$ that are defined in analogy with $G_\nu(A)$. In these games, the players play elements $(u,v)$ of $({}^{<\nu}\nu)^2$ with $\len(u)=\len(v)$ and $A$ is a subset of $({}^{\nu}\nu)^2$. 
%\todo{check below} 
It is easy to check that all results for $G_\nu$ in this section also hold for $G_\nu^2$, since the proofs can be easily modified to work for this game. 

%The next two results relate the Banach-Mazur game with the almost Baire property. 
The next two results show the equivalence between the determinacy of $G_\nu(A)$ and the almost Baire property for $A$. 

\begin{lemma} \label{characterization of dense function} 
The following are pairs of equivalent statements for any subset $A$ of ${}^{\nu}\nu$. 
\begin{enumerate-(1)} 
\item 
\begin{enumerate-(a)} 
\item 
Player I has a winning strategy in $G_\nu(A)$. 
\item 
There is a dense homomorphism $f\colon {}^{<\nu}\nu\rightarrow {}^{<\nu}\nu$ with $\ran{f^*}\subseteq A$. 
\end{enumerate-(a)} 
\item 
\begin{enumerate-(a)} 
\item 
Player II has a winning strategy in $G_\nu(A)$. 
\item 
There is a dense continuous homomorphism $f\colon {}^{<\nu}\nu\rightarrow {}^{<\nu}\nu$ with $\ran{f^*}\subseteq {}^{\nu}\nu\setminus A$ and $f(\emptyset)=\emptyset$. 
\end{enumerate-(a)} 
\end{enumerate-(1)} 
\end{lemma} 
\begin{proof} 
%\todo{double check this proof} 
%\todo{more about the second equivalence?} 
We will only prove the first equivalence, since the proof of the second equivalence is analogous. 

%\todo{read this again?} 
To prove the first implication, suppose that player I has a winning strategy $\sigma$ in $G_\nu(A)$. 
For all $t\in {}^{<\nu}\nu$, by induction on $\len(t)$, we will define $f(t)$ and 
%We define $f(t)$ by induction on $\len(t)=\gamma_t$ and 
partial runs 
$$\vec{s}_t=\langle s_t(\alpha)\mid \alpha<2\cdot \len(t)+1 \rangle$$ 
according to $\sigma$ such that $\vec{s}_t\subseteq \vec{s}_u$ for all $t\subseteq u$ and $f(t)(\alpha)=s_t(2\cdot \alpha)$ for all $\alpha<\len(t)$. 
%$\langle f(t)(\alpha)\mid \alpha<\gamma_t\rangle$. 
%Let $f(\emptyset)=\emptyset$. 

We begin by considering the first move $v=\sigma(\emptyset)$ of player I according to $\sigma$ and defining $f(\emptyset)=v$ and $\vec{s}_{\emptyset}=\langle v\rangle$. 
In the successor step, suppose that $t\in {}^{<\nu}\nu$ and $f(t)$, $\vec{s}_t$ are defined. 
Moreover, suppose that $\langle u_\alpha\mid \alpha<\nu\rangle$ is an enumeration of the possible responses of player II to $\vec{s}_t$ 
%Suppose that $\langle u_\alpha\mid \alpha<\nu\rangle$ is an enumeration of these moves. 
and that for each $\alpha<\nu$, $v_\alpha$ is the response of player I to $\vec{s}_t^\smallfrown\langle u_\alpha\rangle$ according to $\sigma$. 
Let $\vec{s}_{t^\smallfrown\langle\alpha\rangle}= \vec{s}_{t}^\smallfrown \langle u_\alpha, v_\alpha\rangle$ 
and $f(t^\smallfrown \langle \alpha\rangle)= v_\alpha$. 
%=s_{t^\smallfrown\langle \alpha\rangle}(2\cdot \len(t)).$$ 

In the limit step, suppose that $\len(t)$ is a limit and that $\vec{s}_{t\upharpoonright \alpha}$ and $f(t{\upharpoonright} \alpha)$ are defined for all $\alpha<\len(t)$. 
If $v$ is the response of player I to $\bigcup_{\alpha<\len(t)}\vec{s}_{t\upharpoonright\alpha}$ according to $\sigma$, 
let $\vec{s}_{t}=(\bigcup_{\alpha<\len(t)}\vec{s}_{t\upharpoonright\alpha})^\smallfrown \langle v\rangle$ and $f(t)=v$. 
This completes the definition of $f$ and by the construction, $f$ is a dense homomorphism with $\ran{f^*}\subseteq A$. 

To prove the second implication, suppose that $f\colon {}^{<\nu}\nu\rightarrow {}^{<\nu}\nu$ is a dense homomorphism with $\ran{f^*}\subseteq A$. 
We will define a winning strategy $\sigma$ for player I in $G_\nu(A)$. 
%It is sufficient to define $\sigma(\vec{s})$ for all partial runs $\vec{s}$ according to $\sigma$ of even length. 
To this end, by induction on $\len(\vec{s})$, we will define $t_{\vec{s}}, \sigma(\vec{s})\in \nu^{<\nu}$ for all partial runs $\vec{s}$ of even length according to $\sigma$ such that 
$\len(t_{\vec{s}{\upharpoonright}2\cdot\alpha})=\alpha$, $t_{\vec{s}{\upharpoonright}2\cdot\alpha}\subseteq t_{\vec{s}{\upharpoonright}2\cdot\beta}$ 
and $\sigma(\vec{s}{\upharpoonright}2\cdot \alpha)=f(t_{\vec{s}{\upharpoonright}2\cdot \alpha})$
for all $\alpha$, $\beta$ with $2\cdot \alpha\leq2\cdot\beta\leq \len(\vec{s})$. 
%$t_{\vec{s}{\upharpoonright}\alpha}\subsetneq t_{\vec{s}{\upharpoonright}\beta}$ for all even $\alpha<\beta\leq\len(\vec{s})$ and 
%$f(t_{\vec{s}{\upharpoonright}\alpha})\subseteq \sigma(\vec{s}{\upharpoonright}\alpha)$ for all even $\alpha\leq \len(\vec{s})$. 
%and $f(t_{\vec{s}{\upharpoonright}\alpha})\subseteq \vec{s}(\alpha)$ 
%$\vec{s}(2\cdot\alpha) \subseteq f(t_{\vec{s}{\upharpoonright}\alpha})\subseteq \vec{s}(2\cdot\alpha+2)$
%$t_{\vec{s}{\upharpoonright}(\alpha+2)}\subseteq\sigma(\vec{s}{\upharpoonright}\alpha)$ 
%$\sigma(\vec{s})$
%$\len(t_{\vec{s}{\upharpoonright}2\cdot \alpha})=\alpha$ for all $\alpha$ with $2\cdot \alpha\leq\len(\vec{s})$ and $\sigma((\vec{s}{\upharpoonright}2\cdot \alpha)^\smallfrown\langle f(t_{\vec{s}{\upharpoonright}\alpha})\rangle)=\vec{s}{\upharpoonright}(2\cdot \alpha+2)$ for all $\alpha$ with $2\cdot \alpha<\len(\vec{s})$. 

We begin by defininig $\sigma(\emptyset)=f(\emptyset)$. 
In the successor step, suppose that $\len(\vec{s})$ is even and that $t_{\vec{s}{\upharpoonright}\alpha}$, $\sigma(\vec{s}{\upharpoonright}\alpha)$ are defined for all even $\alpha\leq\len(\vec{s})$. 
Moreover, suppose that $u$ is a possible move of player II extending the partial run $\vec{s}^\smallfrown \langle \sigma(\vec{s})\rangle$, so that $\sigma(\vec{s})\subsetneq u$. 
%$\sigma(\vec{s})=\vec{s}^\smallfrown\langle u\rangle$ and $u\subsetneq v$. 
Since $f$ is dense, there is  some $\alpha<\nu$ with $u\subsetneq f(t_{\vec{s}}^\smallfrown\langle \alpha\rangle)$. Let 
$t_{\vec{s}^\smallfrown \langle\sigma(\vec{s}), u\rangle}=t_{\vec{s}}^\smallfrown\langle\alpha\rangle$ and 
$\sigma(\vec{s}^\smallfrown \langle\sigma(\vec{s}), u\rangle)=f(\vec{s}^\smallfrown \langle\sigma(\vec{s}), u\rangle)$. 
%and let $\sigma(\vec{s})=\vec{s}^\smallfrown\langle f(t_{\vec{s}})\rangle$. 

In the limit step, suppose that $l(\vec{s})=\gamma$ is a limit and $t_{\vec{s}{\upharpoonright}\alpha}$, $\sigma(\vec{s}{\upharpoonright}\alpha)$ are defined for all even $\alpha<\len(\vec{s})$. 
Let $t_{\vec{s}}=\bigcup_{\alpha<\gamma} t_{\vec{s}\upharpoonright\alpha}$ and $\sigma(\vec{s})=f(t_{\vec{s}})$. 
It is now easy to check that $\sigma$ is a a winning strategy for player I in $G_\nu(A)$. 
\end{proof} 
%\begin{lemma} \label{characterization of dense continuous function} 
%The following statements are equivalent. 
%\begin{enumerate-(1)} 
%\item 
%Player II has a winning strategy in $G_\nu(A)$. 
%\item 
%There is a dense continuous homomorphism $f\colon {}^{<\nu}\nu\rightarrow {}^{<\nu}\nu$ with $f(\emptyset)=\emptyset$ and $\ran{f^*}\subseteq {}^{\nu}\nu\setminus A$. 
%\end{enumerate-(1)} 
%\end{lemma} 
%\begin{proof} 
%This is analogous to the proof of Lemma \ref{characterization of dense function}. 
%\end{proof} 

In the next result, we will consider the following stronger type of strategy for $G_\nu$ that only relies on the union of the previous moves. 

\begin{definition} 
A \emph{tactic} in $G_\nu$ is a strategy $\sigma$ such that there is a map $\bar{\sigma}\colon {}^{<\nu}\nu\rightarrow {}^{<\nu}\nu$ with the property that 
$$\sigma(\vec{s})=\bar{\sigma}(\bigcup_{\alpha<\gamma} s_\alpha)$$ 
for all $\vec{s}=\langle s_\alpha\mid \alpha<\gamma\rangle\in \dom{\sigma}$. 
\end{definition} 

The next result, which follows from \cite[Lemma 7.3.2]{Kovachev-thesis}, relates the Banach-Mazur game of length $\nu$ with the $\nu$-Baire property. 

\begin{lemma} \label{characterization of Baire property by game} 
%\todo{see (cite dissertation of Kovacev)} 
Suppose that $A$ is a subset of ${}^{\nu}\nu$ and $t\in {}^{<\nu}\nu$. 
\begin{enumerate-(1)} 
\item \label{II wins if the set is meager} 
(Kovachev) The following conditions are equivalent. 
\begin{enumerate-(a)} 
\item 
$A$ is meager in $N_t$.
\item 
Player II has a winning strategy in $G^t_\nu(A)$. 
\item 
Player II has a winning tactic in $G^t_\nu(A)$. 
\end{enumerate-(a)}  
\item \label{I wins if the set is not meager} 
If $A\cap N_t$ is $\nu$-Baire, then the following conditions are equivalent. 
\begin{enumerate-(a)} 
\item 
$A$ is meager in $N_t$. 
\item 
Player I does not have a winning strategy in $G^t_\nu(A)$. 
\end{enumerate-(a)} 
\item \label{I wins if the set is somewhere comeager} 
If $\nu=\omega$, then the following conditions are equivalent. 
\begin{enumerate-(a)} 
\item 
$A$ is comeager in $N_u$ for some $u\supseteq t$. 
%$u\in {}^{<\nu}\nu$ with $t\subseteq u$. 
%\cap N_u$ 
\item 
Player I has a winning strategy in $G^t_\omega(A)$. 
% if and only if there is some $u\in {}^{<\nu}\nu$ with $t\subseteq u$ such that $A$ is comeager in $N_u$. 
\end{enumerate-(a)} 
\end{enumerate-(1)} 
\end{lemma} 
\begin{proof} 
%\ref{II wins if the set is meager} 
The first claim is proved in \cite[Lemma 7.3.2]{Kovachev-thesis}. Since the remaining claims are easy consequences of this, we only sketch the proofs.  
%\ref{I wins if the set is not meager} 
%Suppose first that $A$ is meager in $N_t$. Then player II has a winning strategy in $G^t_\nu(A)$ by \ref{II wins if the set is meager}, so player I does not have a winning strategy. 

For the second claim, suppose that $A$ is not meager in $N_t$. Since $A$ is $\nu$-Baire, $A\cap N_u$ is comeager in $N_u$ for some $u\supseteq t$. 
%$u\in {}^{<\nu}\nu$ with $t\subseteq u$. 
%\todo{maybe define this more generally} Let $G_\nu^t$ denote $G_\nu$ with the additional requirement that the first move strictly extends $t$. 
By the first claim, there is a winning strategy $\sigma$ for player II in $G_\nu^u(N_u\setminus A)$.
% by \ref{II wins if the set is meager}. 
This means that player II succeeds with playing in $A$. 
Since it is harder for player II to win because she or he does not play at limits, we easily obtain a winning strategy $\tau$  for player I in $G^t_\nu(A)$ with the first move $u$ from $\sigma$. 
%Let $u$ be the first move of $\tau$. 
%It is straighforward to define $\tau$ from $\sigma$, similar to the proof of Lemma \ref{strategy from almost strategy}. 
%Suppose that $\langle U_\alpha\mid \alpha<\nu\rangle$ is a sequence of open dense sets such that $A\cap (\bigcap_{i<\nu} U_\alpha) =\emptyset$. 
%The second player can win by playing into $U_\alpha$ in step $\alpha$. 

For the third claim, suppose that $A$ is comeager in $N_u$ for some $u\supseteq t$. Since player II has a winning strategy in $G^u_\omega({}^{\nu}\nu\setminus A)$ by the first claim, we obtain a winning strategy for player I in $G^t_\omega(A)$ with the first move $u$ by switching the roles of the players. 
The reverse implication follows similarly from the first claim. 
%\ref{I wins if the set is somewhere comeager} 
%The game obtained by fixing the first move of player I and switching the roles of player I and II after this move corresponds to $G_\nu$ for $\nu=\omega$. Hence \ref{I wins if the set is somewhere comeager} follows from \ref{II wins if the set is meager}. 
%The last claim follows from the first claim, since $G_\omega$ is symmetric 
\end{proof} 

This shows together with Lemma \ref{characterization of dense continuous homomorphism} that for any class $\Gamma$ of subsets of the Baire space ${}^{\omega}\omega$ that is closed under continuous preimages, the statement that $G_\omega(A)$ is determined for all sets $A\in \Gamma$ is equivalent to the statement that all sets in $\Gamma$ have the property of Baire. 
%implies that the Baire property for all subsets of the Baire space ${}^{\omega}\omega$ definable from elements of ${}^{\omega}\Ord$ is precisely characterized by the determinacy of the Banach-Mazur game. 

Moreover, the previous result shows that $G_\nu(A)$ is determined for every $\nu$-Baire subset $A$ of ${}^{\nu}\nu$. 
The game is also  
%the game $G_\nu(A)$ 
determined for some $\Sigma^1_1$ subsets of ${}^{\nu}\nu$ that are not $\nu$-Baire, 
since it is easy to see that 
%\begin{example} 
player I has a winning strategy in $G_\nu(A)$ if $A$ is one of the sets $\Cl_\nu$, 
%on $\nu$ and for the non-stationary ideal 
$\mathrm{NS}_\nu$ that are defined after Definition \ref{definition of almost Baire}. 
%on $\nu$. 
%\end{example} 
This leads to the question for which definable subsets $A$ of $\kk$ the Banach-Mazur game is determined. We study this question in the next section.

%%%%%%%%%%%%%%%%%%%%%%%%%%%%%%%%%%%%%%%%%%%%%%%%%%%%%%%%%%%%%%%%
\subsection{The almost Baire property for definable sets}
%Consistency of the almost Baire property} 
\label{subsection: consistency of the almost Baire property} 

%\todo[inline]{maybe call: The almost Baire property for definable sets?} 

As before, we always assume that $\kappa$ is an uncountable regular cardinal with $\kappa^{<\kappa}=\kappa$. 

In this section, we will prove that it is consistent for the Banach-Mazur game $G_\kappa$ to be determined for all subsets of ${}^{\kappa}\kappa$ that are definable from elements of $\Ordk$. 
This will also imply that it is consistent that the almost Baire property holds for all such sets by the results in the previous section. 
%This is proved via determinacy of the game $G_\kappa$. 
% of length $\kappa$. 
%Banach-Mazur game is \todo{AND THE ALMOST Baire property holds} determined for all sufficiently definable subsets $A$ of ${}^{\kappa}\kappa$. 

The following notions will be used to construct strategies for the first player in $G_\kappa$. 

\begin{definition} 
\begin{enumerate-(i)}
\item 
An \emph{almost strategy} 
%for \todo{can we also define this for player II?} 
for player I in $G_\kappa$ 
%for a subset $A$ of ${}^{\kappa}\kappa$ 
is a partial strategy $\sigma$ such that $\dom{\sigma}$ is dense in the following sense. 
Suppose that $\gamma<\kappa$ is odd, 
$\vec{s}=\langle s_\alpha\mid \alpha<\gamma\rangle$ is a strictly increasing sequence in ${}^{<\kappa}\kappa$ according to $\sigma$ 
%$\delta=\gamma+1<\kappa$, 
and $\bigcup_{\alpha<\gamma} s_\alpha \subsetneq v$. 
% and $\len(u)<\len(v)$. 
Then there is some $w\in {}^{<\kappa}\kappa$ with $v \subseteq w$ and $\vec{s}^\smallfrown \langle w\rangle\in \dom{\sigma}$. 
%$\langle s_\alpha\mid \alpha<\gamma\rangle^\smallfrown \langle u\rangle \in \dom{\sigma}$. 
%\item 
%An \emph{almost tactic} is a partial tactic that is an almost strategy. 
\item 
If $\sigma$, $\tau$ are partial strategies for player I in $G_\kappa$, then $\tau$ \emph{expands} $\sigma$ if for every run $\vec{s}=\langle s_\alpha\mid \alpha<\kappa\rangle$ according to $\tau$, there is a run $\vec{t}=\langle t_\alpha\mid \alpha<\kappa\rangle$ according to $\sigma$ with the same \emph{outcome} $\bigcup_{\alpha<\kappa} s_\alpha= \bigcup_{\alpha<\kappa} t_\alpha$. 
\item 
Suppose that $A$ is a subset of ${}^{\kappa}\kappa$. 
A partial strategy $\sigma$ for player I in $G_\kappa(A)$ is \emph{winning} if for every run $\vec{s}=\langle s_\alpha\mid \alpha<\kappa\rangle$ according to $\sigma$, the outcome $\bigcup_{\alpha<\kappa} s_\alpha$ is in $A$. 
\end{enumerate-(i)} 
\end{definition} 

The next result shows that to construct a winning strategy for player I in $G_\kappa(A)$, it is sufficient to construct a winning almost strategy. 
%for player I. 
%\todo{ok, or as definition?} 
In the statement, we call a definition or a formula \emph{$V_\kappa$-absolute} if it is absolute to outer models $W\supseteq V$ with $(V_\kappa)^W=V_\kappa$. 

\begin{lemma} \label{strategy from almost strategy} 
There is a $V_\kappa$-absolute definable function that 
%that is absolute to outer models with the same $V_\kappa$ and that, in every such model, 
maps every almost strategy $\sigma$ for player I in $G_\kappa$ to a strategy $\tau$ that expands $\sigma$ and moreover, this property of $\sigma$, $\tau$ is $V_\kappa$-absolute. 
%\todo{WHAT KIND OF ABSOLUTENESS?} If $\sigma$ is an almost strategy for player I in $G_\kappa$, then there is a strategy $\tau$ for player I in $G_\kappa$ that expands $\sigma$. 
%an expansion $\tau$ of $\sigma$ that is 
%Suppose that $A$ is a subset of ${}^{\kappa}\kappa$ and that $\sigma$ is a winning almost strategy for player I in $G_\kappa(A)$. 
%Suppose that for every run $\vec{s}=\langle s_\alpha\mid \alpha<\kappa\rangle$ following $\sigma$, the outcome $\bigcup_{\alpha<\kappa} s_\alpha$ is in $A$. 
%Then there is a winning strategy $\tau$ for player I in $G_\kappa(A)$. 
\end{lemma} 
\begin{proof} 
We fix a wellordering $\prec$ of $\klk$. 
We will define $\tau$ by induction on the length of partial runs. 
To this end, for any partial run $\vec{t}=\langle t_\alpha\mid \alpha<\gamma\rangle$ that is according to $\tau$, as defined up to this stage, we will define a \emph{revised partial run} $\mathrm{rev}(\vec{t})=\langle r_\alpha\mid \alpha<\gamma\rangle$ according to $\sigma$ with 
$r_\alpha=t_\alpha$ for all even $\alpha<\gamma$ 
and let $\tau(\vec{t})=\sigma(\mathrm{rev}(\vec{t}))$. 

In the successor step, suppose that the construction has been carried out for some even ordinal $\gamma<\kappa$ 
and that $\vec{t}=\langle t_\alpha\mid \alpha<\gamma+2\rangle$ is a partial run. 
If $\vec{t}$ is not according to $\tau$, then we give $\tau(\vec{t})$ the $\prec$-least possible value. 
If $\vec{t}$ is according to $\tau$, then $\vec{t}{\upharpoonright}\gamma$ is according to $\tau$ and hence $\sigma(\mathrm{rev}(\vec{t}{\upharpoonright}\gamma))=\tau(\vec{t}{\upharpoonright}\gamma)=t_\gamma$ by the induction hypothesis for $\gamma$.  
Since $\sigma$ is an almost strategy, there is some $u\supsetneq t_{\gamma+1}$ with $\vec{t}^\smallfrown \langle t_\gamma, u\rangle\in \dom{\sigma}$. 
For the $\prec$-least such $u$, we let 
$$\mathrm{rev}(\vec{t})=\mathrm{rev}(\vec{t}{\upharpoonright}\gamma)^\smallfrown\langle t_\gamma, u\rangle.$$ 

In the limit step, suppose that $\vec{t}=\langle t_\alpha\mid \alpha<\gamma\rangle$ is a partial run of limit length $\gamma<\kappa$ and that the construction has been carried out strictly below $\gamma$. If $\vec{t}$ is not according to $\tau$, then we give $\tau(\vec{t})$ the $\prec$-least possible value. 
If $\vec{t}$ is according to $\tau$, then we let 
$$\mathrm{rev}(\vec{t})=\bigcup_{2\cdot \alpha<\gamma} \mathrm{rev}(\vec{t}{\upharpoonright}2\cdot \alpha).$$ 
Moreover, let $\tau(\vec{t})=\sigma(\mathrm{rev}(\vec{t}))$. 

It is easy to see that the construction of the function and its required properties are absolute to any 
model of set theory with the same $V_\kappa$ that also contains $\prec$. 
\end{proof}

We now collect some definitions that are relevant for the following proofs. 
The subsets $S$ of $\Add{\kappa}{1}^2$ introduced below will 
%We will work with a subset $S$ of $\Add{\kappa}{1}^2$ that 
represent two-step iterated forcings that are sub-equivalent to $\Add{\kappa}{1}$. 
%subforcing of $\Add{\kappa}{1}\times \Add{\kappa}{1}$. 

\begin{definition} \label{perfect subset of kappa^kappa x kappa^kappa}  
A set $S$ is called a \emph{level subset} of $\Add{\kappa}{1}^2$ 
%$({}^{<\kappa}\kappa)^2$} 
%is called a \emph{level subset} 
if it consists of pairs $(s,t)\in\Add{\kappa}{1}^2$ with 
%${}^{<\kappa}\kappa\times {}^{<\kappa}\kappa$ 
$\len(s)=\len(t)$. 
% for all $(s,t)\in S$. 
We further define the following properties, which such a set might have. 
%Suppose that $S$ is a level subset of $({}^{<\kappa}\kappa)^2$. 
%${}^{<\kappa}\kappa\times {}^{<\kappa}\kappa$ 
\begin{enumerate-(a)} 
\item 
$S$ is \emph{closed} if for every strictly increasing sequence $\langle (s_\alpha, t_\alpha) \mid \alpha<\gamma \rangle$ in $S$, there is some $(s,t) \in S$ with 
$s\supseteq \bigcup_{\alpha<\gamma} s_\alpha$ and $t\supseteq \bigcup_{\alpha<\gamma} t_\alpha.$ 
%$s_\alpha\subseteq s$ and $t_\alpha\subseteq t$ for all $\alpha<\gamma$. 
%which extends $(s_\alpha $ for all $\alpha<\gamma$. 
\item 
$S$ is \emph{limit-closed} if for every strictly increasing sequence $\langle (s_\alpha, t_\alpha) \mid \alpha<\gamma \rangle$ in $S$,  
$s=\bigcup_{\alpha<\gamma} s_\alpha$ and $t=\bigcup_{\alpha<\gamma} t_\alpha$, we have $(s,t)\in S$. 
\item 
%\todo{maybe ask for limit-closed?}
$S$ is \emph{perfect} if it is closed and every element of $S$ has incompatible successors in $S$. 
%and for every $t\in S$, there are incompatible nodes above $t$ in $S$. 
%\end{enumerate-(i)} 
%\item 
%$p$ is closed under initial segments, 
%\item 
%$S$ has no maximal elements and 
%\item 
%$S$ is $<\kappa$-closed in the following sense: 
%if $\langle s_\alpha \mid \alpha<\gamma \rangle$ is a strictly increasing sequence in $S$, then there is some $s \in p$ which extends $s_\alpha $ for all $\alpha<\gamma$. 
%$\gamma<\kappa$ and $(s_\alpha, t_\alpha)\in S$ for all $\alpha<\gamma$, 
%then $(s,t)\in S$ for $s=\bigcup_{\alpha<\gamma} s_\alpha$ and $t=\bigcup_{\alpha<\gamma} t_\alpha$. 
%\end{enumerate-(i)} 
%
%Suppose that $S$ is a level subset of $({}^{<\kappa}\kappa)^2$. 
%\begin{enumerate-(i)} 
\end{enumerate-(a)} 
%\item 
Moreover, we let 
%\todo{\ \ \ \ \ \ \ \ used?} 
$\splitt(S)$ denote the set of \emph{splitting nodes}, i.e. the elements of $S$ with incompatible direct successors in $S$, and 
%\item 
%\todo{\ \ \ \ \ \ \ \ used?} 
$\succsplit(S)$ the set of direct successors of splitting nodes. 
%elements of $\splitt(S)$. 
%in $S$ of nodes $(s,t)\in \splitt(S)$. 
\end{definition} 

Note that for subtrees, the notions of closure and limit closure that we have just defined are equivalent. 

The next definitions will be used below to define a forcing that adds a 
%\todo{or just an almost strategy?} 
winning set for player I in $G_\kappa$. 
%We will define sets which will be used to prove the existence of winning strategies in $G_\kappa$. The sets have the structure of a tree. 

\begin{definition} \label{definition: S-tree} 
Suppose that $S$ is a level subset of $\Add{\kappa}{1}^2$. 
%${}^{<\kappa}\kappa\times {}^{<\kappa}\kappa$ 
%\todo{we don't need closure under initial segments, right?}
An \emph{$S$-tree} $p$ consists of pairs $(s,t)$ such that $s$, $t$ are strictly increasing sequences with $\len(s)=\len(t)$ and the following conditions hold for all $(s,t), (u,v)\in p$ and all $\alpha<\len(s)$. 
%if the elements of $p$ are of the form $(s,t)$ and for all $(s,t)\in p$ 
\begin{enumerate-(a)} 
%\item \label{S-node 1}
%$s,t\in {}^{\gamma}({}^{<\kappa}\kappa)$ for some $\gamma<\kappa$, %=\gamma_{s,t}<\kappa$, 
%\item \label{S-node 2}
%$s,t$ are strictly increasing with respect to inclusion, 
% sequences in ${}^{\gamma}({}^{<\kappa}\kappa)$ with $\gamma<\kappa$. 
%\item \label{S-node 3}
%\todo{\ \ \ \ \ \ \ \ used?} $(s(\alpha), t(\alpha))\in \succsplit(S)$. 
% for all $\alpha<\len(s)$, 
%$\alpha<\gamma_{s,t}$, 
\item 
$(s(\alpha), t(\alpha))\in S$. 
\item \label{S-node 4}
$(s{\upharpoonright} \alpha, t{\upharpoonright} \alpha)\in p$. 
% for all $\alpha<\len(s)$, 
%$\alpha<\gamma_{s,t}$ and 
%$p$ is closed under coordinate-wise restriction to ordinals. 
\item \label{S-node 5}
%\todo{CHANGE: Lengths not necessarily the same} 
If $\gamma,\delta<\len(s)$
%$\len(s)=\gamma$ and $\len(u)=\delta$ 
are even, $\bigcup \ran{s{\upharpoonright}\gamma} = \bigcup \ran{u{\upharpoonright}\delta}$ and $\bigcup \ran{t{\upharpoonright}\gamma} = \bigcup \ran{v{\upharpoonright}\delta}$, then $s(\gamma)=u(\delta)$ and $t(\gamma)=v(\delta)$. 
%\item \label{S-node 5}
%%MISTAKEN PREVIOUS VERSION: If $\len(s)=\len(u)=\gamma$ is even, $\bigcup (s\upharpoonright\gamma) = \bigcup (u\upharpoonright\gamma)$ and $\bigcup (t\upharpoonright\gamma) = \bigcup (v\upharpoonright\gamma)$, then $s(\gamma)=u(\gamma)$ and $t(\gamma)=v(\gamma)$. 
%%ORIGINAL DEFINITION: $s\upharpoonright \gamma=u\upharpoonright \gamma$ and $t\upharpoonright \gamma=v\upharpoonright \gamma$ 
\end{enumerate-(a)} 
%The pairs $(s,t)$ with the properties \ref{S-node 1}, \ref{S-node 2} and \ref{S-node 3} are called \emph{$S$-nodes}. 
\end{definition} 

\begin{remark} 
The condition in Definition \ref{definition: S-tree} \ref{S-node 5} can be replaced with the following statement. 
%the conditions $\bigcup \ran{s{\upharpoonright}\gamma} = \bigcup \ran{u{\upharpoonright}\delta}$ and $\bigcup \ran{t{\upharpoonright}\gamma} = \bigcup v{\upharpoonright}\delta$ can be replaced with the conditions 
If $\gamma<\len(s)$ is even, $s{\upharpoonright} \gamma=u{\upharpoonright} \gamma$ and $t{\upharpoonright} \gamma=v{\upharpoonright} \gamma$, then $s(\gamma)=u(\gamma)$ and $t(\gamma)=v(\gamma)$. 
Using this alternative definition, 
%\todo{check below} 
one can prove analogous results to all that follows. 
%results in this section. 
\end{remark} 

%\max \{\sup_{\alpha\in dom(s)} |s(\alpha)|,\ \sup_{\alpha\in dom(t)} |t(\alpha)|\}$ for $p\in\mathbb{P}$. 

%The following fullness condition is used to define an almost strategy for player I in $G_\kappa$ from an $S$-tree. 

The $S$-trees of size ${<}\kappa$ will be the conditions in a forcing that adds an $S$-tree with the following properties. 

\begin{definition} 
Suppose that $S$ is a level subset of $\Add{\kappa}{1}^2$ and $p$ is an $S$-tree. 
%${}^{<\kappa}\kappa\times {}^{<\kappa}\kappa$. 
\begin{enumerate-(a)} 
\item 
Let $\len(p)=\sup_{(s,t)\in p} \len(s)$ and 
$$\he(p)=\sup_{(s,t)\in p, \alpha< \len(s) }  \len(s(\alpha)).$$ 
\item 
%\todo{\ \ \ \ \ check where \ \\ 
%\ \ \ \ \ this is used?} 
An $S$-tree $p$ is called \emph{superclosed} if  
\begin{enumerate-(i)} 
%\item 
%$p$ is closed under initial segments, 
\item 
%$p$ is $<\kappa$-closed, i.e. 
if $\langle (s_\alpha, t_\alpha) \mid \alpha<\gamma \rangle$ is a strictly increasing sequence in $p$, then there is some $(s,t)\in p$ which extends $(s_\alpha, t_\alpha)$ for all $\alpha<\gamma$. 
\item 
$p$ has no maximal elements. 
%$\gamma<\kappa$ and $(s_\alpha, t_\alpha)\in S$ for all $\alpha<\gamma$, 
%then $(s,t)\in S$ for $s=\bigcup_{\alpha<\gamma} s_\alpha$ and $t=\bigcup_{\alpha<\gamma} t_\alpha$. 
\end{enumerate-(i)} 
\item 
%\begin{definition} 
%Suppose that $S$ is a subset of ${}^{<\kappa}\kappa\times {}^{<\kappa}\kappa$. 
An $S$-tree $p$ is called \emph{strategic} if it is superclosed and the following condition holds. 
If $(s,t)\in p$, $\len(s)=\len(t)=\gamma+1$, $\gamma$ is even and $u\supsetneq s(\gamma)$, then there are $v,w\in \klk$ with $v\supseteq u$ and $(s^\smallfrown \langle v\rangle, t^\smallfrown \langle w\rangle)\in p$. 
%\end{definition} 
\end{enumerate-(a)} 
\end{definition} 

%\todo{typesetting?} Since $s$ is strictly increasing for all $(s,t)\in p$, we have 
Note that we have $\len(s)\leq \he(s)$ for all $(s,t)\in p$, since  $s$ is strictly increasing by the definition of $S$-trees. 
We will further work with the following weak projection of superclosed $S$-trees, 
%The $S$-trees can either have size ${<}\kappa$ and these trees are used as forcing 
%we will use the following weak projection 
which differs from the standard notion of projection. 
%Note that the following differs from the standard notation. 

\begin{definition} 
If $S$ is a level subset of $\Add{\kappa}{1}^2$
%${}^{<\kappa}\kappa\times {}^{<\kappa}\kappa$ 
and $T$ is a superclosed $S$-tree, we define the following objects. 
\begin{enumerate-(a)} 
\item 
The \emph{body} $[T]$ of $T$ is the set  of $(x,y)\in \Add{\kappa}{1}^2$
%\times{}^{\kappa}\kappa$ 
such that there are $\vec{s}=\langle s_\alpha\mid \alpha<\kappa \rangle$ and $\vec{t}=\langle t_\alpha\mid \alpha<\kappa \rangle$ with $\langle (s_\alpha, t_\alpha) \mid \alpha<\gamma \rangle \in T$ for all $\gamma<\kappa$ and 
$$x=\bigcup_{\alpha<\kappa} s_\alpha,\ y=\bigcup_{\alpha<\kappa} t_\alpha.$$ 
\item 
The \emph{projection} $\proj[T]$ of $T$ is the set  of $x\in {}^{\kappa}\kappa$ such that $(x,y)\in [T]$ for some $y\in \kk$. 
%there are $\vec{s}=\langle s_\alpha\mid \alpha<\kappa \rangle\}$ and $\vec{t}=\langle t_\alpha\mid \alpha<\kappa \rangle\}$ such that $\langle (s_\alpha, t_\alpha) \mid \alpha<\gamma \rangle \in T$ for all $\gamma<\kappa$ and $x=\bigcup_{\alpha<\kappa} s_\alpha$. 
%\item 
%\todo{delete this} The \emph{sequential body} $\llbracket T \rrbracket$ of $T$ is defined as the set of $(\vec{s},\vec{t})\in {}^{\kappa}({}^{<\kappa}\kappa)\times {}^{\kappa}({}^{<\kappa}\kappa)$ such that $\vec{s}=\langle s_\alpha\mid \alpha<\kappa \rangle\}$, $\vec{t}=\langle t_\alpha\mid \alpha<\kappa \rangle\}$ and $\langle (s_\alpha, t_\alpha) \mid \alpha<\gamma \rangle \in T$ for all $\gamma<\kappa$. 
%\item 
%\todo{delete this} The \emph{sequential projection} $\proj\llbracket T \rrbracket$ of $T$ is defined as the set  of $\vec{s}\in {}^{\kappa}({}^{<\kappa}\kappa)$ such that $\vec{s}=\langle s_\alpha\mid \alpha<\kappa \rangle\}$ and there is some $\vec{t}=\langle t_\alpha\mid \alpha<\kappa \rangle\}$ such that $\langle (s_\alpha, t_\alpha) \mid \alpha<\gamma \rangle \in T$ for all $\gamma<\kappa$. 
\end{enumerate-(a)} 
\end{definition} 

%Let $G_\kappa$ denote the Banach-Mazur game of length $\kappa$. 
%he following will be used to prove the existence of winning strategies in $G_\kappa$. 

The strategic $S$-trees are defined for the following purpose. 

\iffalse 
\begin{lemma} \label{trees and strategies} 
%Suppose that $p$ is a subset of ${}^{<\kappa}\kappa$. 
Suppose that $S$ is a perfect level subset of $({}^{<\kappa}\kappa)^2$ 
%${}^{<\kappa}\kappa\times {}^{<\kappa}\kappa$ 
and $T$ is a strategic $S$-tree. 
%The following conditions are equivalent. 
%\begin{enumerate} 
%\item 
%$p$ is a full $S$-tree. 
%\item
Then there is an almost tactic $\sigma$ for player I in $G^2_\kappa([T])$ such that $\sigma$ is winning in every generic extension of $V$ by $<\kappa$-distributive forcing. 
%Then there is an almost strategy $\sigma$ for player I in $G_\kappa$ such that the following conditions hold in all generic extensions of $V$ via $<\kappa$-distributive forcings. 
\end{lemma} 
\begin{proof} 

\end{proof} 
\fi 

\begin{lemma} \label{trees and strategies} 
Suppose that $S$ is a perfect level subset of $\Add{\kappa}{1}^2$ 
and $T$ is a strategic $S$-tree. 
Then there is a winning strategy for player I in $G_\kappa(\proj[T])$ that remains so in all outer models $W\supseteq V$ with $(V_\kappa)^W=V_\kappa$. 
\iffalse 
\begin{enumerate-(a)} 
\item \label{generic full tree 1} 
Every partial run $\vec{s}$ for $\sigma$ is a strictly increasing sequence in \todo{define above?} $\pro(T)=\{u\in {}^{<\kappa}S\mid \exists v\in {}^{<\kappa}S\ (u,v)\in T\}$. 
\item \label{generic full tree 2} 
$\sigma$ is a winning almost strategy for player I in $G_\kappa(\proj[T])$. 
\end{enumerate-(a)} 
In particular, there is a winning strategy for player I in $G_\kappa(\proj[T])$ by Lemma \ref{strategy from almost strategy}. 
\fi
\end{lemma} 
\begin{proof} 
We fix a wellordering $\prec$ of $\klk$. %$V_\kappa$. 
It is sufficient to construct a winning almost strategy for player I in $G_\kappa(\proj[T])$ by Lemma \ref{strategy from almost strategy}, and this will be done as follows, by induction on $\delta<\kappa$. 
We will define $\sigma$ for partial runs of length strictly below $\delta$, and will simultaneously, for each partial run $\vec{s}$ according to $\sigma$ with 
%\todo{check}
odd length $\len(\vec{s})\leq\delta$, define a sequence $\vec{t}_{\vec{s}}$ with $(\vec{s},\vec{t}_{\vec{s}})\in T$ and $\vec{t}_{\vec{s}{\upharpoonright}\alpha}\subseteq \vec{t}_{\vec{s}}$ for all 
%\todo{check}
odd $\alpha<\len(\vec{s})$. 
%We will define, by simultaneous induction, $\sigma(\vec{s})$ for any run $\vec{s}$ according to $\sigma$ of even length, and $\vec{t}_{\vec{s}}$ for any run $\vec{s}$ according to $\sigma$ of odd length, with $(\vec{s},\vec{t}_{\vec{s}})\in T$ and $\vec{t}_{\vec{s}{\upharpoonright}\alpha}\subseteq \vec{t}_{\vec{s}}$ for all odd $\alpha<\len(\vec{s})$. 

In the successor step, we assume that the construction has been carried out up to $\delta=2\gamma+1$ for some $\gamma<\kappa$ and that 
$\vec{s}=\langle s_\alpha\mid \alpha<\delta\rangle$ 
is a partial run 
%\todo{"compatible with" everywhere?} 
according to $\sigma$. 
Let 
$$\Psi(u)\Longleftrightarrow u \supseteq s_{2\gamma}\text{ and } \exists v\supseteq \vec{t}_{\vec{s}}(2\gamma)\ (\vec{s}^\smallfrown\langle u\rangle, \vec{t}_{\vec{s}}^\smallfrown\langle v\rangle)\in T\}$$ 
Since $T$ is 
strategic, the set 
$D=\{u\mid \Psi(u)\}$ is dense above $s_{2\gamma}$, in the sense that for every $u\supseteq s_{2\gamma}$, there is some $v\supseteq u$ with $\Psi(v)$. 
Since we are constructing an almost strategy, it is sufficient to define $\sigma(\vec{s}^\smallfrown\langle u\rangle)$ for all $u\in D$. 

Given $u\in D$, let $v\supseteq \vec{t}_{\vec{s}}(2\gamma)$ be $\prec$-least with $(\vec{s}^\smallfrown\langle u\rangle, \vec{t}_{\vec{s}}^\smallfrown\langle v\rangle)\in T\}$. 
%Since $T$ is full, for densely many $u\supseteq s_{2\gamma}$, there is some $v\supseteq \vec{t}_{\vec{s}}(2\gamma)$ with $(\vec{s}^\smallfrown\langle u\rangle, \vec{t}_{\vec{s}}^\smallfrown\langle v\rangle)\in T$, and we can choose the $\prec$-least $v$ for each such $u$. Since we would like to define an almost strategy, it is sufficient to define $\sigma(\vec{s}^\smallfrown\langle u\rangle)$ for each such $u$. 
Since $T$ is an $S$-tree and by Definition \ref{definition: S-tree} \ref{S-node 5}, there is a unique pair $(u^*,v^*)$ with 
$$(\vec{s}{}^\smallfrown\langle u,u^*\rangle, \vec{t}_{\vec{s}}{}^\smallfrown\langle v,v^*\rangle)\in T.$$ 
Now let $\sigma(\vec{s}^\smallfrown\langle u\rangle)=u^*$ and 
$\vec{t}_{\vec{s}^\smallfrown\langle u, u^* \rangle}=\vec{t}_{\vec{s}}^\smallfrown\langle v,v^*\rangle$. 

In the limit step, we assume that the construction has been carried out strictly below $\gamma$ for some $\gamma\in\Lim$ and that 
$\vec{s}=\langle s_\alpha\mid \alpha<\gamma\rangle$ 
is a partial run according to $\sigma$. 

We first let $\vec{t}=\bigcup_{\alpha<\gamma}\vec{t}_{\vec{s}{\upharpoonright}2\alpha+1}$. 
Since $T$ is superclosed, there is a pair $(u,v)$ with $(\vec{s}^\smallfrown\langle u\rangle,\vec{t}^\smallfrown\langle v\rangle)\in T$, and moreover this pair is unique, since $T$ is an $S$-tree and by Definition \ref{definition: S-tree} \ref{S-node 5}. 
Let $\sigma(\vec{s})=u$ and 
$\vec{t}_{\vec{s}^\smallfrown\langle u \rangle}=\vec{t}_{\vec{s}}^\smallfrown\langle v\rangle$. 

This completes the construction of $\sigma$. To prove that $\sigma$ wins, suppose that $\vec{s}=\langle s_\alpha\mid \alpha<\kappa \rangle$ is a run according to $\sigma$ and let $\vec{t}=\bigcup_{\alpha<\kappa}\vec{t}_{\vec{s}{\upharpoonright}2\alpha+1}$. 
Then $\langle (\vec{s}{\upharpoonright}2\alpha+1, \vec{t}_{\vec{s}{\upharpoonright}2\alpha+1})\mid \alpha<\kappa\rangle$ witnesses that the outcome $\bigcup_{\alpha<\kappa}s_\alpha$ is in $\proj[T]$ and hence player I wins, proving the claim. 
\end{proof} 

%\todo[inline]{state that "closed" (define) is sufficient instead of "limit-closed"} 

\begin{definition} \label{definition of P} 
Suppose that $S$ is a perfect level subset of $\Add{\kappa}{1}^2$. 
The forcing $\PP_S$ consists of all $S$-trees of size strictly less than $\kappa$, ordered by reverse inclusion. 
\end{definition} 

If $G$ is a $\PP_S$-generic filter over $V$, we will write $T_G=\bigcup G$. 
Moverover, for any perfect level subset $S$ of $\Add{\kappa}{1}^2$, we will write $\pi_S\colon S\rightarrow \Add{\kappa}{1}$ for the projection to the first coordinate.

In the situation below, we will additionally assume that 
%$S$ is a perfect level subset of $\Add{\kappa}{1}^2$ such that 
$\pi_S\colon S\rightarrow \Add{\kappa}{1}$ is a projection. 
It is then easy to see that the forcing $\PP_S$ is non-atomic, ${<}\kappa$-closed and has size $\kappa$, and is hence sub-equivalent to $\Add{\kappa}{1}$ by Lemma \ref{forcing equivalent to Add(kappa,1)}. 

%If $G$ is $\PP$-generic over $V$, let $T_G=\bigcup G$. 

\begin{lemma} \label{the generic tree is full} 
%\todo{actually this only proves something weaker than full, is it sufficient? fullness seems to correspond to strong projection. do we get this too?\ \\  maybe re-define fullness?} 
If $S$ is a perfect level subset of $\Add{\kappa}{1}^2$ such that $\pi_S\colon S\rightarrow \Add{\kappa}{1}$ is a projection and $G$ is $\PP_S$ -generic over $V$, then $T_G$ is a 
%\todo{notation: strategic perfect tree?} 
strategic $S$-tree. 
\end{lemma} 
\begin{proof} 
Since every condition in $\PP_S$ is an $S$-tree, it follows immediately that $T_G$ is again an $S$-tree. 
Moreover, since $S$ is perfect, it can be shown by a straightforward density argument that $T_G$ is superclosed. 
%Since $S$ is perfect, every condition $p$ can be extended to some $q\leq p$ so that any terminal node in $p$ is extended to incompatible nodes in $q$. Similarly, we can add an upper bound in $q$ to a given cofinal branch in $p$ of length $<\kappa$. Hence genericity of $G$ implies that $T_G$ is perfect. 

To see that $T_G$ is strategic, suppose that $(s,t)\in T_G$, $\len(s)=\len(t)=\gamma+1$, $\gamma$ is even and $u\supsetneq s(\gamma)$. 
Then there is some $p\in G$ with $(s,t)\in p$. 
Since $\pi_S\colon S\rightarrow \Add{\kappa}{1}$ is a projection by our assumption, there is some $(v,w)\in S$ with $u\subseteq v$. 
We now claim that the set 
$$D=\{q
%\in \PP_S\mid q
\leq p\mid 
%\exists (v,w)\in \Add{\kappa}{1}^2\ v\supseteq u\ \&\ 
(s^\smallfrown \langle v\rangle, t^\smallfrown \langle w\rangle)\in q\}$$ 
is dense below $p$. 
To see this, suppose that $q\leq p$. 
Since $\gamma$ is even, it is easy to check that $q\cup\{s^\smallfrown \langle v\rangle,t^\smallfrown \langle w\rangle\}$ is again a condition in $\PP_S$, and thus $D$ is dense below $p$. It follows immediately that $T_G$ is strategic. 
%then there are $v,w\in \klk$ with $v\supseteq u$ and $(s^\smallfrown \langle v\rangle, t^\smallfrown \langle w\rangle)\in p$. 
\end{proof} 
%\todo{move to earlier as Definition} 
%If $\QQ\lessdot \RR$, we define a \emph{nice projection} $\pi\colon \RR\rightarrow \QQ$ to mean that $\pi$ is a projection such that $\pi(q)\geq q$ for all $q\in \QQ$ and $\pi{\upharpoonright}\QQ=\mathrm{id}_\QQ$. 

%In the application below, the set is $S$ derived from a complete subforcing of $\BB(\Add{\kappa}{1})$ as in the next result. In the next proof, we will use the following fact, which follows immediately from the definition of quotient forcings. If $\PP$, $\QQ$ are forcings, $\PP$ is separative and $\pi\colon\QQ\rightarrow \PP$ is a projection, then for all $(p,q)\in \PP\times \QQ$, $$(p,\check{q})\in \PP*(\QQ/\PP)^\pi \Longleftrightarrow p\leq\pi(q).$$ 

%Since $\Add{\kappa}{1}$ is separative, we have that for all $(u,v)\in\Add{\kappa}{1}^2$ $$(\iota(u),\check{v})\in \QQ*\dot{\QQ} \Longleftrightarrow \iota(u)\leq\pi\nu(v)$$ by the definition of $\dot{\QQ}$ as a quotient forcing. 

%\todo[inline]{IN THE NEXT PROOF, MAYBE $\iota$ NEEDS TO BE A SPECIFIC SUB-ISOMORPHISM, AND THE STATEMENT NEEDS TO BE CHANGED} 

In the next lemma, we will write $\QQ_p$ for the subforcing 
$$\QQ_p=\{q\in\QQ\mid q\leq p\}$$ 
of a forcing $\QQ$ below a condition $p\in\QQ$. 

\begin{lemma} \label{two step iteration coded by a set S}
%\todo{\ \ \ \ \ consistent with \ \\ \ \ \ \ \ application \ \\ \ \ \ \ \ below?}
Suppose that $\RR$ is a complete Boolean algebra and $\QQ$ is a complete subalgebra such that $\QQ$, $\RR$, $\Add{\kappa}{1}$ are sub-equivalent. 
Moreover, suppose that $p\in\QQ$, $r\in \Add{\kappa}{1}$ and $\iota\colon \Add{\kappa}{1}_r\rightarrow \QQ_p$ is a sub-isomorphism. 
%Then there is a sub-isomorphism $\iota\colon \Add{\kappa}{1}\rightarrow \QQ$ and 
Then there is 
a perfect limit-closed level subset $S$ of $\Add{\kappa}{1}_r^2$ 
such that 
%\footnote{See Definition \ref{definition: equivalent and similar forcings} and Definition \ref{pull back names}.} 
\label{two step iteration coded by a set S 2} 
$\pi_S$ is a projection and 
%\item 
\label{two step iteration coded by a set S 3} 
%$$\Vdash_{\QQ} \RR/\QQ \simeq (S/\Add{\kappa}{1})^{\pi_S}.$$ 
%\end{enumerate-(a)} 
%\todo{define notation above} 
%\todo{check notation with iota} 
$$\Vdash_{\Add{\kappa}{1}_r} \RR_p/\QQ_p^{(\iota)} \simeq S/\Add{\kappa}{1}_r^{\pi_S}.$$ 
%In particular, $S$ ordered by reverse inclusion and $\QQ*\dot{\QQ}$ have the same generic extensions. 
%is \todo{DEFINE THIS} $2$-step equivalent equivalent to $\QQ*\dot{\QQ}$. 
% (see Definition \ref{definition: equivalent and similar forcings}). 
%In particular, $\RR*\dot{\RR}$ is equivalent to $\QQ*\dot{\QQ}$ by \ref{R is equal to Q^pi} and to $S$ ordered by reverse inclusion by \ref{S is dense in R*dot(R)}. 
\end{lemma} 
\begin{proof} 
%\todo{\ \ \ \ \ rewrite because of \ \\ \ \ \ \ \ $\iota$? }
Since $\Add{\kappa}{1}_r$ is isomorphic to $\Add{\kappa}{1}$, we can assume that $r=\one_{\Add{\kappa}{1}}$ and $p=\one_\QQ$. 
Let $\QQ_0=\iota[\QQ]$ (note that $\iota$ necessarily preserves infima) 
%There is a sub-isomorphism of $\Addd{\kappa}{1}$ into $\QQ$ by Lemma \ref{forcing equivalent to Add(kappa,1)}. 
%In fact, since $\QQ$ is a complete Boolean algebra, it is easy to see that there is a sub-isomorphism $\iota\colon\Add{\kappa}{1}\rightarrow \QQ$ onto a dense subset $\QQ_0$ of $\QQ$ that additionally preserves infima. 
%We further
and fix an arbitrary sub-isomorphism $\nu\colon\Add{\kappa}{1}\rightarrow \RR$. 
Moreover, we let $\pi\colon\RR\rightarrow \QQ$ denote the natural projection as given in Definition \ref{definition of natural projection}. 
Since $\pi(r)\geq r$ for all $r\in \RR$, it is then easy to show that $\RR/\QQ= \RR/\QQ^\pi$. 

Since $\pi$, $\nu$ are projections, it follows that $\pi \nu\colon \Add{\kappa}{1}\rightarrow \QQ$ is also a projection. 
Hence we can define 
$$\dot{\QQ}=\Add{\kappa}{1}/\QQ^{\pi \nu}.$$ 
%Since 
%\todo{CHECK!!!!!!!!!!!!!!!}$$
%\Vdash_{\QQ} [(\QQ*\dot{\QQ})/\QQ]^{\pi_{\QQ*\dot{\QQ}} } = \dot{\QQ}$$ 
%and by the choice of $\nu$, 
Moreover, since $\nu$ is a sub-isomorphism, $\QQ$ forces that $\nu\colon \dot{\QQ}\rightarrow (\RR/\QQ)^{\pi}$ is a sub-isomorphism. 
%we have 
%$\Vdash_{\QQ}\ddot{\QQ}\sim\dot{\QQ}$. 
Thus by Lemma \ref{equivalence of forcings is transitive}, it is sufficient 
%\todo{by the previous lemma (cite?)} 
to prove the existence of a set $S$ as above with  
$$\Vdash_{\Add{\kappa}{1}} \dot{\QQ}^{(\iota)} \simeq S/\Add{\kappa}{1}^{\pi_S}$$ 
and we will prove this in the following claims. 

We will write $\mathrm{Lim}$ for the class of limit ordinals. 
%Let \begin{center}  
%$S=\{(s,t)\in \Add{\kappa}{1}^2 \mid \len(s)=\len(t)\in\mathrm{Lim},\ \inf^\RR_{\alpha<\len(s)} s{\upharpoonright}\alpha = \inf^\RR_{\alpha<\len(s)} \pi \nu(t{\upharpoonright} \alpha)\}$ \end{center} 
%\forall \alpha<\len(s)\ \exists \beta<\len(s)\ s{\upharpoonright} \beta\Vdash_{\QQ} \check{t}{\upharpoonright} \alpha\in \ddot{\QQ}\}.$$
For any pair $(s,t)\in \Add{\kappa}{1}^2$ with $\len(s)=\len(t)\in\mathrm{Lim}$, we further say that 
%$$\vec{s}=
$$\langle (s_\alpha, t_\alpha)\mid \alpha<\cof{\len(s)}\rangle$$ 
% \vec{t}=\langle t_\alpha\mid \alpha<\cof{\len(s)}\rangle$$ 
is an \emph{intertwined sequence for $(s,t)$} if 
$$s=\bigcup_{\alpha<\cof{\len(s)}}s_\alpha,\ \  t= \bigcup_{\alpha<\cof{\len(s)}}t_\alpha$$ 
and $\pi\nu(t_{\alpha+1})\leq\iota(s_\alpha)\leq \pi\nu(t_\alpha)$ for all for all $\alpha<\cof{\len(s)}$. 
% and $\iota(s_{\alpha+1})\leq \pi\nu(t_\alpha)$. 
%$\pi(t{\upharpoonright}\beta)\leq s{\upharpoonright}\alpha$ for all $\alpha\leq\beta<\cof{\len(s)}$ and $s{\upharpoonright}\beta\leq \pi(t{\upharpoonright}\alpha)$ for all $\alpha<\beta<\cof{\len(s)}$. 
We now consider the subset $S$ of $\Add{\kappa}{1}^2$ that consists of all pairs $(s,t)\in \Add{\kappa}{1}^2$ with $\len(s)=\len(t)\in\mathrm{Lim}$ such that there is an intertwined sequence for $(s,t)$. 
%\todo{need?} Let $\SSS=\{(s,\check{t})\mid (s,t)\in S\}$. 

\begin{claim*} 
For every $(s,t)\in S$, there is some $(u,v)\leq (s,t)$ with  
$(\iota(u),\check{v})\in\QQ_0*\dot{\QQ}$. 
\end{claim*} 
\begin{proof} 
Suppose that 
%$\vec{s}=
$\langle (s_\alpha, t_\alpha)\mid \alpha<\cof{\len(s)}\rangle$ 
% $\vec{t}=\langle t_\alpha\mid \alpha<\cof{\len(s)}\rangle$ are 
is an intertwined sequence for $(s,t)$. 
%We choose any $v\leq t$ in $\Addd{\kappa}{1}$. 
% and let $u=\pi\nu(v)$. 
%\todo{write in the intro: we use Kunen's 2-step iteration} 
%Since $\vec{s}$, $\vec{t}$ are intertwined and 
Since $\pi\nu$ is order-preserving, we have 
$$\pi\nu(t)\leq \pi\nu(t_{\alpha+1})\leq \iota(s_\alpha)$$ 
for all $\alpha<\cof{\len(s)}$ 
and hence $\pi\nu(t)\leq\iota(s)$ by the assumption that $\iota$ preserves infima. 

Since $\QQ_0$ is dense in $\QQ$, there is some $u\in\Add{\kappa}{1}$ with $\iota(u)\leq\pi\nu(t)$. 
Then 
$$\iota(u)\leq\pi\nu(t)\leq\iota(s)$$ 
and since $\iota$ is a sub-isomorphism, this implies that $u\leq s$ and hence $(u,t)\leq(s,t)$. 
Thus by the remark before the claim, $(u,t)$ witnesses the conclusion of the claim. 
%Since $(\pi\nu(t),\check{t})\in\Add{\kappa}{1}*\dot{\QQ}$ by the definition of $\dot{\QQ}$ as a quotient forcing, this proves the claim. 
\end{proof} 

\begin{claim*} 
For every $(u,v)\in\Add{\kappa}{1}^2$ with $(\iota(u),\check{v})\in \QQ_0*\dot{\QQ}$, there is some $(s,t)\leq (u,v)$ in $S$. 
\end{claim*} 
\begin{proof} 
We can assume that $\len(u)>\len(v)$ by extending $u$. 
%Suppose that $\langle \gamma_\alpha\mid \alpha<\cof{\len(s)}\rangle$ is cofinal in $\len(s)$. 
We will construct an intertwined sequence 
%$\vec{s}=
$\langle (s_n, t_n)\mid n<\omega\rangle$
% $\vec{t}=\langle t_n\mid n<\omega\rangle$ in $\Add{\kappa}{1}^2$ 
by induction.
%$(s_n,t_n)$ by induction on $n<\omega$. 
%Since $\QQ$ is separative, for all $(u,v)\in \QQ\times \Add{\kappa}{1}$, we have $(u,\check{v})\in  \Add{\kappa}{1}*\dot{\QQ}$ if and only if $u\leq \pi(v)$. 

We choose $(s_0,t_0)=(u,v)$, so that $\iota(s_0)\leq\pi\nu(t_0)$ by the remark before the first claim. 
Now suppose that we have already constructed $(s_n,t_n)$ with $\iota(s_n)\leq\pi\nu(t_n)$. 
Since $\pi\nu$ is a projection, there is some $t_{n+1}\leq t_n$ with $\pi\nu(t_{n+1})\leq \iota(s_n)$, and we can further assume that $\len(t_{n+1})>\len(s_n)$. 
Moreover, since $\QQ_0$ is dense in $\QQ$, there is some $s_{n+1}\in\Add{\kappa}{1}$ with $\iota(s_{n+1})\leq\pi\nu(t_n)$, and we can further assume that $\len(s_{n+1})>\len(t_{n+1})$. 
Then 
$$\iota(s_{n+1})\leq\pi\nu(t_n)\leq\iota(s_n)$$ 
and since $\iota$ is a sub-isomorphism, this implies that $s_{n+1}\leq s_n$ and hence $(s_{n+1},t_{n+1})\leq(s_n,t_n)$. 
%For $s=\bigcup_{n<\omega}s_n$, $t=\bigcup_{n\in\omega}t_n$, it follows from the construction that $\langle s_n\mid n\in\omega\rangle$, $\langle t_n\mid n\in\omega\rangle$ are intertwined sequences for $(s,t)$ and hence $(s,t)\in S$. 

Letting $s=\bigcup_{n<\omega}s_n$, $t=\bigcup_{n\in\omega}t_n$, we have $\len(s)=\len(t)$ and 
%$\vec{s}$, $\vec{t}$ are 
there is an intertwined sequence for $(s,t)$ by the construction. 
Thus $(s,t)\leq(u,v)$ and $(s,t)\in S$. 
\end{proof} 

Since $\QQ_0$ is non-atomic, it follows immediately from the two previous claims that $S$ is perfect. Moverover, since the projection onto the first coordinate of $\QQ*\dot{\QQ}$ is a projection in the sense of Definition \ref{definition: quotient forcing}, the claims show that $\pi_S\colon S\rightarrow \Add{\kappa}{1}$ is also a projection.

\begin{claim*} 
$S$ is limit-closed. 
\end{claim*} 
\begin{proof} 
Suppose that $\langle (s_\alpha, t_\alpha)\mid \alpha<\cof{\gamma}\rangle$ is a strictly increasing sequence in $S$ and 
$$s=\bigcup_{\alpha<\cof{\len(s)}}s_\alpha,\ \  t= \bigcup_{\alpha<\cof{\len(s)}}t_\alpha.$$ 
%such that 
%$\len(s_\alpha)=\len(t_\alpha)=\gamma_\alpha$ and 
For each $\alpha<\cof{\gamma}$, we choose an element $(u_\alpha,v_\alpha)$ of an intertwined sequence for $(s_{\alpha+1},t_{\alpha+1})$ 
%We choose someof such a sequence 
%an intertwined sequence for $(s_{\alpha+1},t_{\alpha+1})$ 
with $\len(u_\alpha)>\len(s_\alpha)$. 
% for each $\alpha<\cof{\gamma}$. 
It follows that $\langle (u_\alpha, v_\alpha)\mid \alpha<\cof{\len(\gamma)}\rangle$ is an intertwined sequence for $(s,t)$. 
\end{proof} 

%\todo[inline]{read the rest of the proof again} 

\begin{claim*}
$\Vdash_{\Add{\kappa}{1}} \dot{\QQ}^{(\iota)} \simeq (S/\Add{\kappa}{1})^{\pi_S}$. 
\end{claim*}
\begin{proof} 
We consider the forcing 
$$T=\{(s,t)\in\Add{\kappa}{1}\mid (s,t)\in S\text{ or }(\iota(s),\check{t})\in \QQ_0*\dot{\QQ}\}.$$ 
We first claim that $\Add{\kappa}{1}$ forces that $S/\Add{\kappa}{1}^{\pi_S}$ is a dense subforcing of $T/\Add{\kappa}{1}^{\pi_T}$. 
To prove this, assume that $G$ is $\Add{\kappa}{1}$-generic over $V$ and 
$$(s,t)\in[S/\Add{\kappa}{1}^{\pi_S}]^G,$$ 
%\todo{better??} By the remark before the first claim, this implies that 
so that $\pi_S(s,t)=s\in G$. 
Since $\pi_S$ is a projection and by the claims above, the set 
$$D=\{u\leq s\mid \exists v\ (u,v)\leq (s,t),\ (\iota(u),\check{v})\in\QQ_0*\dot{\QQ}\}$$ 
is dense below $s$ in $\Add{\kappa}{1}$. 
Letting $u\in G\cap D$, there is some $v$ with $(u,v)\leq (s,t)$ and $(\iota(u),\check{v})\in\QQ_0*\dot{\QQ}$. 
Since $(\iota(u),\check{v})\in\QQ_0*\dot{\QQ}$ and $\QQ_0$ is separative, we have $\iota(u)\leq \pi\nu(v)$ by the definition of $\QQ_0$. 
%the remarks before the first claim. 
We now write $G^{(\iota)}$ for the upwards closure of $\iota[G]$ in $\QQ$. 
Since 
%$p\leq u$ is in $G$, we have 
$u\in G$, we have $\iota(u)\in G^{(\iota)}$, $ \pi\nu(v)\in G^{(\iota)}$ and hence 
$$(u,v)\in [\dot{\QQ}^{(\iota)}]^G=[\Add{\kappa}{1}/\QQ^{\pi \nu}]^G.$$ 

An analogous argument shows that $\Add{\kappa}{1}$ also forces that $\dot{\QQ}^{(\iota)}$ is a dense subforcing of $T/\Add{\kappa}{1}^{\pi_T}$. 
%proving the claim. 
\end{proof} 

The last claim completes the proof of Lemma \ref{two step iteration coded by a set S}. 
\end{proof} 

We now fix a perfect level subset $S$ of $\Add{\kappa}{1}^2$ such that $\pi_S\colon S\rightarrow \Add{\kappa}{1}$ is a projection and let $\PP=\PP_S$. Since $S$ is perfect, it is easy to see that $\PP$ is a non-atomic ${<}\kappa$-closed forcing of size $\kappa$ and hence $\PP$ and $\Add{\kappa}{1}$ are sub-equivalent by Lemma \ref{forcing equivalent to Add(kappa,1)}. 

In the remainder of this section, we will consider $\PP$-names $\dot{f}$, $\dot{g}$ such that 
%$\one_\PP$ forces that $\dot{f},\dot{g}\colon \kappa\rightarrow \klk$ are strictly increasing and 
%$\dot{f}$, $\dot{g}$ are $\PP$-names with 
$$\one_\PP\Vdash_{\PP} \dot{f},\dot{g}\colon \kappa\rightarrow \klk,\ \forall \alpha<\kappa\ (\dot{f}{\upharpoonright}\alpha, \dot{g}{\upharpoonright}\alpha)\in T_{\dot{G}},$$ 
where $\dot{G}$ is a fixed name for the $\PP$-generic filter. 
We will call such pairs $(\dot{f},\dot{g})$ \emph{adequate} and will always 
%\todo{IS THE NOTATION BELOW OK?} 
assume below that $(\dot{f},\dot{g})$, $(\dot{h},\dot{k})$ are such pairs. 
%for all $\alpha<\kappa$ 
%$\one\Vdash_{\PP}(\dot{f},\dot{g})\in \llbracket T_{\dot{G}}\rrbracket$. 
%and that $\dot{b}$, $\dot{c}$ are $\PP$-names with 
%$$\one\Vdash_{\PP} \bigcup_{\alpha<\kappa} \dot{f}(\alpha)=\dot{b},\ \one\Vdash_{\PP} \bigcup_{\alpha<\kappa} \dot{g}(\alpha)=\dot{c}.$$ 

The aim of the next lemmas is to show that for any adequate pair $(\dot{f},\dot{g})$, there is a dense subforcing of $\PP$ that projects onto a forcing for adding $\bigcup\ran{\dot{f}}$, $\bigcup\ran{\dot{g}}$ with a nice quotient forcing. 
This follows a similar line of reasoning as the arguments for the perfect set property in Section \ref{subsection perfect set property}. 
%$\bigcup\ran{\dot{f}}$ is added by a \todo{check this} complete subforcing of $\PP$ that has a good quotient. 

\begin{definition} \label{definition of P star} 
%\todo{$\DD$ really depends on $\dot{b}$, $\dot{c}$, $\dot{f}$, $\dot{g}$} 
Let $\PP^*_{\dot{f},\dot{g}}$ be the subforcing of $\PP$ consisting of the conditions $p$ such that the following statements hold for some $\gamma_p<\kappa$ and some $f_p, g_p\in {}^{<\kappa}\Add{\kappa}{1}$. 
%with the properties 
\begin{enumerate-(a)} 
\item \label{condition for P*-1} 
$\len(p)=\he(p)=\gamma_p\in \Lim$. 
\item \label{condition for P*-2} 
$p\Vdash_\PP \dot{f}{\upharpoonright} \gamma_p= f_p,\ \dot{g}{\upharpoonright} \gamma_p= g_p$. 
\item \label{condition for P*-3} 
$(f_p{\upharpoonright}\alpha,g_p{\upharpoonright}\alpha)\in p$ for all $\alpha<\gamma_p$. 
%\item \label{condition for P*-3} 
%$\forall \alpha<\gamma_p\ (f_p{\upharpoonright}\alpha,g_p{\upharpoonright}\alpha)\in p.$ 
\end{enumerate-(a)} 
Let further $\PP^\diamond_{\dot{f},\dot{g}}$ be the subforcing of $\PP$ 
%^*_{\dot{f},\dot{g}}$ 
consisting of the conditions $p$ that satisfy requirements \ref{condition for P*-1} and \ref{condition for P*-2}. 
Moreover, let $s_p=\bigcup\ran{f_g{\upharpoonright}\gamma_p}$ and $t_p=\bigcup\ran{g_p{\upharpoonright}\gamma_p}$ for any $p\in \PP^\diamond_{\dot{f},\dot{g}}$. 
\end{definition} 

We will also denote the corresponding values for an adequate pair $(\dot{h},\dot{k})$ and any $q\in \PP^\diamond_{\dot{h},\dot{k}}$ by $h_p, k_p\in{}^{<\kappa}\Add{\kappa}{1}$ and $u_p, v_p\in \klk$. 

\begin{lemma} \label{P-diamond is dense}
$\PP^\diamond_{\dot{f},\dot{g}}\cap \PP^\diamond_{\dot{h},\dot{k}}$ is a dense subforcing of $\PP$. 
\end{lemma} 
\begin{proof} 
%To prove that $\DD$ is dense in $\PP$
Note that in general, we have $\len(p)\leq\he(p)$ for all $p\in\PP$ by the definition of the length and the height. 
To prove the claim, we assume that $p$ in $\PP$ and construct a strictly decreasing sequence $\langle p_n\mid n\in\omega\rangle$ in $\PP$ with $p_0=p$ as follows. 

If $p_n$ is defined and $\he(p_n)=\alpha$, we choose a condition $p_{n+1}$ with $\len(p_{n+1})>\alpha$ that decides $\dot{f}{\upharpoonright}\alpha$, $\dot{g}{\upharpoonright}\alpha$, $\dot{h}{\upharpoonright}\alpha$ and $\dot{k}{\upharpoonright}\alpha$. 
%$\dot{b}{\upharpoonright}\alpha$ and $\dot{c}{\upharpoonright}\alpha$. 
Then $p^\diamond=\bigcup_{n\in\omega} p_n$ is a condition in $\PP$ with $p^\diamond\leq p$ that satisfies requirements \ref{condition for P*-1} and \ref{condition for P*-2} in Definition \ref{definition of P star} for both $(\dot{f},\dot{g})$ and $(\dot{h},\dot{k})$, and thus $p^\diamond\in \PP^\diamond_{\dot{f},\dot{g}}\cap \PP^\diamond_{\dot{h},\dot{k}}$. 
%$\PP^{\diamond}_{\dot{f},\dot{g}}\cap \PP^{\diamond}_{\dot{h},\dot{k}}$ with $p^\diamond\leq p$. 
%Then $p\in \PP^*$. 
%Choose $p_{\alpha+1}\leq p_{\alpha}$ decicing $\dot{f}_b\upharpoonright |p_{\alpha}|$, $\dot{b}\upharpoonright |p_{\alpha}|$ with $|p_{\alpha}|<|p_{\alpha+1}|$ for $\alpha<\omega^2$ and let $p_{\beta}=\bigcup_{\alpha<\beta} p_{\alpha}$ for limits $\alpha<\omega^2$. Let $p=\bigcup_{\alpha<\omega^2} p_{\alpha}$. Then $p\leq p_0$ and $p\in\mathbb{D}$. 
\end{proof} 
%\begin{claim*} 
%$\PP^*_{\dot{f},\dot{g}}\cap \PP^*_{\dot{h},\dot{k}}$ is dense in $\PP$. 
%\end{claim*} 

Using the following lemma, we will see that $\PP^*_{\dot{f},\dot{g}}$ is also a dense subforcing of $\PP$. 

\begin{lemma} \label{values decided by P-diamond}
Suppose that $p$ is a condition in $\PP^\diamond_{\dot{f},\dot{g}}$ and $\beta,\gamma\leq\gamma_p$ are even. 
Moreover, suppose that $q\leq p$ is a condition in $\PP$ and $(s,t)\in q$ with $\len(s)=\len(t)>\beta$ and 
%the following conditions. 
%\begin{enumerate-(i)} 
%\item \label{properties of ranges 1} 
\begin{center} 
$\bigcup \ran{s{\upharpoonright}\beta} = \bigcup \ran{f_p{\upharpoonright}\gamma},\ \bigcup \ran{t{\upharpoonright}\beta} = \bigcup \ran{g_p{\upharpoonright}\gamma}.$ 
\end{center} 
%\item \label{properties of ranges 2} 
Then %$s(\beta)=f_p(\gamma)$, $t(\beta)=g_p(\gamma)$ and 
$q\Vdash_{\PP} \dot{f}(\gamma)=s(\beta),\ \dot{g}(\gamma)=t(\beta).$ 
%\end{enumerate-(i)} 
\end{lemma} 
\begin{proof} 
%We fix an even ordinal $\beta<\gamma_p$ and a
We assume that $G$ is any $\PP$-generic filter over $V$ with $q\in G$ and let  
$(u,v)=
%(f_p{\upharpoonright}\alpha+1,g_p{\upharpoonright}\alpha+1)=
(\dot{f}^G{\upharpoonright}\gamma+1,\dot{g}^G{\upharpoonright}\gamma+1)$. 
Since $(\dot{f},\dot{g})$ is an adequate pair, it follows that  
%before Definition \ref{definition of P star}, we have 
$(u,v)\in T_G$. 
%Moreover, we have  $\bigcup \ran{s{\upharpoonright}\beta} = \bigcup \ran{u{\upharpoonright}\gamma}$ and $\bigcup \ran{t{\upharpoonright}\beta} = \bigcup \ran{v{\upharpoonright}\gamma}$ by the choice of $(u,v)$ and by the condition \ref{properties of ranges 1}. 
Thus $(s,t), (u,v)$ are elements of the same $S$-tree $T_G$ and therefore $s(\beta)=u(\gamma)$ and $t(\beta)=v(\gamma)$ by Definition \ref{definition: S-tree} \ref{S-node 5}, as required. 
%and thus \ref{properties of ranges 2} holds. 
\end{proof}

\begin{lemma} \label{dense subforcing of P_S} 
%\todo{$\DD$ really depends on $\dot{b}$, $\dot{c}$, $\dot{f}$, $\dot{g}$} 
$\PP^*_{\dot{f},\dot{g}}\cap \PP^*_{\dot{h},\dot{k}}$ is a dense subforcing of $\PP$. 
%=\DD_{\dot{f},\dot{g}}$ 
\end{lemma} 
\begin{proof} 
%In the first step of the proof, let $\PP^\diamond_{\dot{f},\dot{g}}$ denote the set of conditions $p$ in $\PP$ that satisfy requirements \ref{condition for P*-1} and \ref{condition for P*-2}. 

We will derive the conclusion from the next claim. 
%Using the next claim, we will see that the set of conditions that satisfy requirement \ref{condition for P*-3} for $(\dot{f},\dot{g})$ and $(\dot{h},\dot{k})$ is dense in $\PP^\diamond_{\dot{f},\dot{g}}\cap \PP^\diamond_{\dot{h},\dot{k}}$. 

\begin{claim*} \label{extension of condition by initial segments of branches}
For any condition $p\in \PP^\diamond_{\dot{f},\dot{g}}$, we have that 
$p\cup\{(f_p{\upharpoonright}\alpha,g_p{\upharpoonright}\alpha)\mid\alpha<\gamma_p\}$ is again a condition in $\PP^\diamond_{\dot{f},\dot{g}}$. 
\end{claim*} 
\begin{proof} 
We fix a condition $p\in\PP^\diamond_{\dot{f},\dot{g}}$. 
% and will first prove the following claim. CHANGE 
For any even ordinal $\gamma<\gamma_p$, let $\Psi_\gamma$ denote the statement that 
%For any even ordinal $\alpha<\gamma$, 
there exist an even ordinal $\beta<\gamma_p$ and some $(s,t)\in p$ with $\len(s)=\len(t)>\beta$ that satisfy the following conditions. 
\begin{enumerate-(a)} 
\item \label{properties of ranges 1} 
$\bigcup \ran{s{\upharpoonright}\beta} = \bigcup \ran{f_p{\upharpoonright}\gamma}$ and $\bigcup \ran{t{\upharpoonright}\beta} = \bigcup \ran{g_p{\upharpoonright}\gamma}$. 
\item \label{properties of ranges 2} 
$s(\beta)=f_p(\gamma)$ and $t(\beta)=g_p(\gamma)$.
\end{enumerate-(a)} 

\begin{subclaim*} 
If $\delta\leq\gamma_p$ is an even ordinal and $\Psi_\gamma$ holds for all even ordinals $\gamma<\delta$, then $q=p\cup\{(f_p{\upharpoonright}\gamma,g_p{\upharpoonright}\gamma)\mid\gamma<\delta\}$ is 
%\todo{is it in $\PP^*$?}
a condition in $\PP^\diamond_{\dot{f},\dot{g}}$. 
\end{subclaim*} 
\begin{proof} 
It is sufficient to check that $q$ satisfies Definition \ref{definition: S-tree} \ref{S-node 5}. 
To this end, suppose that $\gamma<\delta$ is even, $(u,v)\in p$, $\len(u)=\delta$ is even, 
$ \bigcup \ran{u{\upharpoonright}\alpha}=\bigcup \ran{f_p{\upharpoonright}\gamma}$ and $ \bigcup \ran{v{\upharpoonright}\alpha}=\bigcup \ran{g_p{\upharpoonright}\gamma} $. 
Now let $\beta<\gamma_p$ and $(s,t)\in p$ witness $\Psi_\gamma$. 
It follows from condition \ref{properties of ranges 1} and Definition \ref{definition: S-tree} \ref{S-node 5} for $p$ that $u(\alpha)=s(\beta)$ and $v(\alpha)=t(\beta)$. 
Moreover, by condition \ref{properties of ranges 2}, $u(\alpha)=s(\beta)=f_p(\gamma)$ and $v(\alpha)=t(\beta)=g_p(\gamma)$, as required. 
\end{proof} 

\begin{subclaim*} \label{f_p and g_p are contained in p} 
$\Psi_\gamma$ holds for all even ordinals $\gamma<\gamma_p$. 
%For any even ordinal $\alpha<\gamma_p$, there is an even ordinal $\beta<\gamma_p$ and some $(s,t)\in p$ with $\len(s)=\len(t)>\beta$ and the following properties.  
%\begin{enumerate-(a)} 
%\item \label{properties of ranges 1} 
%$\bigcup \ran{s{\upharpoonright}\beta} = \bigcup \ran{f_p{\upharpoonright}\alpha}$ and $\bigcup \ran{t{\upharpoonright}\beta} = \bigcup \ran{g_p{\upharpoonright}\alpha}$. 
%\item \label{properties of ranges 2} 
%$s(\beta)=f_p(\alpha)$ and $t(\beta)=g_p(\alpha)$.
%\end{enumerate-(a)} 
\end{subclaim*} 
\begin{proof} 
Towards a contradiction, we assume that $\gamma<\gamma_p$ is the least even ordinal such that $\Psi_\gamma$ fails. 
%are no such $(u,v)\in p$ and $\beta<\gamma_p$. 
Since $\Psi_\alpha$ holds for all even ordinals $\alpha<\gamma$ by the minimality of $\gamma$, the previous subclaim implies that 
$$q=p\cup\{(f_p{\upharpoonright}\alpha,g_p{\upharpoonright}\alpha)\mid\alpha<\gamma\}$$ 
is 
%\todo{is it in $\PP^*$?}
a condition in $\PP$. 

Since $S$ is perfect, there is some 
%\todo{notation for u,v? because of the types?} 
$(u,v)\in S$ with $u\supseteq\bigcup \ran{f_p{\upharpoonright}\alpha}$ and $v\supseteq \bigcup \ran{g_p{\upharpoonright}\alpha}$. 
We can further assume that 
%$(u,v)$ 
$(u,v)\neq (f_p(\gamma),g_p(\gamma))$ 
%and \todo{\ \ \ \ \ \ \ \ \ \ \ necessary?} $(u,v)\in\succsplit(S)$ 
by extending $u$, $v$. 

If \ref{properties of ranges 1} holds for an even ordinal $\beta<\gamma_p$ and some $(s,t)\in p$ with $\len(s)=\len(t)>\beta$, we also have \ref{properties of ranges 2} by Lemma \ref{values decided by P-diamond}. 
Hence we can assume that there are no such $\beta<\gamma_p$ and $(s,t)\in p$.  
%even ordinal $\beta<\gamma_p$ and and no $(s,t)\in p$ with $\len(s)=\len(t)>\beta$ such that \ref{properties of ranges 2} holds. 
%$\bigcup \ran{s{\upharpoonright}\beta} = \bigcup \ran{f_p{\upharpoonright}\gamma}$ and $\bigcup \ran{t{\upharpoonright}\beta} = \bigcup \ran{g_p{\upharpoonright}\gamma}$. 
It follows that $q\cup\{(u,v)\}$ 
%$$q=p\cup\{(f_p{\upharpoonright}\alpha,g_p{\upharpoonright}\alpha)\mid\alpha<\gamma\}\cup\{(u,v)\}$$ 
is 
%\todo{is it in $\PP^*$?} 
a condition in $\PP$ by Definition \ref{definition: S-tree} \ref{S-node 5} 
and further $q\Vdash_\PP (\dot{f}(\gamma),\dot{g}(\gamma))=(u,v)$ by Lemma \ref{values decided by P-diamond}. 
However, since $q\leq p$, this contradicts the fact that $(u,v)\neq (f_p(\gamma),g_p(\gamma))$. 
\end{proof} 

The previous subclaims show that 
$r=p\cup\{(f_p{\upharpoonright}\alpha,g_p{\upharpoonright}\alpha)\mid\alpha<\gamma_p\}$ 
is a condition in $\PP$. Since moreover $p\in \PP^\diamond_{\dot{f},\dot{g}}$ and $\len(r)=\he(r)=\gamma_p$, we have $r\in\PP^\diamond_{\dot{f},\dot{g}}$. 
% thus proving the claim. 
%$p\cup\{(f_p{\upharpoonright}\alpha,g_p{\upharpoonright}\alpha)\mid\alpha<\gamma_p\}$ is a condition in $\PP$ \todo{\ \ \ \ \ is it in $\PP^*$?}, as required. 
\end{proof} 

To see that $\PP^*_{\dot{f},\dot{g}}\cap \PP^*_{\dot{h},\dot{k}}$ is a dense subforcing of $\PP$, assume that $p$ is an arbitrary condition in $\PP$. By Lemma \ref{P-diamond is dense}, there is some $q\leq p$ in $\PP^\diamond_{\dot{f},\dot{g}}\cap \PP^\diamond_{\dot{h},\dot{k}}$. 
By the previous claim applied to $(\dot{f},\dot{g})$ and $q$, we obtain some $r\leq q$ in $\PP^*_{\dot{f},\dot{g}}\cap \PP^\diamond_{\dot{h},\dot{k}}$, 
and by then applying the claim to $(\dot{h},\dot{k})$ and $r$, we obtain the required condition $s\leq r$ in $\PP^*_{\dot{f},\dot{g}}\cap \PP^*_{\dot{h},\dot{k}}$. 
% as required. 
\end{proof} 

As for $\PP$, it is easy to see that  $\PP^*_{\dot{f},\dot{g}}$ is a non-atomic ${<}\kappa$-closed forcing of size $\kappa$ and hence 
$\PP^*_{\dot{f},\dot{g}}$ and $\Add{\kappa}{1}$ are sub-equivalent by Lemma \ref{forcing equivalent to Add(kappa,1)}.

As defined before Lemma \ref{two step iteration coded by a set S}, we will write $\QQ_p$ for the subforcing 
$$\QQ_p=\{q\in\QQ\mid q\leq p\}$$ 
of a forcing $\QQ$ below a condition $p\in\QQ$ in the following lemmas. 

%We consider the subset 
%$$S^*=\{(s_p,t_p)\in \PP^*\mid p\in\PP^*\}$$ 
%of $S$. 

\begin{lemma} \label{projection onto SxS} 
Letting $\PP^*=\PP^*_{\dot{f},\dot{g}}\cap \PP^*_{\dot{h},\dot{k}}$, 
for any condition $r$ in $\PP^*$ with $(s_r,t_r)\neq (u_r,v_r)$, the map 
$$\tau_{r}\colon \PP^*_{r}\rightarrow S_{(s_{r},t_{r})}\times S_{(u_{r},v_{r})},\ \tau_r(p)=((s_p,t_p),(u_p,v_p))$$ 
%$$\rho_{p^*}\colon [\PP_{\dot{f},\dot{g}}^*]_{p^*}\rightarrow S_{(s_{p^*},t_{p^*})},\ \rho(p)=(s_p,t_p)$$ 
%$$\rho_{p^*}\colon \PP_{\dot{f},\dot{g}}^*{\upharpoonright}{p^*}\rightarrow S{\upharpoonright}(s_{p^*},t_{p^*}),\ \rho(p)=(s_p,t_p)$$ 
is a projection. 
\end{lemma} 
\begin{proof} 
It follows from the definition of $s_p$, $t_p$, $u_p$, $v_p$ that $\rho_r$ is order-preserving. 
For the remaining requirement on projections, suppose that $p\in \PP^*$, $p\leq r$ and $((s,t),(u,v))\in S_{(s_{r},t_{r})}\times S_{(u_{r},v_{r})}$ are given with $s_p\subseteq s$, $t_p\subseteq t$, $u_p\subseteq u$, $v_p\subseteq v$. 
We can moreover assume that these subsets are strict by extending $s$, $t$, $u$, $v$. 

By the definition of $\PP^*_{\dot{f},\dot{g}}$ and $\PP^*_{\dot{h},\dot{k}}$, we have $(f_p{\upharpoonright}\alpha, g_p{\upharpoonright}\alpha), (h_p{\upharpoonright}\alpha, k_p{\upharpoonright}\alpha)\in p$ for all $\alpha<\gamma_p$. 
Since moreover $(s_r,t_r)\neq (u_r,v_r)$, 
$$q=p\cup \{(f_p,g_p),(f_p^\smallfrown (\gamma_p,s), g_p^\smallfrown (\gamma_p,t)), (h_p,k_p),(h_p^\smallfrown (\gamma_p,s), k_p^\smallfrown (\gamma_p,t)) \}$$ 
is downwards closed and satisfies Definition \ref{definition: S-tree} \ref{S-node 5}, hence it is a condition in $\PP$. 
Finally, 
$$q\Vdash_\PP \dot{f}(\gamma_p)=s,\  \dot{g}(\gamma_p)=t,\  \dot{h}(\gamma_p)=u,\  \dot{k}(\gamma_p)=v$$ 
by Lemma \ref{values decided by P-diamond}. 
Now any condition $r\leq q$ in $\PP^*$ is as required. 
\end{proof}

In the next two lemmas, we let $\PP^*=\PP^*_{\dot{f},\dot{g}}$. 

\begin{lemma} \label{projection onto S} 
Letting $\PP^*=\PP^*_{\dot{f},\dot{g}}$, 
for any condition $r$ in $\PP^*$, the map 
$$\rho_{r}\colon \PP^*_{r}\rightarrow S_{(s_{r},t_{r})},\ \rho_r(p)=(s_p,t_p)$$ 
%$$\rho_{p^*}\colon [\PP_{\dot{f},\dot{g}}^*]_{p^*}\rightarrow S_{(s_{p^*},t_{p^*})},\ \rho(p)=(s_p,t_p)$$ 
%$$\rho_{p^*}\colon \PP_{\dot{f},\dot{g}}^*{\upharpoonright}{p^*}\rightarrow S{\upharpoonright}(s_{p^*},t_{p^*}),\ \rho(p)=(s_p,t_p)$$ 
is a projection and $(s_{r},t_{r})$ forces that the quotient forcing $[\PP^*_{r}/S_{(s_{r},t_{r})}]^{\rho_{r}}$ 
and $\Add{\kappa}{1}$ are sub-equivalent. 
\end{lemma} 
\begin{proof} 
It can be proved as in the proof of Lemma \ref{projection onto SxS} that $\rho_r$ is a projection and moreover, it follows from Lemma \ref{projection onto SxS} that the quotient forcing $[\PP^*_{r}/S_{(s_{r},t_{r})}]^{\rho_{r}}$ is non-atomic. Since the quotient forcing had size $\kappa$ and is ${<}\kappa$-closed by the definitions of $s_p$, $t_p$ and $\PP^*_{\dot{f},\dot{g}}$, it is sub-equivalent to $\Add{\kappa}{1}$ by Lemma \ref{forcing equivalent to Add(kappa,1)}. 
\end{proof}

Our next aim is to calculate a quotient forcing for a given branch in the superclosed $S$-tree that is added by $\PP$. 
Since it is convenient to work with a separative forcing, but $\PP$ and $\PP^*$ are not separative, 
%let $\PP^*=\PP^*_{\dot{f},\dot{g}}$ and 
we will assume that $\TT$ is a dense subforcing of $\PP^*$ that is isomorphic to $\Addd{\kappa}{1}$ and that 
%$\dot{G}_\TT$ be a name for the $\TT$-generic filter and 
$\dot{T}_\TT$ is a name for the superclosed $S$-tree added by $\TT$. 
%We will now determine a quotient forcing for a given $\TT$-name $\dot{b}$ for a branch in this tree. 
%as defined after Lemma \ref{quotient forcing for a sequence of branches}. 
We will further assume that $\dot{b}$ is a $\TT$-name with $\one_{\PP}\Vdash \dot{b}=\ran{\bigcup{\dot{f}}}$ for the adequate pair $(\dot{f},\dot{g})$ considered above. 

If moreover $r$ is any condition in $\TT$, 
%$\TT$ is dense in $\PP$. 
%and if we force below $r$, we can assume that $\dot{b}$ is a $\TT_r$-name by replacing $\dot{b}$ with such a name. 
then  
$$\pi_S\rho_r\colon \PP^*_r\rightarrow \Add{\kappa}{1}_{s_r},\ \pi_S\rho_r(p)=s_p$$ 
is a projection, since $\rho_{r}\colon \PP^*_{r}\rightarrow S_{(s_{r},t_{r})}$ is a projection by Lemma \ref{projection onto S} and $\pi_S$ is a projection by the assumption on $S$. 

For any $r\in \TT$, we further choose a $\TT_r$-name $\dot{b}_r$ with $r\Vdash_\TT \dot{b}=\dot{b}_r$. 
It follows from the definition of $s_p$ that $\one_\TT$ forces that $\dot{b}_r=\bigcup_{p\in \dot{G}}s_p$, where $\dot{G}$ is a name for the $\TT$-generic filter. 
Using the fact that $\pi_S\rho_r$ is a projection, it then follows easily that $r$ forces that $\dot{b}_r$ is $\Add{\kappa}{1}$-generic over $V$. 
Moreover, since this holds for every condition $r$ in $\TT$, it follows that $\one_\TT$ forces that $\dot{b}$ is $\Add{\kappa}{1}$-generic over $V$. 

In the next lemma, we will fix a condition $r$ in $\TT$ and let $\RR=\BB(\TT_r)$, $\QQ=\BB^\RR(\dot{b})$. 
%Since we argued that $\one_\TT$ forces that $\dot{b}$ is $\Add{\kappa}{1}$-generic over $V$, it follows from the definition of $\QQ$ that  
It is clear that the map 
$$\iota\colon \Add{\kappa}{1}_{s_r}\rightarrow \QQ,\ \iota(s)=\llbracket s\subseteq \dot{b}\rrbracket$$ 
preserves $\leq$ and $\perp$, 
and since $\pi_S\rho_r$ is a projection, 
%\todo{ok?} 
we have that $\iota(s)\neq 0_\QQ$ for all $s\in\Add{\kappa}{1}_{s_r}$ and that $\ran{\iota}$ is dense in $\QQ$, so that $\iota$ is a sub-isomorphism. 

We will further consider the natural projection 
$\pi\colon \RR\rightarrow \QQ,\ \pi(p)=\inf_{p\leq q\in\QQ}q$. 
%induced by $\dot{b}$ below $r$, respectively. \todo{REWRITE?} 
Since $\TT$ is dense in $\PP^*$, $\pi{\upharpoonright}\TT_r$ and $ \pi_S \rho_r {\upharpoonright}\TT_r$ 
%$$ \pi_S \rho_r {\upharpoonright}\TT_r \colon \TT_r\rightarrow \Add{\kappa}{1}_{s_r}$$ 
%$$\pi{\upharpoonright}\TT_r\colon \TT_r\rightarrow \QQ$$ 
are projections and 
%\todo{\ \ \ \ OK? WOULD BE \ \\  \ \ \ \ NICE TO WRITE \ \\  \ \ \ \ MORE!}
it can be checked from the definitions of $\pi_s$, $\rho_r$ that 
$\pi{\upharpoonright}\TT_r= \iota\pi_S \rho_r {\upharpoonright}\TT_r $.

\begin{lemma} \label{quotients for elements of the projection} 
%Suppose that $\dot{f}$, $\dot{g}$ are as above and $\DD=\DD_{\dot{f},\dot{g}}$. 
%\todo{$\DD$ really depends on $\dot{b}$, $\dot{c}$, $\dot{f}$, $\dot{g}$} 
Suppose that $\TT$ and $\dot{b}$ are as above and $r\in \TT$. 
%-name with $\one_{\PP}\Vdash \dot{b}=\ran{\bigcup{\dot{f}}}$. 
\begin{enumerate-(1)} 
%for an element of $\proj[T_G]$. 
%\todo{suppose that $\DD$ is defined as above...} 
%\item 
%$\TT$ forces that $\dot{b}$ is $\Add{\kappa}{1}$-generic over $V$ 
\item 
If 
$\pi\colon\RR\rightarrow\QQ$ and 
$\iota\colon \Add{\kappa}{1}_{s_r}\rightarrow \QQ$ 
are as above, then 
%\todo{\ \ \ \ \ \ \ \ check defn of \ \\ \ \ \ \ \ \ \ \ $Q^{(\iota)}$ }
%the projection and the sub-isomorphism induced by $\dot{b}$ below $r$, respectively, then 
$$ 
\Vdash_{\Add{\kappa}{1}_{s_r}} (\RR_r/\QQ
%_{\iota(s_r)}]
^{\pi
%{\upharpoonright}\RR_r
})^{(\iota)} \simeq [S_{(s_r,t_r)}/\Add{\kappa}{1}_{s_r}]^{\pi_S}\times \Add{\kappa}{1}.$$ 
%has $$[S_{(s_r,t_r)}/\Add{\kappa}{1}_{s_r}]^{\pi_S} \times \Add{\kappa}{1}$$ as a quotient in $V[G]$. 

\item 
If 
$G$ is $\TT$-generic over $V$ with $r\in G$, then 
%\in \proj[\dot{T}_\TT]$ 
there is an $([S_{(s_r,t_r)}/\Add{\kappa}{1}_{s_r}]^{\pi_S})^{\dot{b}^G}\times\Add{\kappa}{1}$-generic filter $h$ over $W=V[\dot{b}^G]$ with $W[h]= V[G]$. 
\end{enumerate-(1)} 
%the quotient for $\dot{b}^G$ in $\PP_S$ is equivalent to $(S/\Add{\kappa}{1})^{\pi_{S, \Add{\kappa}{1}}} \times \Add{\kappa}{1}$. 
\end{lemma} 
\begin{proof} 
%\todo{NEED TO REWRITE} 
%We choose an \todo{check notation above}adequate pair $(\dot{f},\dot{g})$ with $\one_{\PP_S}\Vdash \bigcup \ran{\dot{f}}=\dot{b}$ and let $\PP^*=\PP^*_{\dot{f},\dot{g}}$. 
%Since we can assume that $\dot{b}$ is a $\TT_r$-name, 
%we can define $\RR=\BB(\TT_r)$, $\QQ=\BB^\RR(\dot{b})$ and consider the natural projection 
%\todo{definition above?}
%\begin{center} 
%$\pi\colon \RR\rightarrow \QQ,\ \pi(p)=\inf_{q\in\QQ, q\geq p}q.$ 
%\end{center} 
%$\TT$ is a separative dense subforcing of $\PP^*$, 
% Suppose that $\dot{G}$ is a $\PP_S$-name for $G$ 
%\ \\ 
%\todo[inline]{DO WE NEED THIS BELOW: Using the fact that $\pi(q)\geq q$ for all $q\in \RR$, it is easy to check that $\one_{\QQ}$ forces the quotient forcings for $\QQ$ in $\RR$, for $\QQ$ as a complete subforcing of $\RR$ and with respect to $\pi$, to be equal. Hence it is sufficient to prove that $\RR/\QQ^\pi$ has the required property. } 
%$\BB(\TT)/\BB(\dot{b})=$ 
%so that we can work with $\BB(\TT)/\BB(\dot{b})^\pi$. 

%\todo[inline]{WHERE EXACTLY SHOULD WE ASSUME THAT $\pi_S$ is a projection?} 

Since we argued before this lemma that $\pi{\upharpoonright}\TT_r= \iota\pi_S \rho_r {\upharpoonright}\TT_r $, we have 
\begin{equation} \label{equation 1 for comparing quotient forcings}
\Vdash_{\Add{\kappa}{1}_{s_r}} (\TT_r/\QQ
%_{\iota(s_r)}
^{\pi{\upharpoonright}\TT_r})^{(\iota)} = [\TT_r/\Add{\kappa}{1}_{s_r}]^{\pi_S \rho_r{\upharpoonright}\TT_r}. 
\end{equation} 
%\todo{\ \ \ DEFINE notation \ \ \ \ \ \  $\QQ^{(\iota)}$} 
Moreover, since $\TT_r$ is dense in both $\RR_r$ and $\PP^*_r$, 
$\Add{\kappa}{1}_{s_r}$ forces that 
$$(\TT_r/\QQ
%_{\iota(s_r)}]
^{\pi{\upharpoonright}\TT_r})^{(\iota)}\subseteq (\RR_r/\QQ
%_{\iota(s_r)}
^{\pi
%{\upharpoonright}\RR_r
})^{(\iota)}$$ 
$$[\TT_r/\Add{\kappa}{1}_{s_r}]^{\pi_S \rho_r{\upharpoonright}\TT_r} \subseteq [\PP^*_r/\Add{\kappa}{1}_{s_r}]^{\pi_S \rho_r}$$ 
are dense subforcings. 
With equation \ref{equation 1 for comparing quotient forcings}, this shows that 
\begin{equation} \label{equation 2 for comparing quotient forcings}
\Vdash_{\Add{\kappa}{1}_{s_r}} (\RR_r/\QQ
%_{\iota(s_r)}]
^{\pi
%{\upharpoonright}\RR_r
})^{(\iota)} \simeq [\PP^*_r/\Add{\kappa}{1}_{s_r}]^{\pi_S \rho_r}. 
\end{equation} 
%\todo{ok? this could use an extra argument} 
Using Lemma \ref{projection onto S} and the properties of projections, one can now show that 
\begin{equation} \label{equation 3 for comparing quotient forcings} 
\Vdash_{\Add{\kappa}{1}_{s_r}}  [\PP^*_r/\Add{\kappa}{1}_{s_r}]^{\pi_S \rho_r} \simeq [S_{(s_r,t_r)}/\Add{\kappa}{1}_{s_r}]^{\pi_S}\times \Add{\kappa}{1}. 
\end{equation} 
By equations \ref{equation 2 for comparing quotient forcings} and \ref{equation 3 for comparing quotient forcings} and Lemma \ref{equivalence of forcings is transitive}, 
\begin{equation} \label{equation 4 for comparing quotient forcings}
\Vdash_{\Add{\kappa}{1}_{s_r}} (\RR_r/\QQ
%_{\iota(s_r)}]
^{\pi
%{\upharpoonright}\RR_r
})^{(\iota)} \simeq [S_{(s_r,t_r)}/\Add{\kappa}{1}_{s_r}]^{\pi_S}\times \Add{\kappa}{1}. 
\end{equation} 

For the second claim, it follows from the definition of $\iota$ that 
%\todo{check that $G^{(\iota)}$ is defined, and check this equation!} 
$$[(\RR_r/\QQ^{\pi})^{(\iota)}]^G=(\RR_r/\QQ^{\pi})^{G^{(\iota)}}=(\RR_r/\QQ^{\pi})^{\dot{b}^G},$$ 
where $G^{(\iota)}$ denotes the upwards closure of $\iota[G]$ in $\QQ$. 
Then by equation \ref{equation 4 for comparing quotient forcings}, 
$$(\RR_r/\QQ^{\pi})^{\dot{b}^G} \simeq ([S_{(s_r,t_r)}/\Add{\kappa}{1}_{s_r}]^{\pi_S})^{\dot{b}^G} \times \Add{\kappa}{1}.$$ 
%\todo{ WHICH PROPERTIES? CITE!}
The claim now follows from the standard properties of quotient forcings. 
\end{proof} 

\iffalse 
$\QQ/\QQ_0$ given in Definition \ref{pull back names} and the quotient forcing $(\QQ/\QQ_0)^{\pi}$ with respect to $\pi$ given in Definition \ref{definition: quotient forcing} are equal. 

The map $\pi_{\PP^*, \Add{\kappa}{1}}$ 
%\colon \PP^*\rightarrow \Add{\kappa}{1}$, $\zeta(p)=s_p$ 
is a projection by Lemma \ref{projection from D} \ref{projection from D 2}. 
The quotient forcing for an $\Add{\kappa}{1}$-generic filter $H$ is $\PP^*/H=\{p\in \PP^*\mid s_p\in G\}$. 
Let $\pi_{\PP^*, S}$, $\pi_{S, \Add{\kappa}{1}}$ and $\pi_{\PP^*,\Add{\kappa}{1}}$ denote the maps in Lemma \ref{projections from D onto S and Add}. 
Then the quotient $(\PP^*/\Add{\kappa}{1})^{\pi_{\PP^*, \Add{\kappa}{1}}}$ is equivalent to the product $(S/\Add{\kappa}{1})^{\pi_{S,\Add{\kappa}{1}}}\times \Add{\kappa}{1}$. 
%\todo{NEED TO work below some condition? the quotient is not always the same........} Since $\rho$ is the composition $\rho=\tau\circ \pi$, the quotient forcing $(\DD/\Add{\kappa}{1})^{\rho}$ with respect to the projection $\rho\colon \DD\rightarrow \Add{\kappa}{1}$ is equivalent to the product of the quotient $S/\Add{\kappa}{1}$ with respect to $\tau\colon S\rightarrow \Add{\kappa}{1}$ and the forcing $\Add{\kappa}{1}$. 
\fi 

%Suppose that $G$ is $\mathbb{P}$-generic over $V$ and $A=\proj[T_G]^{V[G]}$. 
%Then player I has a winning strategy for $G_\kappa(A)$ in $V[G]$ by Lemma \ref{trees and strategies} and  Lemma \ref{the generic tree is full}. 
%It remains to prove the determinacy of the games $G_\kappa(A_{\varphi,z})$. 

%\todo[inline]{RECALL the definition of $A_\varphi$?} 

\begin{lemma} \label{winning strategy for Cohen subset} 
Suppose that $\SSS$ is a ${<}\kappa$-distributive forcing and $F=G\times H\times I$ is $\Add{\kappa}{1}\times \Add{\kappa}{1}\times \SSS$-generic over $V$. 
Moreover, suppose that 
$$ x\in (A_{\varphi,z}^\kappa)^{V[F]}\cap \kk\cap V[G]$$ 
is $\Add{\kappa}{1}$-generic over $V$, where $\varphi(u,v)$ is a formula and $z\in V[I]$. 
Then in $V[F]$, there is a winning strategy for player I in $G_\kappa((A_{\varphi,z}^\kappa)^{V[F]})$. 
% that is a projection of a tactic. 
\end{lemma} 
\begin{proof} 
%\todo{ok? maybe, this should be written out (just write down the forcing for the name and its quotient). but omit for now.}
We first note that $x$ is $\Add{\kappa}{1}$-generic over $V[I]$, since $G$, $I$ are mutually generic. 
%as is easy to see by an argument that involves changing the order of the generic filters. 
Therefore, by replacing $V[I]$ with $V$, the claim 
%statement of the lemma 
follows from the claim for the special case where $\SSS$ does not add any new sets, 
which we assume in the following. 
% argument. 
%that $\RR$ does not add any new sets. 

Suppose that 
%$\dot{G}$ is an $\Add{\kappa}{1}$-name for $G$ and 
$\dot{x}$ is an $\Add{\kappa}{1}$-name for $x$ such that $\one_{\Add{\kappa}{1}}$ forces that $\Vdash_{\Add{\kappa}{1}} \dot{x}\in A_{\varphi,z}$ holds and that $\dot{x}$ is $\Add{\kappa}{1}$-generic over $V$. 
Let further $\RR=\BB(\Add{\kappa}{1})$, $\QQ=\BB(\dot{x})^{\RR}$ and 
$$\nu\colon \Add{\kappa}{1}\rightarrow \QQ,\ \nu(s)=\llbracket s\subseteq \dot{x}\rrbracket^\RR.$$ 

\begin{claim*} 
There is a condition $r\in \Add{\kappa}{1}$ such that $\nu{\upharpoonright}\Add{\kappa}{1}_r\colon\Add{\kappa}{1}_r\rightarrow \QQ_{\nu(r)}$ is a sub-isomorphism. 
\end{claim*} 

\begin{proof} 
We first claim that there is a condition $r\in\Add{\kappa}{1}$ such that for all $s\leq r$ in $\Add{\kappa}{1}$ and all $\alpha<\kappa$, $ \nu(s) \neq  \nu(s^\smallfrown\langle \alpha\rangle )$. 
Otherwise 
$$D=\{s^\smallfrown\langle\alpha\rangle\mid s\in\Add{\kappa}{1},\ \alpha<\kappa,\ \exists \beta\neq \alpha\ \nu(s)= \nu(s^\smallfrown \langle \beta\rangle)\}$$ 
is dense in $\Add{\kappa}{1}$. However, by the definition of $\nu$, this contradicts the assumption that $\dot{x}$ is a name for an $\Add{\kappa}{1}$-generic over $V$. 

We now fix such a condition $r\in\Add{\kappa}{1}$. To prove the claim, it is sufficient to show that the subforcing $\UU=\{\nu(s)\mid s\in \Add{\kappa}{1},\ s\leq r\}$ is dense in $\QQ_{\nu(r)}$. 

\begin{subclaim*} 
$\UU\lessdot \RR_{\nu(r)}$. \end{subclaim*}
% and therefore of $\BB(\dot{x})$. 
\begin{proof} 
Otherwise, there is a subset $A$ of $\UU$ that is an antichain in $\RR_{\nu(r)}$ and is maximal in $\UU$, but not in $\RR_{\nu(r)}$. We can then choose some $q\in \RR_{\nu(r)}$ that is incompatible with all elements of $A$. However, if $J$ is $\RR_{\nu(r)}$-generic over $V$ with $q\in J$, then $\dot{x}^J$ cannot be $\Add{\kappa}{1}$-generic over $V$ by the choice of $A$ and $q$, contradicting the choice of $\dot{x}$. 
\end{proof} 
%and hence of $\QQ_{\nu(r)}$. 

Let  $\VVV$ denote the Boolean subalgebra of $\QQ_{\nu(r)}$ generated by $\UU$. Since $\UU$ is closed under finite conjunctions, $\UU$ is dense in $\VVV$ and it hence follows from the previous subclaim that $\VVV\lessdot \RR_{\nu(r)}$. 
%is a complete subforcing of $\QQ_{\nu(r)}$.  Moreover, since $\UU$ is closed under finite conjunctions, $\UU$ is dense in $\VVV$. 
It then follows from \cite[
%Lemma 30.25 \& 
Exercise 7.31]{MR1940513} 
applied to $\VVV$ and $\RR_{\nu(r)}$ that $\QQ_{\nu(r)}$ is a Boolean completion of $\VVV$, in particular $\VVV$ is dense in $\QQ_{\nu(r)}$. 
%first argue that is is a complete subforcing (that's easy and can be omitted) and then, that the BOOLEAN CLOSURE of a complete subforcing is actually a BOOLEAN COMPLETION. 
%Since $\BB(\dot{x})$ is generated by $\llbracket s\subseteq \dot{x}\rrbracket $ for all $s\in \Add{\kappa}{1}$, 
\end{proof} 

%\todo[inline]{REPLACE THE FOLLOWING SENTENCE\ \\ 
%1. argue that above some condition, $\iota$ is actually an isomorphism ("value 0 is not dense", "incompatible direct successors is not dense")\ \\ 
%2. argue that $\{\llbracket s\subseteq x\rrbracket \mid s\in \klk\}$ is dense in $\BB(\dot{x})$: first argue that is is a complete subforcing (that's easy and can be omitted) and then, that the BOOLEAN CLOSURE of a complete subforcing is actually a BOOLEAN COMPLETION. \ \\ 
%CITE Jech EXERCISE 7.31. 
%} 

Suppose that $r\in \Add{\kappa}{1}$ is chosen as in the previous claim and let $\iota=\nu{\upharpoonright}\Add{\kappa}{1}_r$. 
We can further assume that $r=\one_{\Add{\kappa}{1}}$, since the remaining proof is analogous for arbitrary $r$. 

%It then follows from the properties of $\dot{x}$ that $$\iota\colon \Add{\kappa}{1}\rightarrow \QQ,\ \iota(s)=\llbracket s\subseteq \dot{x}\rrbracket^\RR$$ is a continuous sub-isomorphism. \todo[inline]{ONLY BELOW A CONDITION, THAT WE SHOULD FIX} 

We further choose a $\QQ$-name $\dot{x}_\QQ$ 
%which is also an $\RR$-name since $\QQ\subseteq \RR$, 
with $\Vdash_{\RR}\dot{x}_\QQ=\dot{x}$ 
%for example $\dot{x}_{\QQ}=\{((\alpha,\beta),\llbracket \dot{x}(\alpha)=\beta\rrbracket)\mid \alpha,\beta<\kappa\}$. 
and an $\Add{\kappa}{1}$-name $\dot{y}$ for the $\Add{\kappa}{1}$-generic real, 
% and $\dot{x}_{\QQ}=\{((\alpha,\beta),\llbracket \dot{x}(\alpha)=\beta\rrbracket)\mid \alpha,\beta<\kappa\}$, 
%and let be $\dot{y}$ be an $\Add{\kappa}{1}$-name for the $\Add{\kappa}{1}$-generic real, so that 
so that $\Vdash_{\QQ}\ \iota(\dot{y})=\dot{x}_\QQ$ by the definition of $\iota$. 
% by the definition of $\iota$. 

%We can choose a $\QQ$-name $\dot{x}_\QQ$ 
%which is also an $\RR$-name since $\QQ\subseteq \RR$, 
%with $\Vdash_{\RR}\dot{x}_\QQ=\dot{x}$, for example $\dot{x}_{\QQ}=\{((\alpha,\beta),\llbracket \dot{x}(\alpha)=\beta\rrbracket)\mid \alpha,\beta<\kappa\}$. 

%\todo{CHANGE THE LEMMA ABOVE!!!!} 
By Lemma \ref{two step iteration coded by a set S}, there is 
%sub-isomorphism $\iota\colon \Add{\kappa}{1}\rightarrow \QQ$ and 
a perfect 
%\todo{need?}
limit-closed level subset $S$ of $\Add{\kappa}{1}^2$ 
such that $\pi_S$ is a projection and 
\begin{equation} \label{equation 1 in key lemma} 
\Vdash_{\Add{\kappa}{1}} \RR/\QQ^{(\iota)} \simeq S/\Add{\kappa}{1}^{\pi_S}. 
\end{equation} 

%\todo{ok?}
It follows from the properties of $\dot{x}$ stated above that 
$$\Vdash_{\QQ}\   \Vdash_{\RR/\QQ \times \Add{\kappa}{1}} \dot{x}_\QQ\in A_{\varphi,z}^\kappa.$$ 
Since $\iota$ is a sub-isomorphism and by the properties of $\dot{y}$ and $\dot{x}_\QQ$, this implies 
$$\Vdash_{\Add{\kappa}{1}}\   \Vdash_{\RR/\QQ^{(\iota)} \times \Add{\kappa}{1}} \dot{y}\in A_{\varphi,z}^\kappa$$ 
and by equation \ref{equation 1 in key lemma}, 
\begin{equation} \label{equation 2 in key lemma} 
\Vdash_{\Add{\kappa}{1}}\   \Vdash_{S/\Add{\kappa}{1}^{\pi_S} \times \Add{\kappa}{1}} \dot{y}\in A_{\varphi,z}^\kappa. 
\end{equation} 
%Suppose that $g\times h$ is $\PP_S\times\Add{\kappa}{1}$-generic over $V[I]$. 

Now suppose that 
%\todo{write $T^{(\nu)}$ below, where $\nu$ is the "natural embedding"?}
$\dot{T}$ is a $\PP_S$-name for the tree added by the $\PP_S$-generic filter. In the next claim, we will identify $\dot{T}$ with the induced $\PP_S\times \Add{\kappa}{1}$-name. 

\begin{claim*} 
$\PP_S\times \Add{\kappa}{1}$ forces that 
$\proj[\dot{T}]
%^{V[I\times J]}
\subseteq A_{\varphi,z}^\kappa$. 
%^{V[I\times J]}$. 
\end{claim*} 
\begin{proof} 
%We work in $V[J*K]$. 
%Let $\dot{J}$, $\dot{K}$ be the canonical $\PP_S* \dot{\UU}$-names for $J$, $K$. 
%Suppose that $e\in \proj[T_J]^{V[J*K]}$. 
%Then there is a $\PP_S* \dot{\UU}$-name $\dot{e}$ with $\dot{e}^{J*K}=e$ and $\one\Vdash_{\PP_S* \dot{\UU}} \dot{e}\in \proj[T_{\dot{J}}]$. 
%It follows from the homogeneity of $\dot{\UU}$ that there is such a $\PP_S$-name $\dot{e}$ with support in $\PP_S$, 
%i.e. there is a $\PP_S$-name $\dot{e}_1$ and $\dot{e}$ is obtained from $\dot{e}_1$ by adding $\one$ in the second coordinate. 
%Let $\dot{J}$ denote the canonical $\PP_S$-name for the $\PP_S$-generic filter over $V[I]$. 
Suppose that $\dot{b}$ is a $\PP_S$-name with $\one\Vdash_{\PP_S}\dot{b}\in \proj[\dot{T}]$. 
%Let 
%$$\BB_{\dot{b}}=\langle \{\llbracket \dot{b}(\alpha)=\beta\rrbracket \mid \alpha,\beta<\kappa\} \rangle^{\BB(\Add{\kappa}{1}*\Add{\kappa}{1})}.$$ 
%Then $\BB_{\dot{b}}$ is equivalent to $\Add{\kappa}{1}$ by the choice of $\dot{b}$. 
We can then find an 
%\todo{check notation above}
adequate pair $(\dot{f},\dot{g})$ with $\one_{\PP_S}\Vdash \bigcup \ran{\dot{f}}=\dot{b}$ 
and 
%\todo{cite definition above} 
let $\PP^*=\PP^*_{\dot{f},\dot{g}}$. 

Now let $\TT$ be the dense subforcing of $\PP^*$ that is introduced before Lemma \ref{quotients for elements of the projection}. 
Moreover, suppose that $G$ is $\TT$-generic over $V$ and $r\in G$. 
Since $\TT$ is dense in $\PP_S$, we can assume that $\dot{b}$ is a $\TT$-name. 
Then there is an $([S_{(s_r,t_r)}/\Add{\kappa}{1}_{s_r}]^{\pi_S})^{\dot{b}^G}\times \Add{\kappa}{1}$-generic filter $h$ over $W=V[\dot{b}^G]$ with $W[h]= V[G]$ by Lemma \ref{quotients for elements of the projection}. 

Since $\Add{\kappa}{1}^2\simeq \Add{\kappa}{1}$ and since $r$ forces that $[S_{(s_r,t_r)}/\Add{\kappa}{1}_{s_r}]^{\pi_S}$ is a complete subforcing of $[S/\Add{\kappa}{1}]^{\pi_S}$, the claim now follows from equation \ref{equation 2 in key lemma}. 
\end{proof} 

Lemma  \ref{trees and strategies} implies that $\PP_S\times \Add{\kappa}{1}$ forces that player I has a winning strategy in $G_{\kappa}(\proj[\dot{T}])$. 
Since $\PP_S\times \Add{\kappa}{1}$ is sub-equivalent to $\Add{\kappa}{1}^2$, the statement now follows from the previous claim. 
%Player I has a winning strategy for $G_{\kappa}(\proj[T_J]^{V[J*K]})$ in $V[J*K]$ by Lemma \ref{trees and strategies}. 
%This is a winning strategy for $G_\kappa(A_\varphi^{V[J* K]})$ by the previous claim. 
%This shows \ref{winning strategy for Cohen subset 1}. 
%Any projection $\sigma$ of $\tau$ is a winning strategy in $G_{\kappa}(\proj[T_J]^{V[J*K]})$ by Lemma \ref{trees and strategies} and Lemma \ref{the generic tree is full}. 
%There is a projection $\sigma$ of $\tau$ by Lemma \ref{projections of strategies} \ref{projections of strategies 1}. 
%Since $\PP_S$ is equivalent to $\Add{\kappa}{1}\times\Add{\kappa}{1}$, it follows from \ref{winning strategy for Cohen subset 1} that there is a winning strategy for player I in $G_\kappa(A_{\varphi,z}^{V[G\times H\times I])})$ in $V[G\times H\times I])$. 
%This shows \ref{winning strategy for Cohen subset 2}. 
%$<\kappa$-closed and has size $\kappa$, hence it is 
%Player I has a winning tactic for $G_{\kappa}(\proj[T_J]^{V[J*K]})$ in $V[J*K]$
%This is a winning strategy for $G_\kappa(A_\varphi^{V[J* K]})$ by the previous claim. 
%Suppose that $\nu$ is an isomorphism between dense subsets witnessing this equivalence. 
%Let $\dot{\SSS}$ denote the $\Add{\kappa}{1}*(S/\Add{\kappa}{1})^{\mu}*\Add{\kappa}{1}$-name induced by $\dot{\RR}$ via $\nu$. 
%and Lemma \ref{quotients for elements of the projection}. 
%Let $\PP=\PP_S$. 
%Let $\dot{c}$ denote the preimage of $\dot{b}$ under $\mu$. Let $c^+$ denote the trivial extension of the $\Addd{\kappa}{1}$-name $\dot{c}$ to 
%$\Add{\kappa}{1}*(S/\Add{\kappa}{1})^{\mu}*\dot{\RR}$ 
\end{proof} 

In the next proof, we will use the notation $\Col{\lambda}{X}$ for collapse forcings that was introduced before Theorem \ref{perfect subsets of definable sets}. 
We will further use the analogous notation $\Add{\lambda}{X}$ to denote the subforcing of $\Add{\lambda}{\nu}$ with support $X\subseteq \nu$ 
and let $G_X=G\cap \Add{\lambda}{X}$ for any $\Add{\lambda}{\nu}$-generic filter $G$. 

\begin{theorem} \label{determinacy result}
%\todo{\ \ \ \ \ \ \ \ write $\lambda$ \ \\ 
%\ \ \ \ \ \ \ \ instead of $\kappa$} 
%other notation for $\kappa$, since $\kappa^{{<}\kappa}$ is not used? why is not used?}  
Suppose that $\lambda$ is an uncountable regular cardinal, $\mu>\lambda$ is inaccessible and $\nu$ is any cardinal. 
%Let $\Col{\kappa}{<\lambda} \times \Add{\kappa}{\mu}$. 
Then $\Col{\lambda}{{<}\mu} \times \Add{\kappa}{\nu}$ forces that $G_\lambda(A)$ is determined for 
every subset $A$ of $\ltl$ that is definable from an element of ${}^\lambda V$. 
%Moreover player I has winning tactic in $G_\kappa(A)$ that is a projection of a tactic, or player II has a winning tactic in $G_\kappa(A)$. 
%$\lambda>\kappa$ is an inaccessible cardinal and $\mu\geq\lambda$ is a regular \todo{check Easton's theorem, maybe for all $\mu$ with $\cof{\mu}<\kappa$?} cardinal. 
%Then there is a $<\kappa$-closed forcing $\PP$ such that the following holds for every $\PP$-generic filter $G$ over $V$. 
%Suppose that $y\in \Ord^\kappa\cap V[G]$, $\varphi$ is a \todo{GIVE THE DETERMINACY a name? $\mathrm{DET}_{\Ord^{\kappa}}$?} first-order formula and 
%$$A_\varphi^{V[G]}=\{x\in {}^{\kappa}\kappa\mid \varphi(x,y)\}^{V[G]}.$$
%in $V[G]$. 
%Then $G_\kappa(A_\varphi^{V[G]})$ is determined in $V[G]$. 
% and in all further generic extensions by $<\kappa$-distributive forcings. 
\end{theorem} 
\begin{proof} 
We work in 
%$V[G\times H]$, where $G\times H$ is 
an extension of $V$ by a fixed $\Col{\lambda}{{<}\mu}\times \Add{\kappa}{\nu}$-generic filter $G\times H$. 
% over $V$. 
First note that every $x\in \ltl$ is an element of $V[G_\xi\times H_X]$ for some $\xi<\mu$ and some subset $X$ of $\nu$ of size strictly less than $\mu$, 
since $\Col{\lambda}{{<}\mu}\times \Add{\kappa}{\nu}$ has the $\mu$-cc by the $\Delta$-system lemma. 
In this situation, we will say that $x$ is \emph{absorbed by $G_\xi$, $H_X$}.

Now assume that $\varphi(x,y)$ is a formula with two free variables and $z\in {}^{\lambda} V$. 
%\todo{DO WE NEED THIS?} Using the set $A_{\varphi,y}$ given in Definition \ref{definition of Aphi}, let 
%$$A_{\varphi,y}^{V[G]}=\{x\in (\kk)^{V[G]}\mid V[G]\vDash \varphi(x,y)\}.$$ 
%Moreover, 
We let 
$$A^M=(A_{\varphi,z}^\lambda)^{V[G\times H]}\cap M$$ 
for any transitive subclass $M$ of $V[G\times H]$, 
where $A_{\varphi,z}^\lambda$ is given in Definition \ref{definition of Aphi}. 

%and from Lemma \ref{forcing equivalent to Add(kappa,1)} that $\Col{\kappa}{{<}\lambda}\times \Add{\kappa}{\nu}$ is equivalent to $\Col{\kappa}{{<}\lambda}$ for all $\nu\leq \lambda$. 
%it is easy to see that the parameter $z$ can be absorbed into an intermediate model \todo{notation defined?} of the form $V[G_\nu\times H_X]$ for some $\nu<\lambda$ and some subset $X$ of $\mu$ of size strictly less then $\lambda$, and that therefore, \todo{ok?} the claim follows from the statement for the special case that the parameter is in the ground model, which we will now assume. 

Since $\Add{\lambda}{1}$ is ${<}\lambda$-closed and $P(\Add{\lambda}{1})^V$ has size $\lambda$, 
%$(2^\kappa)^V$ has size $\kappa$ in $V[G\times H]$,  
the set of $\Add{\lambda}{1}$-generic elements of $\ltl$ over $V$ is comeager. 
Therefore, if there is no $\Add{\lambda}{1}$-generic element of $\ltl$ over $V$ in $A_{\varphi,z}^\lambda$, then by Lemma \ref{characterization of Baire property by game}, player II has a winning strategy in $G_\lambda(A_{\varphi,z}^\lambda)$. 
%^{V[G\times H]})$ in $V[G\times H]$. 
We can hence assume that there is an $\Add{\lambda}{1}$-generic element $x$ of $A_{\varphi,z}^\lambda$ over $V$. 

We will rearrange the generic extension to apply Lemma \ref{winning strategy for Cohen subset}. 
To this end, we assume that $x$ is absorbed by $G_\xi$, $H_X$ as above. 
%Since $\Col{\kappa}{{<}\lambda}\times \Add{\kappa}{\mu}$ has the $\lambda$-cc, some $\nu<\kappa$ and some $X\subseteq \mu$ of size strictly less than $\lambda$ and an $\Col{\kappa}{{<}\nu}\times \Add{\kappa}{X}$-name $\dot{x}$ with $\dot{x}^{G\times H}=x$. 
It follows from Lemma \ref{forcing equivalent to Add(kappa,1)} that we can find a $ \Col{\lambda}{[\xi,\mu)}\times \Add{\lambda}{1}$-generic filter $g\times h$ with $V[G_{[\xi,\mu)}]=V[g\times h]$ 
and hence  
%we can rewrite 
the generic extension 
%$$V[G]=V[G_\nu\times G_{[\nu,\lambda)}\times H_X \times  H_{\lambda\setminus X}]$$ 
can be written as 
$$V[G\times H]=  V[ g \times H_{\mu\setminus X} \times h \times G_\xi\times H_X]. $$ 
Since the filters $g \times H_{\nu\setminus X}\times h$ and $G_\xi\times H_X$ are mutually generic, it follows that $x$ is also $\Add{\lambda}{1}$-generic over $V[g\times  H_{\nu\setminus X}\times h]$. 

Now let $W=V[g \times H_{\nu\setminus X}]$. 
Since the forcing $\Col{\lambda}{{<}\xi}\times \Add{\lambda}{X}$ has size $\lambda$ in $W$, is ${<}\lambda$-closed and non-atomic, there is an $\Add{\lambda}{1}$-generic filter $k$ over $W$ with $W[k]=W[G_\xi\times H_X]$ by Lemma \ref{forcing equivalent to Add(kappa,1)}. 
Then $V[G\times H]=W[h\times k]$ is an $\Add{\lambda}{1}^2$-generic extension of $W$ and 
%$W[k]$ contains the $\Add{\kappa}{1}$-generic $x$ over $W$. 
$$ x\in (A_{\varphi,z}^\lambda)^{W[h\times k]}\cap \ltl\cap W[h\times k].$$ 
By Lemma \ref{winning strategy for Cohen subset}, player I has a winning strategy in $G_\lambda(A_{\varphi,z}^\lambda)$. 
\end{proof} 

By Lemma \ref{characterization of dense function}, the previous result implies that the almost Baire property for the class of definable sets considered there is consistent with arbitrary values of $2^\lambda$. 
%In particular, determinacy of $G_\kappa(A)$ for all subsets $A$ of ${}^{\kappa}\kappa$ definable from elements of $\Ord^\kappa$ is consistent with $2^\kappa=\kappa^+$ and with arbitrarily large values of $2^\kappa$. 
Moreover, as in the proof of Theorem \ref{perfect set property at all regulars}, we immediately obtain the following result. 

\begin{theorem} \label{almost Baire property at all regulars} 
Suppose there is a proper class of inaccessible cardinals. Then there is a class generic extension $V[G]$ of $V$ in which for every regular cardinal $\lambda$  
and for every subset $A$ of $\ltl$ that is definable from an element of ${}^\lambda V$, $G_\lambda(A)$ is determined. 
%the perfect set property holds for all subsets of ${}^{\kappa}\kappa$ that are definable from elements of $\Ord^\kappa$. 
%$\mathsf{PSP}^{\kappa}_{od}$ holds for all infinite regular cardinals $\kappa$ in a class forcing extension. 
\end{theorem} 
%\begin{proof} 
%As in the proof of Theorem \ref{perfect set property at all regulars} via Solovay's theorem \cite{MR0265151} and Theorem \ref{determinacy result}. 
%\end{proof} 

%\todo[inline]{We work in $V[G]$ as above. Suppose that $f\colon {}^{\kappa}\kappa\rightarrow {}^{\kappa}\kappa$ is continuous. Is there a subset $C$ of ${}^{\kappa}\kappa$ such that player I has a winning strategy in $G_{\kappa}(C)$ and $f\upharpoonright C$ is continuous? 
%s} 

\iffalse 
\begin{lemma} 
Suppose that $\kappa$ is an uncountable regular cardinal with $\kappa^{<\kappa}=\kappa$ and all subsets of ${}^{\kappa}\kappa$ that are definable from elements of $\Ord^\kappa$ have the almost Baire property. 
Then there is no well-order on ${}^{\kappa}\kappa$ that is definable from an element of $\Ord^\kappa$. 
\end{lemma} 
\begin{proof} 
The assumptions imply the Bernstein property for all subsets $A$ of ${}^{\kappa}\kappa$ definable from elements of ${}^{\kappa}\kappa$, i.e. $A$ or ${}^{\kappa}\kappa\setminus A$ has a perfect subset. 
%\todo{cite } 
Since $\kappa^{<\kappa}=\kappa$, it is straightforward to construct a definable Bernstein set by induction along a definable well-ordering of ${}^{\kappa}\kappa$, using an enumeration of all perfect subsets of ${}^{\kappa}\kappa$. 
\end{proof} 
\fi 

Since the almost Baire property immediately implies the Bernstein property, we obtain the following result as in the proof of Lemma \ref{perfect set property and Bernstein property}. 
%We further mention the following immediate applications of the almost Baire property for definable sets. 

\begin{lemma} 
%\todo{\ \ \ \ \ \ \ \ write $\lambda$ \ \\ 
%\ \ \ \ \ \ \ \ instead of $\kappa$} 
Suppose that $\lambda$ is an uncountable regular cardinal 
%with $\kappa^{<\kappa}=\kappa$ 
and all subsets of $\ltl$ that are definable from elements of $\Ordl$ have the almost Baire property. 
Then the following statements hold. 
\begin{enumerate-(1)} 
\item 
All subsets of $\ltl$ that are definable from elements of $\Ordl$ have the Bernstein property. 
\item 
There is no well-order on $\ltl$ that is definable from an element of $\Ordl$. 
\end{enumerate-(1)} 
\end{lemma} 
\iffalse 
\begin{proof} 
The first claim is immediate. 
%The assumptions imply the Bernstein property for all subsets $A$ of ${}^{\kappa}\kappa$ that are definable from elements of ${}^{\kappa}\kappa$. 
% i.e. $A$ or ${}^{\kappa}\kappa\setminus A$ has a perfect subset. 
To prove the second claim, suppose towards a contradiction that there is a well-order on $\kk$ that is definable from an element of $\Ordk$. 
%Since $\kappa^{<\kappa}=\kappa$, 
Using a standard construction, one can then construct a definable Bernstein set by induction. 
%along a definable well-ordering of ${}^{\kappa}\kappa$, using an enumeration of all perfect subsets of ${}^{\kappa}\kappa$. 
\end{proof} 
\fi

It is further possible to obtain results for homogeneous sets for definable colorings for which player I has a winning strategy in $G_\kappa$, which extend Theorem \ref{function on 2^kappa continuous on a perfect set} and will appear in a later paper. 

\section{Implications of resurrection axioms} 
%Variants of the main results from iterated resurrection axioms}
%Consequences of Iterated Resurrection Axioms} 
\label{resurrection axioms}

In this section, we obtain versions of the main theorems from a variant of 
%$\mathrm{RA}_{<\omega}^\lambda$ 
the resurrection axiom introduced by Hamkins and Johnstone \cite{MR3194674}. 
% that was introduced above. 
As above, we assume that $\lambda$ is an uncountable regular cardinal. 
Moreover, we will use the sets $A_{\varphi,z}$ and $A_{\varphi}$ given 
%\todo{check notation}
in Definition \ref{definition of Aphi}. 
%always assume that $\kappa$ is an uncountable regular cardinal with $\kappa^{<\kappa}=\kappa$. 
% having a perfect subset.  
%The following auxiliary results feature sufficient conditions for the existence of a perfect subset of a given definable subset of $\kk$. 
Our result is motivated by the following sufficient condition for the existence of a perfect subset of a given ${\bf \Sigma}^1_1$ subset of ${}^\lambda\lambda$. 
%gives such a condition for $\Sigma^1_1$ subsets of ${}^{\lambda}\lambda$ and 
%moreover shows that this is optimal. 

\begin{lemma} \label{perfect subsets of Sigma-1-1 sets} 
\begin{enumerate-(1)} 
\item 
%Let $\nu=2^\kappa$. 
%\todo{change Aphi to Aphi,z everywhere}
Suppose that $\varphi(x,y)$ is a $\Sigma^1_1$-formula and $z\in {}^{\lambda}\Ord$ is a parameter. 
If $|A_{\varphi,z}^\lambda|>\lambda$ holds in every $\Col{\lambda}{2^\lambda}$-generic extension of $V$, then $A_{\varphi,z}^\lambda$ has a perfect subset. 
\item 
Suppose that $V=L$. 
% and $\kappa$ is an uncountable regular cardinal. 
%with $\kappa^{<\kappa}=\kappa$. 
Then there is a $\Pi^1_1$ formula $\varphi(x)$ such that $|A_\varphi^\lambda|>\lambda$ holds in every generic extension of $V$, 
but $A_\varphi^\lambda$ does not have a perfect subset. 
\end{enumerate-(1)} 
\end{lemma} 
\begin{proof} 
For the first claim, 
it follows by standard arguments 
%as in classical descriptive set theory 
that there is a level subset $S$ of $({}^{<\lambda}\lambda)^2$ 
%${}^{<\lambda}\lambda\times {}^{<\lambda}\lambda$ 
with the property that $A_{\varphi,z}^\lambda$ is the projection of $S$ in every outer model with the same $V_\lambda$ as $V$. 
%generic extension by ${<}\kappa$-distributive forcing. 
%We identify the set $[T]$ of branches in $T$ with a subset of ${}^{\kappa}\kappa\times {}^{\kappa}\kappa$. Let $\nu=2^\kappa$. 
By the assumption, there are $\Col{\lambda}{2^\lambda}$-names $\sigma$, $\tau$ such that 
%are $\Col{\lambda}{2^\lambda}$-names such that $\one_{
$\Col{\lambda}{2^\lambda}$ forces that $(\sigma,\tau)$ is a new element of $[S]$. 
%Suppose $A=p[S]$, $S\subseteq ({}^{<\kappa}\kappa)^2$. Let $\dot{x}$, $\dot{y}$ be $Col(\kappa,2^{\kappa})$-names with 
%\begin{itemize} 
%\item $\Vdash(\dot{x},\dot{y})\in [S]$ and 
%\item $\Vdash\dot{x}\notin V$. 
%\end{itemize} 
Using these names, we can construct 
sequences 
$\langle p_u \mid u\in{}^{<\lambda}2\rangle$ of conditions in $\Col{\kappa}{2^\lambda}$ and $\langle ( s_u,t_u) \mid u\in{}^{<\lambda}2\rangle$ of nodes in $S$
such that the following conditions hold for all $u\subsetneq v$ in ${}^{<\lambda}2$. 
\begin{enumerate-(a)} 
%\item 
%$p_s\in Col(\kappa,2^{\kappa})$, 
%\item 
%$t_s, u_s\in {}^{<\kappa}\kappa$, 
\item 
$p_s\Vdash t_s\subsetneq\sigma\ \&\ u_s\subsetneq \tau$.  
\item 
$p_u\subseteq p_v$, $s_u\subsetneq s_v$ and $t_u\subsetneq t_v$. 
\item 
$t_{u^\smallfrown \langle 0\rangle}\neq t_{u^\smallfrown \langle1\rangle}$. 
\end{enumerate-(a)} 
Let $T$ denote the level subset of $({}^{<\lambda}\lambda)^2$ that is obtained as the downwards closure of the set of pairs $(s_u,t_u)$ for $u\in{}^{<\lambda}2$. 
%\times {}^{<\lambda}\lambda$
%The set $S$ of pairs $\langle t,u\rangle$ in ${}^{\kappa}\kappa\times {}^{\kappa}\kappa$ with $\len(t)=\len(u)$ such that $t\subseteq t_s$ and $u\subseteq u_s$ for some $s\in {}^{<\kappa}2$ forms a subtree of $T$. 
By the above conditions, its projection $\pro(T)=\{x\in {}^{\lambda}\lambda\mid \exists y\in {}^{\lambda}\lambda\ (x,y)\in [T]\}$ is a perfect subset of $A_{\varphi,z}^\lambda$. 
%Then $C=\{\bigcup_{\alpha<\kappa} t_{x\upharpoonright \alpha}\mid x\in {}^{\kappa}2\}$ is a perfect subset of $A$. 

For the second claim, we have a subtree $T$ of ${}^{<\lambda}\lambda$ with $|[T]|>\lambda$ and 
no perfect subtrees by \cite[Proposition 7.2]{Hurewicz}. 
%$\kappa$-Kurepa tree and 
We claim that the formula $\varphi(x)$ stating that $x\in [T]$ or $x\notin L$ satisfies the requirement. 
It follows from the choice of $T$ that $[T]$ does not have a perfect subset. 
%Let $A=[T]$ for a $\kappa$-Kurepa tree $T$, or another $\Pi^1_1({}^{\kappa}\kappa)$ set in $L$ with no perfect subset in $L$. Let $B={}^{\kappa}2\setminus L$ and $C=A\cup B$. 
To show the remaining condition, we work in a generic extension $V[G]$ of $V$. 
If $(\lambda^+)^L=\lambda^+$, then 
%we have 
$|A_{\varphi}^\lambda|\geq|[T]|\geq|[T]^L|> \lambda$. 
If $(\lambda^+)^L=\lambda$, 
then $|({}^\lambda\lambda)^L|\leq \lambda$ and hence $|A_\varphi^\lambda|>\lambda$ by the choice of $\varphi$. 
\end{proof} 

We now formulate the resurrection axiom at $\lambda$ for a given class of forcings. 
By a \emph{definable class of forcings} we will mean a class $\Gamma_{\varphi,z}=\{x\mid \varphi(x,z)\}$, where $\varphi(x,y)$ is a formula with two free variables with the property that it is provable in $\mathsf{ZFC}^-$ that $x$ is a forcing for all sets $x$, $y$ with 
%for all $x$, $y$, 
$\psi(x,y)$, 
% implies that $x$ is a forcing, 
and $z$ is a set parameter.

\begin{definition} \label{definition of resurrection axioms} 
Assuming that $\Gamma$ is a definable class of forcings, we define the \emph{resurrection axiom} $\RA^\lambda(\Gamma)$ to hold if for all $\PP\in \Gamma$, there is a $\PP$-name $\dot{\QQ}$ such that $\Vdash_\PP \dot{\QQ}\in \Gamma$ and $H_{\lambda^+}\prec^+ H_{(\lambda^+)^{V[G]}}$ holds for every $\PP*\dot{\QQ}$-generic filter $G$ over $V$. 
%$H_{\lambda^+}\prec^+ H_{(\lambda^+)^{V[G]}}$. 
\end{definition}

If $\lambda$ is a regular cardinal, we say that $\nu$ is \emph{$\lambda$-inaccessible} if $\nu>\lambda$ is regular and $\mu^{<\lambda}<\nu$ holds for all cardinals $\mu<\nu$. 
It can then be shown as in \cite[Theorem 18]{MR3194674} that the axiom $\mathrm{RA}^\lambda(\Gamma)$ for the class of forcings $\Col{\lambda}{{<}\nu}$, where $\nu$ is $\lambda$-inaccessible, is consistent from an uplifting cardinal $\mu>\lambda$ (see \cite[Definition 10]{MR3194674}). 

\begin{lemma} \label{perfect subset of definable set in collapse extension} \label{almost Baire property of definable set in collapse extension}
Suppose that $\nu$ is $\lambda$-inaccessible, $\varphi(x,y)$ is a formula and $z$ is a set parameter. Then $\Col{\lambda}{{<}\nu}$ forces the following statements. 
\begin{enumerate-(1)} 
\item 
If $|A_{\varphi,z}^\lambda|>\lambda$, then $A_{\varphi,z}^\lambda$ has a perfect subset. 
\item 
If there is an $\Add{\lambda}{1}$-generic element of ${}^{\lambda}\lambda$ in $A_{\varphi,z}^\lambda$, then player I has a winning strategy in $G_{\lambda}(A_{\varphi,z}^\lambda)$. 
\end{enumerate-(1)} 
\iffalse 
\begin{enumerate-(1)} 
\item 
%\in \Ord^\kappa$. 
% in the ground model. 
If $\Col{\lambda}{{<}\nu}$ forces that 
%If $\one\Vdash_{\Col{\kappa}{<\nu}} 
$|A_{\varphi,z}|>\lambda$, then 
%$\one_{\Col{\kappa}{<\nu}}$ 
it also forces that $A_{\varphi,z}$ has a perfect subset. 
%$A_\varphi$ has a perfect subset. 
\item 
%Let $\nu=(2^{\kappa})^+$. Suppose that $\varphi(x,y)$ is a formula and $z$ is a set. 
%\in \Ord^\kappa$. 
% in the ground model. 
If 
%$\one_{
$\Col{\lambda}{{<}\nu}$ forces that there is an $\Add{\lambda}{1}$-generic element of ${}^{\lambda}\lambda$ in $A_{\varphi,z}$, then 
%$\one_{\Col{\kappa}{<\nu}}$ 
it also forces that 
player I has a winning strategy in $G_{\lambda}(A_{\varphi,z})$. 
\end{enumerate-(1)} 
\fi 
\end{lemma} 
\begin{proof} 
%Let $\mathbb{P}=\Col{\kappa}{{<}\nu}$. 
Since $\nu$ is $\lambda$-inaccessible, 
%$(2^{\lambda})^{<\lambda}=2^\lambda<\nu$, 
it follows from a standard argument using the $\Delta$-system lemma that $\Col{\lambda}{{<}\nu}$ is $\nu$-cc. 
%Since $\PP$ is homogeneous, we can assume that $\one\Vdash_{\PP}|A_{\varphi,z}|>\kappa$. 
%Then $\mathbb{P}$ adds a new element to $A$. 

For the first claim, it follows from the assumption that there is a $\Col{\lambda}{{<}\nu}$-name $\sigma$ for a new element of $A_{\varphi,z}^\lambda$. 
By the $\nu$-cc, we can assume that $\sigma$ is a $\Col{\lambda}{{<}\mu}$-name for some ordinal $\mu<\nu$. 
%Since $\mathbb{P}$ has the $(2^{\kappa})^+$-c.c., $\sigma$ is supported in $Col(\kappa,<\gamma)$ for some $\gamma<(2^{\kappa})^+$.  Note that $Col(\kappa,<\gamma)$ is equivalent to $Col(\kappa,2^{\kappa})$ for $\gamma\geq 2^{\kappa}$. 
Since $\nu$ is $\lambda$-inaccessible, it is easy to see that there are unboundedly many cardinals $\mu\in \mathrm{Card}\cap\nu$ with $\mu^{<\lambda}=\mu$. 
To prove the claim, we work in a $\Col{\lambda}{{<}\nu}$-generic extension of $V$. 
We can now show as in the proof of Theorem \ref{perfect subsets of definable sets} (for $\Col{\lambda}{{<}\nu}$ instead of $\Col{\kappa}{{<}\lambda}$) that $A_{\varphi,z}^\lambda$ has a perfect subset. 

%More precisely, we work with $\Col{\lambda}{{<}\nu}$ (instead of $\Col{\kappa}{{<}\lambda}$, as in the proof of Theorem  \ref{perfect subsets of definable sets}) 
%and let $g\times h$ be $\Add{\lambda}{1}\times \Col{\lambda}{{<}\nu}$-generic with $V[G]=V[G_{\mu+1}\times g\times h]$. 
%We then prove the claim given in the proof of Theorem \ref{perfect subsets of definable sets} and argue as in the remaining proof. 

For the second claim, it follows from the assumption that there is a $\Col{\lambda}{{<}\nu}$-name $\sigma$ for an $\Add{\kappa}{1}$-generic element of ${}^{\kappa}\kappa$ in $A_{\varphi,z}^\lambda$. 
%To prove the claim, we work in a $\Col{\lambda}{{<}\nu}$-generic extension of $V$. 
We can again argue as in the proof of Theorem \ref{determinacy result} (for $\Col{\lambda}{{<}\nu}$ instead of $\Col{\kappa}{{<}\lambda}$). 
\end{proof} 

Our last result follows immediately from Lemma \ref{perfect subset of definable set in collapse extension} and the definition of the resurrection axiom.

\begin{theorem} \label{implications of resurrection} 
Suppose that 
%that $\kappa$ is a regular uncountable cardinal and 
$\Gamma$ is the class of forcings $\Col{\lambda}{{<}\nu}$, where $\nu$ is $\lambda$-inaccessible. 
%is reasonably closed at $\lambda$, 
%absolute to inaccessibles and contains the forcings $\Col{\lambda}{{<}\nu}$ for all SUCCESSOR cardinals $\nu$ with $\nu^{<\lambda}=\nu$. 
%preserving $\kappa$ that is closed under $<\kappa$-support iterations and lottery sums, is 
%Suppose that $\kappa$ is an uncountable regular cardinal and $\Gamma_\psi$ is a class of forcings as in the \todo{???????????} previous lemma. 
%reasonably closed definable class of forcings 
%that is absolute to inaccessibles 
%and contains the forcings $\Col{\kappa}{{<}\nu}$ for all cardinals $\nu$ with $\nu^{<\kappa}=\nu$. 
%Suppose that $\Gamma$ is the class of forcings preserving $\kappa$ and closed under $<\kappa$-support iterations and lottery sums that includes all forcings $\Col{\kappa}{<\nu}$, where $\nu^{<\kappa}=\nu$. 
%$<\kappa$-directed closed forcings, $<\kappa$-closed forcings or  $<\kappa$-strategically closed forcings. 
%Suppose that $n\in\omega$, 
Assuming that $\RA^{\lambda}(\Gamma)$ holds, 
% for all $n<\omega$, then 
the following statements hold for every subset $A$ of ${}^\lambda\lambda$ that is definable over $(H_{\lambda^+},\in)$ with parameters in $H_{\lambda^+}$. 
% has the following properties. 
\begin{enumerate-(1)} 
\item \label{perfect set property from resurrection} 
$A$ has the perfect set property. 
% holds for all subsets $A$ of ${}^{\kappa}\kappa$ 
% for every $\PP$-generic filter $G$ over $V$. 
\item \label{almost Baire property from resurrection} 
The game 
%\todo{conflict of notation?} 
$G_{\lambda}(A)$ is determined. 
% for all subsets $A$ of ${}^{\kappa}\kappa$ definable over $(H_{\kappa^+},\in)$ with parameters in $H_{\kappa^+}$. 
\end{enumerate-(1)} 
\end{theorem} 
\section{Questions} 

We conclude with some open questions. 
%We study this in work in preparation. 
We first note that by standard arguments, an inaccessible cardinal is necessary to obtain the perfect set property for $\lambda$-Borel subsets of $\ltl$. 
% but the analogous question is open for the almost Baire property. 
%In particular, w
The most striking question is whether the conclusion of Theorem \ref{determinacy result} can be achieved without an inaccessible cardinal as in \cite{MR768264}. 

\begin{question} 
Can the almost Baire property for all subsets of $\ltl$ definable from an element of $\Ordl$, for some uncountable regular cardinal $\lambda$, be forced over any model of $\mathsf{ZFC}$? 
% strictly less than an inaccessible cardinal? 
\end{question}

%The proofs of the main theorems use the Levy collapse $\Col{\kappa}{<\lambda}$. 
% is a product of $\Col{\kappa}{<\alpha}$ and $\Col{\kappa}{[\alpha,\lambda)}$ for all $\alpha$ with $\kappa<\alpha<\lambda$. It is not clear 
%\todo{REWRITE} It would be useful to obtain different models in which the conlusions of the main theorems hold. 
%We ask whether variants of these proofs work for other collapse forcings. 

Moreover, we 
%would like know if 
ask whether the conclusions of our results hold in the following other well-known models.

\begin{question} 
Do the conclusions of the main results, Theorem \ref{perfect subsets of definable sets} 
and Theorem \ref{determinacy result}, hold in the Silver collapse \cite[Definition 20.1]{MR2768691} and in the Kunen collapse \cite[Section 20]{MR2768691} of an inaccessible cardinal $\mu$ to $\lambda^+$, where $\lambda$ is any uncountable regular cardinal? 
\end{question} 

%A \emph{tactic} in $G_\kappa$ is a strategy $\sigma$ such that for $\vec{s}=\langle s_\alpha\mid\alpha<\gamma\rangle$, $\sigma(\vec{s})$ depends only on $\bigcup_{\alpha<\gamma} s_\alpha$. 
Since the existence of winning strategies implies the existence ot winning tactics for the Banach-Mazur game of length $\omega$, 
%as is easy to see, 
it is natural to consider the same problem in the present context. 

%the existence of strategies and tactics are equivalent by \ref{characterization of Baire property by game}. This suggests the following question. 

\begin{question} 
Is it consistent that for some uncountable regular cardinal $\lambda$ and 
% with $\kappa^{<\kappa}=\kappa$ such that for 
for all subsets $A$ of ${}^{\lambda}\lambda$ that are definable from elements of $\Ordl$, either player I or player II has a winning tactic in $G_\lambda(A)$? 
\end{question} 

Moreover, in analogy to the Baire property, it is natural to ask the following question, which arose in a discussion with Philipp L\"ucke. 

\begin{question} 
Does the almost Baire property for all subsets of $\ltl$ definable from elements of $\Ordl$ imply a version of the Kuratowski-Ulam theorem? 
\end{question} 

Finally, the similarities to other regularity properties suggest that our results can be extended as follows. 
%between the models obtained by forcing with $\Col{\omega}{{<}\lambda}$ and $\Col{\kappa}{{<}\lambda}$ for regular $\kappa$ and inaccessible $\lambda>\kappa$ suggest the question whether these results extend to other regularity properties. 

\begin{question} 
Can we prove results analogous to the main results 
%Theorem \ref{perfect subsets of definable sets} and Theorem \ref{determinacy result}, 
for games associated to other regularity properties such as the Hurewicz dichotomy? 
\end{question}

\bibliographystyle{alpha}
 \bibliography{references}

\end{document}